\documentclass[10pt,letterpaper]{amsart}

\usepackage{xcolor}
\usepackage[letterpaper, margin=1in]{geometry}
\usepackage{multicol}
\usepackage{listings}
\usepackage{amsmath, amsthm, amssymb, wasysym, verbatim, bbm, color}
\usepackage{graphicx}
\usepackage{caption}
\usepackage{subcaption}
\usepackage{url}
\usepackage{mathtools}
\usepackage{enumerate}
\usepackage{thmtools}
\usepackage{esint}
\makeatletter
\renewcommand\thmt@autorefsetup{%
  \@xa\def\csname\thmt@envname autorefname\@xa\endcsname\@xa{\thmt@thmname}%
}
\makeatother
\usepackage[final, colorlinks,allcolors=blue!75!black]{hyperref}

\usepackage{xfrac}
\usepackage{todonotes}

\usepackage{booktabs} 

\newcommand{\R}{\mathbb{R}}

\newcommand{\N}{\mathbb{N}}
\newcommand{\Q}{\mathbb{Q}}
\newcommand{\HH}{\mathcal{H}}
\newcommand{\tu}{\widetilde{u}}
\newcommand{\hu}{\widehat{u}}
\newcommand{\hQ}{\widehat{Q}}

\DeclareMathOperator{\tr}{\text{tr}}

\newcommand\norm[1]{\left\lVert#1\right\rVert}
\newcommand{\Grad}{\nabla}
\newcommand{\Div}{\operatorname{div}}
\newcommand{\dom}{\Omega}

\newcommand{\Ph}{\mathbb{P}_h}
\newcommand{\Mh}{\mathbb{M}_h}
\newcommand{\Mhz}{\mathbb{M}_{h,0}}
\newcommand{\Sh}{\mathbb{S}_h}
\newcommand{\Shz}{\mathbb{S}_{h,0}}
\newcommand{\Xh}{\mathbb{X}_h}
\newcommand{\Yh}{\mathbb{Y}_h}
\newcommand{\Ih}{\mathcal{I}_h}
\newcommand{\Kh}{\mathbb{K}_h}
\newcommand{\Uh}{\mathbb{U}_h}
\newcommand{\weak}{\rightharpoonup}

\newcommand{\weakstar}{\overset{\star}\rightharpoonup}
\newcommand{\PUh}{\mathcal{P}_{\Uh}}

\newcommand{\PMh}{\mathcal{P}_{\Mh}}

\newcommand{\Sym}{\mathcal{S}_0}

\newtheorem{theorem}{Theorem}[section]

\newtheorem{definition}[theorem]{Definition}

\newtheorem{remark}[theorem]{Remark}
\newtheorem{lemma}[theorem]{Lemma}

\allowdisplaybreaks

\numberwithin{equation}{section}

\title[Numerics for the Beris--Edwards system]{Convergence of a fully discrete finite element method for the Beris--Edwards system of liquid crystal dynamics}
\author[G.A. Benavides]{Gonzalo A. Benavides}
\address[Gonzalo A. Benavides]{\newline Department of Mathematics \newline University of Maryland \newline College Park, MD 20742, USA.}
\email[]{gonzalob@umd.edu}
\author[F. Weber]{Franziska Weber}
\address[Franziska Weber]{\newline Department of Mathematics \newline University of California, Berkeley \newline Berkeley, CA 94720, USA.}
\email[]{fweber@berkeley.edu}
\date{\today}
\thanks{G.B.~was partially supported by NSF DMS 2512392.
F.W.~was partially supported by NSF DMS 2438083.}

\begin{document}
\begin{abstract}
    We propose and analyze a fully discrete finite element scheme for the Beris–Edwards system of nematic liquid crystal dynamics, in which the incompressible Navier--Stokes equations are coupled to a gradient flow for the Landau--de Gennes Q-tensor. The scheme combines a linearly implicit formally second-order accurate backward differentiation formula in time with an incremental Chorin projection step for the incompressibility constraint, conforming finite elements in space, and the invariant energy quadratization approach with mass lumping for the nonlinear bulk potential. Each time step requires the solution of one linear system and one Poisson problem. We show that the scheme is uniquely solvable, that it preserves the symmetry and trace-free structure of the discrete Q-tensor and molecular field, and that it satisfies a discrete energy law without any restriction on the time step. Our main result is that, as the mesh size $h$ and the time step $\Delta t$ tend to zero subject to {$h^{2} = o(\Delta t)$},
    the approximations converge along a subsequence to a weak solution of the Beris–Edwards system. The convergence proof addresses two difficulties: the projection method produces two velocity approximations, only one of which is uniformly bounded in $L^2(0,T;H^1_0(\dom))$, and the coupling term $\mathcal{H}\Grad Q$ requires strong convergence of $\Grad Q$, which we obtain from the structure of the equation for $\mathcal{H}$ rather than from any discrete $H^2$-bound.
    Numerical experiments in two dimensions exhibit approximately second-order convergence in space and time, and reproduce the splitting of a $+1$ point defect into two $+1/2$ defects and the transport and deformation of a skyrmion induced by a constant pressure gradient.
\end{abstract}
\maketitle

\section{Introduction}\label{sec:intro}
Liquid crystals represent a unique state of matter that occupies the middle ground between conventional liquids and solid crystals. While they possess the fluidity of a liquid, they maintain a degree of long-range orientational order characteristic of crystals, typically due to the anisotropic shape of their constituent molecules~\cite{Stewart2004,Ball2017}. This dual nature allows liquid crystals to respond sensitively to external electric and magnetic fields, a property that has made them indispensable in modern display technologies and biological sensors~\cite{Castellano2005,Stewart2004}.

To describe the complex phase transitions and defect structures within these materials, the Landau--de Gennes theory provides a robust phenomenological framework. Unlike vector-based models (such as the Oseen--Frank theory), which are unable to describe all types of physically appearing defects, the Landau--de Gennes model utilizes a symmetric, traceless second order tensor known as the Q-tensor~\cite{Ball2017,Virga1995}.

The state of the system is characterized by the free energy functional, typically expressed as~\cite{Mottram2014}:
\begin{equation*}
E_{LG}(Q,\Grad Q)= \int_{\dom} \mathcal{F}_B(Q)+\mathcal{F}_E(\Grad Q) dx.
\end{equation*}
Here $\dom$ is a polytopal connected domain in $\R^d$, $d=2,3$. $\mathcal{F}_B$ is the bulk potential energy density and $\mathcal{F}_E$ is the elastic energy density. In the model by Landau and de Gennes, in the so-called one-constant approximation, they are given by
\begin{equation}
\label{eq:FBFE}
\mathcal{F}_B(Q) = \frac{a}{2}\tr(Q^2)-\frac{b}{3}\tr{(Q^3)}+\frac{c}{4}(\tr(Q^2))^2,\quad \mathcal{F}_E(\Grad Q) = \frac{L}{2}|\Grad Q|^2.
\end{equation}
Here $a,b,c,L$ are material parameters and $c,L>0$, $a,b\in\R$. $c>0$ guarantees a lower bound on the bulk potential energy, see \cite{Majumdar2010,Majumdar2010b}.
In a non-equilibrium scenario, the liquid crystal molecules are advected by the flow. The resulting model, termed Beris--Edwards system, is a system of partial differential equations coupling the incompressible Navier--Stokes equations to the Q-tensor gradient flow~\cite{Beris1994}:
\begin{subequations}
	\label{seq:BerisEdwards}
	\begin{align}
	\label{eq:momentum}
	\partial_t u +(u\cdot \Grad)u+\Grad p & = \mu\Delta u +\Div\Sigma-\HH\Grad Q+f,\\
	\label{eq:mass}
	\Div u & = 0,\\
	\label{eq:Qtensorflow}
	\partial_t Q +(u\cdot\Grad)Q-S & = M\HH,\\
	\label{eq:defH}
	\HH=-\frac{\delta E_{LG}}{\delta Q}& = L\Delta Q -\left[aQ-b\left(Q^2-\frac{1}{d}\tr(Q^2) I\right)+c\tr(Q^2)Q\right],
	\end{align}
\end{subequations}
where $u:[0,T]\times\dom\to \R^d$ is the fluid velocity, and $Q:[0,T]\times\dom\to \R^{d\times d}$ is the Q-tensor modeling the orientation of the liquid crystal molecules, and $f:[0,T]\times\dom\to \R^d$ is an external force, for which we assume that $f\in L^2(0,T;L^2(\dom))$.
Here $I$ denotes the $d\times d$ identity matrix, $\mu>0$ is the viscosity, $(\HH\Grad Q)_k = \sum_{i,j=1}^d \HH_{ij}\partial_k Q_{ij}$ and $((u\cdot\Grad)Q)_{ij} = \sum_{k=1}^d u_k\partial_k Q_{ij}$. 
We also denoted  the divergence $\Div u = \sum_{j=1}^d \partial_j u^j$ and the gradient $\Grad u = (\partial_j u^i)_{ij=1}^d$. 
$M>0$ is the mobility constant.
The tensor $S$ and the stress tensor $\Sigma$ are given by
\begin{equation}
\label{eq:S}
S = S(u,Q)= WQ-QW+\xi(QD+DQ)+\frac{2\xi}{d}D -2\xi(D:Q)\left(Q+\frac{1}{d}I\right)
\end{equation}
and 
\begin{equation}
\label{eq:Sigma}
\Sigma = \Sigma(Q,\HH) = Q\HH-\HH Q-\xi(\HH Q+Q\HH)-\frac{2\xi}{d}\HH + 2 \xi (Q:\HH)\left(Q+\frac{1}{d}I\right),
\end{equation}
where
\begin{equation}
\label{eq:DandW}
D = \frac{1}{2}(\Grad u+(\Grad u)^\top),\quad W = \frac12(\Grad u - (\Grad u)^T).
\end{equation}
We denoted the contraction of two $d\times d$-matrices $A$ and $B$ by $A:B = \tr(AB^T) = \sum_{i,j=1}^d A_{ij} B_{ij}$. The constant $\xi\in\R$, whose value
is contingent upon the specific molecular characteristics of a given liquid crystal, quantifies the proportion
between the tumbling effect and the aligning effect that a shear flow would exert on the liquid crystal director~\cite{Paicu2011}. $\Sigma$ is an elastic stress tensor term~\cite{Cavaterra2016}.
We note that the isotropic part in the last term in $\Sigma$, $2\xi (Q:\HH)\frac{1}{d}I$, results in a gradient term when applying the divergence operator in~\eqref{eq:momentum}. Hence we can modify the pressure $p$ to include this term as the conservation of mass~\eqref{eq:mass} constrains $u$ to be a divergence-free field. Therefore, we will be working with $\Sigma$ in~\eqref{eq:momentum} replaced by $\sigma$ as defined by
\begin{equation}\label{eq:sigmaused}
\sigma  = \sigma(Q,\HH) = Q\HH-\HH Q-\xi(\HH Q+Q\HH)-\frac{2\xi}{d}\HH + 2 \xi (Q:\HH)Q,
\end{equation}
and a correspondingly modified pressure $p$, in what follows. Several equivalent formulations of the Beris--Edwards system are available, see~\cite[Section 2.1]{Abels2014}.

We supplement these equations  with homogeneous Dirichlet boundary conditions for $u$ and $Q$ and for given initial data
\begin{equation*}
Q(0,\cdot)=Q_0\in H^1_0(\dom),\quad u(0,\cdot)=u_0\in L^2_{\text{div}}(\dom),
\end{equation*}
and boundary conditions
\begin{equation*}
	\left. Q(t,x)\right|_{\partial\dom} = 0,\quad \left. u(t,x)\right|_{\partial\dom} = 0,\quad t\in [0,T].
\end{equation*}
We require $Q_0$ to be trace-free and symmetric.
 
Our goal in this article is to develop a fully discrete numerical scheme for~\eqref{seq:BerisEdwards} and show its convergence to a weak solution of the system as the discretization parameters go to zero under no additional regularity assumptions besides those that follow from the energy inequality for system \eqref{seq:BerisEdwards}.
Global existence of weak solutions of this system and regularity results were shown in~\cite{Paicu2011,Paicu2012} and the long time behavior of solutions was studied in~\cite{ClimentEzquerra2022}. Existence and uniqueness of local strong solutions was shown in~\cite{Abels2016}. Time discretizations using the IEQ-method were proposed in~\cite{Zhao2017,Yue2023} and convergence of a first order time discretization was shown in~\cite{Yue2023}.
Fully discrete convergence results exist for the related Ericksen-Leslie and Ericksen models~\cite{Becker2008,Walkington2011,Liu2000,Liu2002,Barrett2006}, however, to the best of our knowledge, not for the Beris--Edwards Q-tensor system~\eqref{seq:BerisEdwards}.

In~\cite{Yue2023}, Y.~Yue and the second author designed and analyzed a linearly implicit semi-discrete in time first order accurate numerical scheme for the Beris--Edwards system and showed the energy stability and convergence of the numerical method. That scheme is based on introducing an auxiliary variable to deal with the nonlinear bulk energy potential; the method is known as the Invariant Energy Quadratization (IEQ) approach and was introduced for the Cahn--Hilliard equation in~\cite{GuillenGonzalez2013} and extended to various nonlinear gradient flows in~\cite{JiangIEQ, YANG2017691, IEQ_Cahn-Hilliard, YANG201880, YANG2017104}.
A second order semi-discrete scheme for~\eqref{seq:BerisEdwards} was proposed in~\cite{Zhao2017}. However, no spatial discretization was proposed and convergence or error estimates for the second order scheme have not been established.
In the IEQ approach, one uses the fact that the bulk energy density $\mathcal{F}_B$ is bounded from below. One then introduces the auxiliary variable
\begin{equation}
\label{eq:defr}
r(Q) =\sqrt{2\left(\frac{a}{2}\tr(Q^2)-\frac{b}{3}\tr(Q^3)+\frac{c}{4}\tr^2(Q^2)+A_0\right)},
\end{equation}
where $A_0>0$ is a sufficiently large constant ensuring that $r$ is positive for any choice of symmetric $Q\in\R^{d\times d}$. Defining
\begin{equation*}
V(Q)=aQ-b\left[Q^2-\frac{1}{d}\tr(Q^2)I\right] + c\tr(Q^2)Q,
\end{equation*}
it follows that, differentiating with respect to symmetric and traceless matrices,
\begin{equation}\label{eq:def:P}
\frac{\delta r(Q)}{\delta Q}  =\frac{V(Q)}{r(Q)}:= P(Q),
\end{equation}
Then the Beris--Edwards system can be recast as
\begin{subequations}
	\label{seq:BerisEdwardsIEQ}
	\begin{align}
			\label{eq:momentumIEQ}
		\partial_t u +(u\cdot \Grad)u+\Grad p & = \mu\Delta u +\Div\sigma-\HH\Grad Q+f,\\
		\label{eq:massIEQ}
		\Div u & = 0,\\
		\label{eq:QtensorflowIEQ}
		\partial_t Q +(u\cdot\Grad)Q-S & = M\HH,\\
		\label{eq:req}
		\partial_t r &= P(Q):\partial_t Q,\\
		\label{eq:defHIEQ}
		\HH=-\frac{\delta E_{LG}}{\delta Q}& = L\Delta Q -rP(Q),
	\end{align}
\end{subequations}
with $r(0)=r(Q_0)$.
In the current work, we propose a fully-discrete numerical scheme for system \eqref{seq:BerisEdwardsIEQ} based on a modification of the time discretization in~\cite{Yue2023,Zhao2017} and spatial discretization with finite elements and mass-lumping. The time discretization is based on a Backward Difference Formula (BDF2) and formally second-order accurate. We use the incremental version of Chorin's projection method to deal with the incompressibility constraint in the Navier--Stokes equations.
We show that this scheme is energy stable, preserves the trace-free and symmetry constraint for $Q$ and converges up to a subsequence to a weak solution of~\eqref{seq:BerisEdwards} as the discretization parameters go to zero, under the (inverse CFL-type) condition that the time step size and the spatial discretization parameter $h$ satisfy {$h^{2}=o({\Delta t})$}.
This condition is only needed to show convergence, not for solvability or stability of the scheme. Specifically, it is needed in estimates~\eqref{eq:pressuretimecont} and~\eqref{eq:pressureestimate2} since only a relatively weak bound on the pressure variable, namely $\norm{\Grad p_h}_{L^\infty(0,T;L^2(\dom))}\leq C \Delta t^{-1}$, is available.
Furthermore, it is needed to pass to the limit in the approximation of the term $P(Q):\partial_t Q$ on the right-hand side of~\eqref{eq:req} since $\partial_t Q$ has relatively low integrability.
Because uniqueness of weak solutions to this system is open, only convergence along a subsequence can be shown.

Challenges are caused by the various nonlinear and coupling terms in the system. Therefore the discretizations have to be carefully chosen in order to obtain a discrete version of the energy balance~\eqref{eq:energyineq} that yields stability.
Furthermore, the use of Chorin's projection method results in two approximations of the velocity field: an intermediate velocity that is not divergence free but satisfies a uniform $L^2(0,T;H^1_0(\dom))$-bound in the discretization parameters  thanks to the discrete energy balance; and a final step velocity field that is approximately weakly divergence free but satisfies no uniform $L^2(0,T;H^1_0(\dom))$-bound. To obtain precompactness of the velocity field approximations these have to be combined in a meaningful way, which we do using some of the tools developed in~\cite{Eymard2024,Weber2025,chorinprojection}. Next, we note that the a priori estimates coming from the energy inequality~\eqref{eq:energyineq} yield $L^2([0,T]\times\dom)$-bounds for the variables $\HH$ and $\Grad Q$.
However, to obtain convergence of a product term such as $\HH\Grad Q$ appearing in the equations, we require precompactness of at least one of the two in $L^2([0,T]\times\dom)$. Formally, considering~\eqref{eq:defH}, we see that $\HH$ being in $L^2$ implies $\Delta Q$ being in $L^2$ and therefore by elliptic regularity arguments, $Q$ should have at least (local) $H^2$ spatial regularity. However, in a finite element context, obtaining a conforming discrete analogue of $\Delta Q\in L^2$ and its consequence $Q\in H^2$ would require using $C^1$-finite elements which are prohibitive.
In a standard weak formulation $\Delta Q$ is replaced by the Dirichlet term $(\Grad Q,\Grad Z)$ for test functions $Z$, which does not immediately imply the $H^2$-regularity. Therefore, we do not attempt to prove an $H^2$-bound or a discrete version of it for the variable $Q$.
Rather, by exploiting the structure of the equation for $\HH$, we show that the $L^2([0,T]\times\dom)$-norms of the approximations of $\Grad Q$ converge, which then combined with the weak convergence of $\Grad Q$ (that follows in a straightforward manner from the discrete energy balance), yields precompactness of the approximations of $\Grad Q$ in $L^2([0,T]\times\dom)$. 

Finally, for the equation for the auxiliary variable $r$, we use mass-lumping which is helpful in obtaining the discrete energy inequality.
In addition, it makes the finite element formulation of the equation for $r$,~\eqref{eq:rdisc} a nodal identity (cf. Remark~\ref{rem:eqforr}) and so the approximation for $r$ can be eliminated pointwise in the equations.
Furthermore, it results in the last terms in the discretizations of equations~\eqref{eq:req} and~\eqref{eq:defHIEQ} canceling out exactly, just as in the continuous setting, when deriving the energy balance. 

Using these insights, we show convergence of the approximations to weak solutions of the reformulated system~\eqref{seq:BerisEdwardsIEQ}. Weak solutions of~\eqref{seq:BerisEdwardsIEQ} are equivalent to weak solutions of~\eqref{seq:BerisEdwards} using an argument from the earlier work~\cite{Gudibanda2022}.

A closely related alternative to the IEQ-method is the scalar auxiliary variable (SAV) approach~\cite{Shen2019}, which replaces the field-valued auxiliary variable of IEQ by a single scalar. This yields linear systems with constant coefficients, at the price of a nonlocal coupling; see~\cite{Shen2018,Shen2020} for a comparison of the two approaches.
In this work, we opt for IEQ because the field-valued variable $r$ enters the energy balance locally, which is what our convergence argument exploits.

We demonstrate the performance of the numerical scheme in several numerical experiments.
First we confirm the temporal and spatial accuracy of the scheme when the solution is smooth.
Then we use the scheme to simulate the splitting of a $+1$ defect into two $+1/2$ defects, and to visualize the transport and deformation of a skyrmion under the flow induced by a constant gradient pressure.
Both phenomena are captured qualitatively well by the scheme.

The rest of this article is organized as follows: In \autoref{sec:prelim} we introduce commonly used notation and the definition of weak solutions.
In \autoref{sec:numscheme}, we describe the numerical scheme, prove its solvability, various properties and the discrete energy estimate.
In \autoref{sec:convergence}, we present the convergence proof of the scheme and in \autoref{sec:computations} several numerical experiments that illustrate the spatial and temporal convergence of the scheme and show its capability to capture interesting physical phenomena.
We conclude this work with auxiliary results in the appendices \ref{app:gronwall}, \ref{app:masslumped} and \ref{app:rellich}.

\section{Preliminaries}\label{sec:prelim}
We start by introducing notation that will be used frequently in the following, and the definition of weak solutions for~\eqref{seq:BerisEdwards}.

\subsection{Notation}
We denote the norm of a Banach space $X$ as $\|\cdot\|_X$ and its dual space by $X^*$. We will denote $L^p$ spaces (for example, $L^2(\dom)$ for square integrable functions defined over the polytopal domain $\dom$), Sobolev spaces and Bochner spaces in standard ways, and will not distinguish between scalar, vector-valued and tensor-valued function spaces when it is clear from the context. In particular, we use $L^p(0,T; X)$ to denote the space of strongly measurable functions $f:[0,T]\to X$ which are $L^p$-integrable in the time variable $t\in [0,T]$.  The inner product on $L^2$ will be denoted by $(\cdot, \cdot)$ and the duality pairing between $v\in X$ and $u\in X^*$ will be denoted by $\langle u,v\rangle$.  For matrix-valued mappings $A,B:\dom\to\R^{d\times d}$, we denote the gradient $\Grad A =(\partial_i A^{jk})_{i,j,k=1}^d$, $\Grad A:\Grad B=\sum_{i,j,k=1}^d \partial_i A^{jk}\partial_i B^{jk}$ and $|\Grad A|^2 =\Grad A:\Grad A= \sum_{i,j,k=1}^d (\partial_i A^{jk})^2$.

We will use the subscript $\Div$  to indicate the divergence-free vector spaces, for example, 
\begin{equation*}
\label{eq:sigma_space}
\begin{gathered}
C_{c,\Div}^\infty(\dom)=\{{\phi}\in C_c^\infty(\dom); \Div {\phi}=0\}, \quad L^2_{\Div}(\dom)=\{ {\phi}\in L^2(\dom):\Div  {\phi}=0,  {\phi}\cdot {n}|_{\partial\dom}=0\}=\overline{C_{c,\Div}^\infty(\dom)}^{L^2(\dom)},\\
H^1_{\Div}(\dom)=H^1_0(\dom)\cap L^2_{\Div}(\dom),
\end{gathered}
\end{equation*}
where $n$ denotes the outward unit normal vector of the domain $\dom$,~\cite[Chapter I]{Temam1977}.
\begin{equation*}
	L^2_0(\dom)=\left\{\phi\in L^2(\dom); \, \int_{\dom} \phi dx =0\right\}
\end{equation*}
will be used to denote the space of $L^2$-functions that have zero average over the domain.
We denote the Leray projector by $\mathcal{P}: L^2(\dom)\to L^2_{\Div}(\dom)$, which is an orthogonal projection induced by the Helmholtz--Hodge decomposition $ {f}=\nabla g+ {h}$ for any $ {f}\in L^2(\dom)$~\cite{Temam1977}. Here, $g\in H^1(\dom)$ is a scalar field, and $ {h}\in L^2_{\Div}(\dom)$ is a divergence-free vector field. Then for all $ {f}\in L^2(\dom)$, it holds that $\mathcal{P}{f}={h}$.  

\subsection{Weak solutions}
Next, we define a notion of weak solutions. It is motivated by the notion of Leray--Hopf solutions for the incompressible Navier--Stokes equations. We let $T>0$ be an arbitrary time horizon.
\begin{definition}[Weak solutions of~\eqref{seq:BerisEdwards}]
	\label{def:weaksol} 
Assume that $f\in L^2(0,T;L^2(\dom))$.	Let $Q,\HH :[0,T]\times\dom\to\R^{d\times d}$ be matrix-valued mappings that are trace-free and symmetric for almost every $(t,x)\in [0,T]\times\dom$ and $u:[0,T]\times\dom\to \R^d$ be a vector-field that is divergence-free a.e., i.e., $\Div u=0$ for almost every $(t,x)\in [0,T]\times\dom$; and which satisfy
	\begin{equation}
	\label{eq:regularity}
\begin{split}
	&u\in L^\infty(0,T;L^2_{\Div}(\dom))\cap L^2(0,T;H^1_{\Div}(\dom)),\quad Q\in L^\infty(0,T;H^1_0(\dom)),\quad \HH \in L^2([0,T]\times\dom)\\
	&\partial_t u \in L^2(0,T;(X_6)^*),\quad \partial_t Q \in L^2(0,T;L^{6/5}(\dom)), 
\end{split}
	\end{equation}
where
\begin{equation*}
	X_6:= \{v\in W^{1,6}_0(\dom)\, : \, \Div v=0\},
\end{equation*}
	and the distributional version of~\eqref{seq:BerisEdwards}:
		\begin{subequations}
		\label{eq:weak_formulation_solution}
		\begin{equation}
		\label{eq:weakformu}
		\begin{split}
		&-\int_0^T  \left({u}, \partial_tv\right) dt  +\int_0^T\int_\dom ((u\cdot\Grad )u) \cdot v dxdt\\
		&=-\int_0^T\int_\dom \sigma(Q,\HH): \Grad v dxdt-\mu\int_0^T\int_\dom \Grad u:\Grad v dxdt -\int_0^T\int_\dom (\HH\Grad{Q})\cdot v dxdt+\int_0^T\int_{\dom} f\cdot v dx dt,
		\end{split}
		\end{equation}	
		\begin{equation}
		\label{eq:weakformQr}
		\begin{split}
		&\int_0^T\int_\dom \partial_t{Q} :Y  dx dt  + \int_0^T\int_\dom ((u\cdot \Grad ) Q):   Ydxdt - \int_0^T\int_\dom S(u,Q): Y dxdt \\
		& =M \int_0^T\int_\dom\HH: Y dxdt,
		\end{split}
		\end{equation}
		\begin{equation}
		\label{eq:weakformH}
		\begin{aligned}
		\int_0^T\int_\dom \HH: Zdxdt=&-L\int_0^T\int_\dom \Grad Q:\Grad Z    dxdt\\
		&-\int_0^T\int_\dom \left( a{Q}-b\left( {Q}^2-\frac{1}{d}\tr( {Q}^2)I \right)+c\tr({Q}^2){Q} \right):Z  dxdt,
		\end{aligned}
		\end{equation}
	\end{subequations}
for test functions $v\in C_c^\infty((0,T);C^\infty_{c,\Div}(\dom))$, and $Y,Z\in C^\infty_c((0,T)\times\dom)$. 
Moreover, assume that $(u,Q,\HH)$ satisfy the energy inequality:
\begin{multline}
\label{eq:energyineq}
\frac12\norm{u(t)}_{L^2(\dom)}^2+\frac{L}{2}\norm{\Grad Q(t)}_{L^2(\dom)}^2 + \int_{\dom} \mathcal{F}_B(Q(t)) dx  +\mu\int_0^t\norm{\Grad u(\tau)}_{L^2}^2 d\tau
+ M\int_0^t \norm{\HH(\tau)}^2_{L^2} d\tau \\
\leq \frac12\norm{u_0}_{L^2(\dom)}^2+\frac{L}{2}\norm{\Grad Q_0}_{L^2(\dom)}^2 + \int_{\dom} \mathcal{F}_B(Q_0) dx+ \int_0^t\int_{\dom} f(\tau)\cdot u(\tau) dx d\tau,
\end{multline}
for a.e. $t\in [0,T]$, and assume that $Q$ and $u$ are right-continuous in $H^1$ and $L^2$ respectively at $t=0$:
\begin{equation}
	\label{eq:contatzero}
	\lim_{t\to 0}\norm{Q(t)-Q_0}_{H^1(\dom)}=0,\quad \lim_{t\to 0} \norm{u(t)-u_0}_{L^2(\dom)}=0,
\end{equation}
for given $u_0\in L^2_{\Div}(\dom)$ and $Q_0\in H^1_0(\dom)$ that is trace-free and symmetric.
Then we call $(u,Q,\HH)$ a weak solution of~\eqref{seq:BerisEdwards}.
\end{definition}
\begin{remark}
	This definition of weak solutions is based on the notion of Leray--Hopf weak solutions for the incompressible Navier--Stokes equations found in, e.g.,~\cite{Berselli2021}.
It is weaker than the notions in~\cite{Yue2023,Paicu2011} in the sense that it does not require $Q\in L^2(0,T;H^2(\dom))$ and stronger than those notions in the sense that it requires the energy inequality~\eqref{eq:energyineq} and the continuity at zero~\eqref{eq:contatzero}. However, a local $H^2$ bound follows a posteriori from the integrability of $\HH$ and equation~\eqref{eq:weakformH} and elliptic regularity, and a global bound follows under stronger assumptions on the domain $\dom$, for example if it is convex.
\end{remark}

\section{Numerical scheme}\label{sec:numscheme}
Our scheme will be based on a linearly implicit discretization in time and mass-lumped finite elements in space. To deal with the incompressibility constraint, we will use a version of the projection method introduced by Chorin and Temam~\cite{Chorin1967,Chorin1968,Temam1969,Temam1977}. For the bulk energy we use the IEQ -method as described in the introduction \autoref{sec:intro}.
We start by describing the time discretization.
\subsection{Time discretization} 
We use a formally second-order accurate Backward Difference Formula (BDF2) for the time discretization combined with Chorin's projection method to deal with the incompressibility constraint for the velocity field variable~\cite{Chorin1967,Chorin1968}.
The idea of the projection method (or fractional step method) for the incompressible Navier--Stokes equations is to split the evolution of the Navier--Stokes equations~\eqref{eq:momentumIEQ} into two steps: In the first step the velocity field is evolved according to the convection, the dissipation and the forcing terms, which may violate the divergence constraint. An intermediate velocity field $\tu$ is computed. In the second step the intermediate velocity field is projected onto the space of divergence free fields. In order to preserve the trace-free property of $Q$ numerically, we will have to modify $S$ as follows:
\begin{equation}
\label{eq:littles}
s=s(u,Q)= WQ-QW+\xi(QD+DQ)+\frac{2\xi}{d}D-\frac{2\xi}{d^2}\Div u I -2\xi(D:Q)\left(Q+\frac{1}{d}I\right).
\end{equation} 
This modification is necessary because the intermediate velocity field $\tu$ is not divergence free, cf.~the upcoming Lemma~\ref{lem:tracefreediscrete}.
Note that $s(u,Q)=S(u,Q)$ when $u$ is divergence-free.
Now, we let $\Delta t>0$ be the time step size and discretize the time interval $[0,T]$ into time levels $t^m=m\Delta t$, $m=0,1,\dots, N$ (we pick $\Delta t$ such that $T/\Delta t=N\in \N$). At each time level $t^m$, $m= 1,2,3,\dots$, we seek discrete approximations $(u^m(x), Q^m(x), \HH^m(x), r^m(x), p^m(x))$ to $(u(t^m,x), Q(t^m,x), \HH(t^m,x), r(t^m,x), p(t^m,x))$. We denote the initial data $(u^0,Q^0,r^0, p^0)=(u_0,Q_0,r(Q_0),0)$. Then at every time level, given $(u^m,Q^m,r^m, p^m)$, for $m=0,1,\dots, N-1$, we update $(u^{m+1},Q^{m+1},\HH^{m+1},r^{m+1}, p^{m+1})$ according to the following two steps: \begin{enumerate}
	\item[{\bf Step 1}]
		\begin{subequations}\label{eq:step1semi}
		\begin{align}
		\frac{3\tu^{m+1}-4u^m+u^{m-1}}{2\Delta t}+ (\widehat{u}^{m+1}\cdot\Grad)\tu^{m+1}+\frac12\Div\hu^{m+1} \tu^{m+1}&=-\Grad p^m+ \mu \Delta \tu^{m+1}+\Div \sigma^{m+1}\\
		&\quad -\HH^{m+1}\Grad \hQ^{m+1}+ f^{m+1},\label{eq:fluidsemi}\\
		\frac{3Q^{m+1}-4Q^m+Q^{m-1}}{2\Delta t}+(\tu^{m+1}\cdot\Grad) \widehat{Q}^{m+1} &= s^{m+1}+ M \HH^{m+1},\\
	 3r^{m+1}-4r^m +r^{m-1} & = P(\widehat{Q}^{m+1}):(3Q^{m+1}-4Q^m+Q^{m-1}),\\
		\HH^{m+1} &= L\Delta Q^{m+1}-r^{m+1} P(\widehat{Q}^{m+1}),
		\end{align}
		\end{subequations}
	where
	\begin{equation}
	\label{eq:smSigmamsemi}
	s^{m+1} = s(\tu^{m+1},\widehat{Q}^{m+1}),\quad \sigma^{m+1}=\sigma(\widehat{Q}^{m+1},\HH^{m+1}),\quad f^m(x)=\frac{1}{\Delta t}\int_{t^{m-1/2}}^{t^{m+1/2}}f(\tau,x) d\tau
	\end{equation}
	and
	\begin{equation*}
		\widehat{Q}^{m+1} =2Q^m-Q^{m-1} ,\quad \widehat{u}^{m+1} =2 \tu^m-\tu^{m-1},\quad t^{m\pm 1/2}:=t^m\pm\frac{\Delta t}{2} ,
	\end{equation*}
	with boundary conditions 
	\begin{equation*}
	\left.\tu^{m+1}\right|_{\partial\dom}=0,\quad \left.Q^{m+1}\right|_{\partial\dom}=0.
	\end{equation*}
	\item[{\bf Step 2}] (Projection step): Next we define $(u^{m+1},p^{m+1})$ via
	\begin{subequations}
	\label{eq:projection}
	\begin{align}
	3\frac{u^{m+1}-\tu^{m+1}}{2\Delta t }&= -\Grad (p^{m+1}-p^m),\\
	\Div u^{m+1} & = 0,
		\end{align}	
	\end{subequations}
	with boundary condition
	\begin{equation}
	\label{eq:bcstep2}
	\left.u^{m+1}\cdot n\right|_{\partial \dom}=0.
	\end{equation}
\end{enumerate}

\subsection{Spatial discretization}\label{sec:spatial-discretization}
Next we describe the spatial discretization for the problem. 
We let $\mathcal{T}_h=\{K\}$ be a conforming shape-regular and quasi-uniform triangulation of $\dom$ made of simplices with mesh size $h>0$. Here $h=\max_{K\in \mathcal{T}_h}\text{diam}(K)$.
We let $\Uh\subset H^1_0(\dom)^d$, $\Ph\subset H^1(\dom)\cap L^2_0(\dom)$, $\Yh = \Uh + \Grad\Ph$, $\Mh\subset H^1(\dom)^{d\times d}$, $\Mhz\subset H^1_0(\dom)^{d\times d}$, $\Xh\subset L^2(\dom)$ be finite dimensional subspaces scaled by a meshsize $h>0$ that we will use for the spatial approximation of the velocity, the pressure, the Q-tensor and the auxiliary variable $r$ respectively. $\HH^{m}$ will be approximated in the space $\Mh$ and $Q^m$ in the space $\Mhz$ which is $\Mh$ with vanishing boundary values.

We use $\Uh$ as the space in which we seek the intermediate velocity $\tu^{m+1}_h$ and $\Ph$ as the space in which we seek the approximation of the pressure $p^{m+1}_h$. The final velocity $u^{m+1}_h$ is sought in the space $\Yh = \Uh + \Grad \Ph$. Note that since $\Ph\subset H^1(\dom)$, $\Grad \Ph$ is well-defined, though it may contain discontinuous functions.  Also note that the final velocity $u_h^{m+1}$ may not satisfy homogeneous Dirichlet boundary conditions.
We assume that $\Uh$ and $\Ph$ are based on piecewise polynomial spaces such as
\begin{align*}
	\Uh & = \{v\in H^1_0(\dom)\, | \, \left.v\right|_K\in P_k(K),\, \forall \, K\in \mathcal{T}_h \},\\
    \Ph & = \{q \in H^1(\dom) \cap L^2_0(\dom): \left.q\right|_K \in P_l(K), \, \forall \, K \in \mathcal{T}_h\},
\end{align*}
where $l \geq 1$, and $P_k(K)$ is the space of polynomials of degree $\leq k\in \N$ on $K$.
Other choices may be possible, but we have not checked this in detail.
For the rest of this article, we will use $k$ to denote the maximal degree of the polynomials in the space $\Uh$ and $\ell$ to denote the maximal degree of the elements in $\Ph$.

For $\Xh$, the space for the auxiliary variable $r$, we take the space of piecewise linear continuous functions on $\mathcal{T}_h$, i.e., $\Xh=\mathcal{S}^1(\mathcal{T}_h)=\{v_h\in C(\overline{\dom})\, :\, \left. v_h\right|_{K}\in P_1(K),\, \forall\, K\in \mathcal{T}_h\}$. We also assume that elements $A\in\Mh$, the space for $\mathcal{H}$, can be written as 
\begin{equation*}
A = \sum_{i,j=1}^d\sum_{k=1}^{N_h} a_{ij,k}E_{ij} \varphi_k,
\end{equation*}
where $a_{ij,k}\in\R$, $N_h\in\N$,  $E_{ij}$ is the matrix with a $1$ in position $(i,j)$ and zeros otherwise and $\varphi_k$ are the scalar-valued basis functions of $\Xh$. So $\Mh$ is the matrix-valued version of $\Xh$. $\Mhz$ is the span of the basis functions of the space $\Mh$ associated with interior nodes.
Next, we define the skew-symmetric trilinear form $b$ for functions $u,v,w\in H^1(\dom)$ with $u\cdot n=0$ on $\partial\dom$ as
\begin{equation}
\label{eq:defb}
b(u,v,w) = \int_{\dom} (u\cdot\Grad) v \cdot w dx +\frac12 \int_{\dom} \Div u (v\cdot w) dx= \frac12\int_{\dom} (u\cdot\Grad) v\cdot w -(u\cdot\Grad )w\cdot v dx .
\end{equation}
We observe that $b(u,w,v)= - b(u,v,w)$ and hence $b(u,v,v)=0$ for $u,v\in H^1(\dom)$ with $u\cdot n =0$ on $\partial\dom$ or $v=0$ on $\partial\dom$. Lastly, we need the mass-lumped inner product, which is defined for $f,g\in C(\overline{\dom})$ as
\begin{equation}
\label{eq:masslumped}
(f,g)_h:= \int_{\dom} \mathcal{I}_h (f\cdot g) dx = \sum_{z\in \mathcal{N}_h} f(z) g(z)\int_{\dom} \varphi_z dx,
\end{equation}
where $\mathcal{I}_h$ is the Lagrangian nodal interpolation operator and $\varphi_z$ the Lagrangian basis function associated with node $z\in\mathcal{N}_h$ that has value $1$ at $z$ and $0$ at all other nodes, see e.g.~\cite[Def. 3.13]{Bartels2015book}. 
 This discrete inner product has the following useful properties:
 \begin{lemma}{\cite[Lem. 3.9]{Bartels2015book},\cite{Bartels2015}}\label{lem:masslumped}
We have
\begin{equation}\label{eq:masslumpingnormequivalence}
 \norm{f}_{L^2}^2\leq \norm{f}_h^2 := (f,f)_h\leq (d+2)\norm{f}_{L^2}^2
 \end{equation}
for all $f\in \mathcal{S}^1(\mathcal{T}_h)$, the space of linear Lagrangian finite elements.
Moreover, for $f,g\in\mathcal{S}^1(\mathcal{T}_h)$, we have
\begin{equation}
\label{eq:masslumpingestimate1}
\left|(f,g)_h-(f,g)\right|\leq C h^{1+\ell}\norm{\Grad f}_{L^2}\norm{\Grad^\ell g}_{L^2},
\end{equation}
for $\ell\in\{0,1\}$ and with the convention that $\Grad^0 g=g$ and $\Grad^1 g =\Grad g$. In addition,
\begin{equation}
\label{eq:masslumpingestimate}
\left|(f,\phi)-(f,\phi)_h\right|\leq C h\norm{f}_{L^2} \norm{\phi}_{H^2}
\end{equation} 
for a constant $C>0$, $f\in \mathcal{S}^1(\mathcal{T}_h)$ and $\phi\in H^2(\dom)$. Furthermore, for any $\phi\in C(\overline{\dom})$, it holds
\begin{equation}\label{eq:masslumpinginterpolation}
(f,\phi)_h =(f,\mathcal{I}_h\phi)_h.
\end{equation}
\end{lemma}
Additionally, we have the following $L^p$-variants of the above properties:
\begin{restatable}{lemma}{lplumping}
	\label{lem:masslumpedLp}
	Assume that the mesh is quasi-uniform. Let $f,g\in  \mathcal{S}^1(\mathcal{T}_h)$, the space of linear Lagrangian finite elements. Let $p\in [1,\infty]$ and $p'$ the conjugate exponent, i.e., it satisfies $p^{-1}+ (p')^{-1}=1$. Define
	\begin{equation*}
		\norm{f}_{h,p}^p: = \sum_{z\in \mathcal{N}_h} |f(z)|^p \int_{\dom}\varphi_z dx = \int_{\dom} \Ih(|f|^p) dx, \quad \text{if $p \neq \infty$}
	\end{equation*}
	and
	\begin{equation*}
		\norm{f}_{h,\infty}: = \max_{z\in \mathcal{N}_h} |f(z)| = \|\Ih(f)\|_{L^\infty} = \|f\|_{L^\infty};
	\end{equation*}
	where $\varphi_z$ is the Lagrangian basis function associated with node $z\in \mathcal{N}_h$. Then we have
	\begin{equation}\label{eq:hhoelder}
				|(f,g)_h | \leq \norm{f}_{h,p}\norm{g}_{h,p'}\leq C \norm{f}_{L^p}\norm{g}_{L^{p'}}.
	\end{equation}
	Furthermore,
	\begin{equation}\label{eq:masslumpedstability}
		\norm{f}_{L^p}\leq 	\norm{f}_{h,p}\leq C \norm{f}_{L^p}
	\end{equation}
	and for $\ell\in \{0,1\}$,
			\begin{equation}
			\label{eq:masslumpingLp}
			\left|(f,g)_h-(f,g)\right|\leq \norm{\Ih(f\cdot g)-f\cdot g}_{L^1}\leq C h^{1+\ell}\norm{\Grad f}_{L^p}\norm{\Grad^\ell g}_{L^{p'}}.
		\end{equation}
\end{restatable}
The proof of~\eqref{eq:masslumpingLp} follows roughly~\cite[Lemma 3.9]{Bartels2015book} and can be found in Appendix~\ref{app:masslumped}.
\begin{restatable}{lemma}{lplumperror}
	\label{lem:masslump2}
	Let $v_h, w_h\in \mathcal{S}^1(\mathcal{T}_h)$ and let $F$ be Lipschitz continuous. Let $p\in [1,\infty]$ and $p'$ such that $p^{-1}+(p')^{-1}=1$. Then 
	\begin{equation}
		\label{eq:nonlinearmasslumping}
		\norm{v_h F(w_h)-\Ih(v_hF(w_h))}_{L^1(\dom)}\leq C h \norm{v_h}_{L^p}\norm{\Grad w_h}_{L^{p'}}.
	\end{equation}
Naturally, the same estimate extends to vector- and tensor-valued functions with $v_hF(w_h)$ replaced by the appropriate contraction.
\end{restatable}
The proof can also be found in Appendix~\ref{app:masslumped}.

We start by approximating the initial data for $u$, $Q$ and $r$:
\begin{equation}
\label{eq:initdataapprox}
\tu^0_h = \PUh u_0,\quad Q^0_h =\Pi^{SZ}_h Q_0,\quad r_h^0=\Ih(r(Q^0_h))
\end{equation} 
where $\mathcal{I}_h$ is the Lagrangian interpolation operator,  $\PUh:L^2(\dom)^d\to \Uh$ the $L^2$-orthogonal projection onto the finite element space $\Uh$, e.g.,~\cite[Chapter 22]{Ern2021},~\cite{Douglas1975}, and  $\Pi^{SZ}_h: H^1_0(\dom)^{d\times d}\to \Mhz$ the Scott--Zhang interpolation operator~\cite{Scott1990} used component-wise.

Since $\tu_h^0$ is not necessarily divergence free, we project it onto the space of discretely divergence free functions using the following projection step: Find $(u_h^0,p_h^0)\in \Yh\times\Ph$ such that
\begin{align}\label{eq:zerothstep}\begin{split}					
		\left(\frac{u^{0}_h-\tu^{0}_h}{\Delta t }, v\right)&= -(\Grad p^{0}_h, v),\\
		\left( u^{0}_h,\Grad q\right)&=0,
	\end{split}	
\end{align}
for all $(v,q)\in \Yh\times \Ph$.

Then we define the following iterative scheme: 
\begin{enumerate}
	\item[{\bf Step 1}] 
	For any $m\geq 1$, given $(u_h^{m},p_h^m,Q^{m}_h,r_h^{m})\in \Yh\times \Ph\times\Mhz\times\Xh$,   $(u_h^{m-1},Q^{m-1}_h,r_h^{m-1})\in \Yh\times\Mhz\times\Xh$ and $\tu^m_h,\tu_h^{m-1}\in \Uh$, find $(\tu_h^{m+1},Q^{m+1}_h,\HH_{h}^{m+1},r_h^{m+1})\in \Uh\times \Mhz\times \Mh\times\Xh$, such that for all $(v,Y,Z,w)\in \Uh\times \Mh\times\Mhz\times\Xh$,
	\begin{subequations}\label{eq:step1fully}
	\begin{align}
	\label{eq:udisc}
	\left(\frac{3\tu^{m+1}_h-4u^m_h+u^{m-1}_h}{2\Delta t},v\right)&=-b(\hu^{m+1}_h,\tu^{m+1}_h,v)-(\Grad p^m_h,v)- \mu(\Grad \tu_h^{m+1},\Grad v)  -(\sigma_h^{m+1},\Grad v)\\
	&\quad -(\HH_h^{m+1}\Grad \hQ^{m+1}_h,v)+(\PUh f^{m+1},v),\notag\\
	\label{eq:Qdisc}
	\left(\frac{3Q^{m+1}_h-4Q^m_h+Q^{m-1}_h}{2\Delta t},Y\right) &= -((\tu^{m+1}_h\cdot\Grad) \hQ^{m+1}_h,Y) + (s_h^{m+1},Y)+M (\HH_h^{m+1},Y),\\
	\label{eq:rdisc}
	(3r^{m+1}_h-4r^m_h+r^{m-1}_h,w)_h & = (P(\hQ^{m+1}_h):(3Q_h^{m+1}-4Q_h^m+Q^{m-1}_h),w)_h,\\
	\label{eq:Hdisc}
	(\HH_h^{m+1},Z) &= -L(\Grad Q_h^{m+1},\Grad Z)-(r_h^{m+1} P(\hQ_h^{m+1}),Z)_h,
	\end{align}
\end{subequations}
where
\begin{equation}
\label{eq:smSigmam}
s^{m+1}_h = s(\tu_h^{m+1},\hQ_h^{m+1}),\quad \sigma_h^{m+1}=\sigma(\hQ_h^{m+1},\HH_h^{m+1}),
\end{equation}
\item[{\bf Step 2}] (Projection step): Next we seek $(u^{m+1}_h,p^{m+1}_h)\in \Yh\times\Ph$ which satisfies for all $(v,q)\in \Yh\times\Ph$, 
\begin{subequations}
	\label{eq:projectionfullydiscrete}
	\begin{align}
	\label{eq:projection1}
\left(3\frac{u^{m+1}_h-\tu^{m+1}_h}{2\Delta t }, v\right)&= -( \Grad (p^{m+1}_h-p^m_h), v),\\
\label{eq:projection2}
\left(u^{m+1}_h,\Grad q\right)&=0.
	\end{align}	
\end{subequations}

\end{enumerate}
Here again $\hu_h^{m+1} = 2\tu^{m}_h - \tu_h^{m-1}$ and $\hQ_h^{m+1} = 2 Q_h^m-Q_h^{m-1}$.

\begin{remark}[Formulation of projection step as a Poisson problem]
	The projection step can be rewritten as a Poisson problem as follows: Taking $\Grad q\in \Grad\Ph$ as a test function in~\eqref{eq:projection1} and using the weak divergence constraint~\eqref{eq:projection2}, we see that $p_h^{m+1}\in \Ph$ can be computed from
	\begin{equation}\label{eq:poisson}
		(\Grad p^{m+1}_h,\Grad q) = (\Grad p^m_h,\Grad q)-\frac32\left(\frac{\Div \tu^{m+1}_h}{\Delta t},q\right),
	\end{equation}
	and then $u_h^{m+1}$ can be computed via~\eqref{eq:projection1}. Therefore, no inf-sup/LBB compatibility condition is required for the spaces $\Ph$ and $\Uh$.
	\end{remark}

\begin{remark}
\label{rem:eqforr}
We note that the equation for $r^{m+1}_h$,~\eqref{eq:rdisc}, can be rewritten as a pointwise equation for each $r_z^{m+1}$ where $r^{m+1}_h:= \sum_{z\in\mathcal{N}_h} r_z^{m+1}\varphi_z$ and $\varphi_z$ are the basis functions of $\Xh$:
\begin{multline*}
3\sum_{z\in\mathcal{N}_h} r^{m+1}_z w(z)\int_{\dom} \varphi_z dx  = 4\sum_{z\in\mathcal{N}_h} r^{m}_z w(z)\int_{\dom} \varphi_z dx-\sum_{z\in\mathcal{N}_h} r^{m-1}_z w(z)\int_{\dom} \varphi_z dx \\
+\sum_{z\in \mathcal{N}_h} P(\hQ^{m+1}_z):(3Q^{m+1}_z-4Q^m_z+Q^{m-1}_z)w(z)\int_{\dom} \varphi_z dx,
\end{multline*} 	
i.e., by choosing $w=\varphi_y$, we obtain
\begin{equation}\label{eq:rsimplified}
r^{m+1}_y  = \frac{1}{3}\left( 4 r^{m}_y - r^{m-1}_y +  P(\hQ^{m+1}_y):(3Q^{m+1}_y-4Q^m_y+Q^{m-1}_y)\right).
\end{equation}
	
\end{remark}
\begin{remark}\label{rem:1ststep}
	Since for the first step at $m=0$, the approximations at $m-1=-1$ are not defined, one can either use extrapolation to define the solution at $t=-\Delta t$, or use a first order method to obtain $\tu^1_h,u^1_h,p^1_h,Q^1_h,r^1_h,\HH^1_h$ as follows, as it was done for example in~\cite[Section 3]{Guermond2006}:
	First compute $(\tu^1_h,Q_h^1,\HH_h^1,r_h^1)\in \Uh\times \Mhz\times\Mh\times\Xh$ by
	\begin{subequations}
		\label{eq:step1fully1step}
		\begin{align}
			&	\left(\frac{\tu^{1}_h-u^0_h}{\Delta t},v\right)+ b(\tu_h^{0},\tu^{1}_h,v)+(\Grad p^0_h, v)+\mu(\Grad \tu_h^{1},\Grad v)=-(\sigma_h^1,\Grad v)\\
			& \qquad \qquad \qquad -(\HH_h^1\Grad Q_h^0,v)+ (\PUh f^{1}, v),\quad \forall v\in \Uh,\notag \\
			& \left(\frac{Q^1_h-Q^0_h}{\Delta t},Y\right) + ((\tu^1_h\cdot\Grad)Q^0_h,Y) - (s^1_h,Y) = M(\HH^1_h,Y),\quad \forall Y\in \Mh,\label{eq:Qfirststep}\\
		&	(r^1_h-r^0_h,w)_h =  (P(Q_h^0):(Q^1_h-Q^0_h),w)_h,\quad \forall w\in \Xh,\label{eq:rfirststep}\\
		&(\HH^1_h,Z) = - L ( \Grad Q^1_h,\Grad Z)-(r^1_hP(Q_h^0),Z)_h,\quad \forall Z\in \Mhz,\label{eq:Hfirststep}
		\end{align}
	\end{subequations}
	where
	\begin{equation*}
		s^1_h = s(\tu^1_h,Q^0_h),\quad \sigma_h^1 = \sigma(Q_h^0,\HH_h^1).
	\end{equation*}
	Then obtain $u^1_h\in \Yh$ and $p_h^1\in \Ph$ via the following projection step
	\begin{subequations}
		\label{eq:projectionfullydiscretestep1}
		\begin{align}
			\label{eq:projection1step1}
			\left(\frac{u^{1}_h-\tu^{1}_h}{\Delta t }, v\right)&= -( \Grad (p^{1}_h-p^0_h), v),\quad \forall v\in \Yh\\
			\label{eq:projection2step1}
			\left(u^{1}_h,\Grad q\right)&=0,	\quad \forall q\in \Ph.
		\end{align}	
	\end{subequations}
	In the following analysis, we will assume that this approach has been taken, though we believe that the analysis can be adapted to different approaches for the first time step. 
	
\end{remark}

	\begin{remark}[Nonhomogeneous boundary conditions for $Q$]
		\label{rem:inhomo}
		
		The scheme~\eqref{eq:step1fully} - \eqref{eq:projectionfullydiscrete} is for homogeneous Dirichlet boundary conditions for the Q-tensor variable, but it can be extended to nonhomogeneous non-time dependent Dirichlet conditions in a straightforward manner: Denote $Q_b(x):\partial\dom\to \R^{d\times d}$ the (symmetric and trace-free) boundary condition and $\widetilde{Q}_b$ an extension of $Q_b$ to the whole domain (which is possible as long as $Q_b$ is sufficiently regular and the boundary is uniformly Lipschitz~\cite{Necas2012}). If the boundary data $Q_b$ is trace-free and symmetric, then the extension can be chosen trace-free and symmetric.
		We can then decompose $Q = \widetilde{Q}_b + \breve{Q}$ where $\breve{Q}$ has vanishing boundary values. Denote $\widetilde{Q}_{b,h}=\Ih \widetilde{Q}_b$ the Lagrangian interpolation of  $\widetilde{Q}_b$ on  the finite element space $\Mh$. Then the approximations $\breve{Q}_h\in \Mhz$ of $\breve{Q}$ solve
		\begin{equation*}
			\left(\frac{3\breve{Q}^{m+1}_h-4\breve{Q}^m_h+\breve{Q}^{m-1}_h}{2\Delta t},Y\right)+((\tu^{m+1}_h\cdot\Grad) \hQ^{m+1}_h,Y) =   (s_h^{m+1},Y)+M (\HH_h^{m+1},Y),
		\end{equation*}
		where $Q_h = \breve{Q}_h +\widetilde{Q}_{b,h}$ and $\widehat{Q}_h^{m+1} = 2 \breve{Q}_h^m-\breve{Q}_h^{m-1}+\widetilde{Q}_{b,h}$. 
		All the upcoming analysis can be extended to this case with minor modifications (additional source terms would appear that can be bounded). 
	\end{remark}

\subsection{Solvability and properties of the scheme}

Next, we show that the scheme~\eqref{eq:step1fully}--\eqref{eq:projectionfullydiscrete} is well-posed, i.e., given approximations at time $m$, the approximations at time $m+1$ can be computed in a unique manner.
No relation between the time step $\Delta t$ and the mesh size $h$ will be required for the solvability and stability of the scheme, however, we will need an inverse CFL condition of the form {$h^{2}=o(\Delta t)$} in the next section to show convergence of the scheme to a weak solution.
We start by introducing some notation. We denote
\begin{equation*}
	\mathcal{S}_0 = \left\{Q\in \R^{d\times d}\, | \, Q_{ij}= Q_{ji},\, i,j=1,\dots,d;\, \tr Q =0\right\},
\end{equation*}
the set of symmetric, trace-free tensors and denote by $\Pi:\R^{d\times d}\to \Sym$ the orthogonal projection (with respect to the Frobenius norm) onto the space of symmetric and trace-free tensors. One has $\Pi A = \frac{A+A^\top}{2} - \frac{1}{d}\tr (A) I$. If $A:\dom\to \R^{d\times d}$ is a matrix-valued mapping then by $\Pi A$ we mean the pointwise in $x\in \dom$ action of $\Pi$. We also denote $\Sh = \Pi \Mh$ and $\Shz = \Pi\Mhz$.
We collect the following basic properties of $\Pi$:
\begin{lemma}
	\label{lem:propertiesofPi}
	The orthogonal projection $\Pi:\R^{d\times d}\to \Sym$ satisfies
	\begin{enumerate}[{\rm(a)}]
		\item $\Pi$ is self-adjoint with respect to $(\cdot,\cdot)$ and $(\cdot,\cdot)_h$
		\item $\Pi$ commutes with $\partial_t$ and $\partial_i$, $i=1,\dots, d$ and therefore $\Grad \Pi A \perp \Grad ((I-\Pi)A)$.
		
		\item If $Q\in \Sym$, then $s(u,Q)\in\Sym$ for any vector field $u:\dom\to \R^d$.
		\item If $Q\in \Sym$, then $V(Q),P(Q)\in \Sym$.
		\item If $Q:\dom\to\Sym$, then $(u\cdot \Grad )Q\in \Sym$ for any $u\in \R^d$. 
		\item $\Pi$ commutes with finite element  interpolation operators  and $L^2$-projection operators.
	\end{enumerate}
\end{lemma}
\begin{proof}
	Property (a) follows from
	\begin{equation}\label{eq:skewsym}
		A:(B-B^\top)=(A-A^\top):B,\quad A:(A-A^\top)=\frac12 (A-A^\top):(A-A^\top),
	\end{equation}
	for matrices $A,B\in\R^{d\times d}$, 
	\begin{equation*}
		A:(\tr(B) I) = \tr A \tr B = (\tr (A) I):B,
	\end{equation*}
	and because the mass-lumping is nodal and $\Pi$ acts nodally. Property (b) follows from direct computation. For property (c), we compute (recall~\eqref{eq:DandW})
		\begin{multline*}
		\tr s(u,Q)= \tr(W Q)-\tr(QW)+\xi(\tr(Q D)+\tr(DQ))+\frac{2\xi}{d}\Div u\\
		-\frac{2\xi}{d}\Div u -2\xi(D:Q)\tr\left(Q+\frac{1}{d}I\right)\\
		=\tr(WQ)-\tr(QW)+\xi(\tr(QD)+\tr(DQ))-2\xi(D:Q).
	\end{multline*}
	Since $D$ and $Q$ are  symmetric, we have $\tr(QD)=\tr(DQ)=D:Q$.
	Moreover, 
	\begin{equation*}
		\tr(WQ)=\frac12\tr(\Grad u Q)-\frac12\tr((\Grad u)^\top Q)
		=\frac12\tr(\Grad u Q)-\frac12\tr(Q(\Grad u)^\top )
		= \frac12\tr(\Grad u Q)-\frac12\tr(\Grad u Q^\top )=0, 
	\end{equation*}
	using trace identities and that $Q$ is symmetric. Hence $\tr s(u,Q)=0$. We continue to show the symmetry property of $s$:
	\begin{equation*}
	\begin{split}
		s(u,Q)^\top& =(WQ)^\top-(QW)^\top+\xi((QD)^\top+(DQ)^\top)+\frac{2\xi}{d}D^\top-\frac{2\xi}{d^2}\Div u I- 2\xi (D:Q)\left(Q^\top+\frac1{d}I\right)\\
		&=Q^\top W^\top-W^\top Q^\top+\xi( D^\top Q^\top+Q^\top D^\top)+\frac{2\xi}{d}D-\frac{2\xi}{d^2}\Div u I- 2\xi (D:Q)\left(Q+\frac1{d}I\right)\\
		&=-Q W+W Q+\xi( D Q+Q D)+\frac{2\xi}{d}D-\frac{2\xi}{d^2}\Div u I- 2\xi (D:Q)\left(Q+\frac1{d}I\right)\\
		&=s(u,Q)
		\end{split}
	\end{equation*}
		where we used the symmetry of $D$ and the skew-symmetry of $W$. 
	Properties (d), (e) and (f) follow by direct computation.
\end{proof}

We will also need the following technical lemma from~\cite{Zhao2017,Yue2023} in the fully discrete setting:
\begin{lemma}
	\label{lem:Zhaoetalidentity}
	Assume that $\hQ^{m+1}_h$ and $\HH^{m+1}_h$ are symmetric and trace-free. Then
	the functions $\sigma$ and $s$ defined in~\eqref{eq:sigmaused} and~\eqref{eq:littles} satisfy
	\begin{equation}\label{eq:ssigma}
		(s^{m+1}_h,\HH^{m+1}_h)+(\sigma^{m+1}_h,\Grad\tu^{m+1}_h)=0.
	\end{equation}
\end{lemma}
\begin{remark}\label{rem:hQ}
	$\hQ^{m+1}_h=2Q^m_h - Q^{m-1}_h$ is symmetric and trace-free if $Q^m_h$ and $Q^{m-1}_h$ are.
\end{remark}
\begin{proof}
	We have, using that $\HH_h^{m+1}$ is trace-free and $\hQ^{m+1}_h$ is symmetric,
	\begin{align*}
		&(s^{m+1}_h,\HH^{m+1}_h) = \Bigg( \widetilde{W}^{m+1}_h \hQ^{m+1}_h-\hQ^{m+1}_h \widetilde{W}^{m+1}_h+\xi(\hQ^{m+1}_h \widetilde{D}^{m+1}_h+\widetilde{D}^{m+1}_h\hQ^{m+1}_h)+\frac{2\xi}{d}\widetilde{D}^{m+1}_h\\
		&\qquad\qquad \qquad\qquad  -\frac{2\xi}{d^2}\Div \tu^{m+1}_h I -2\xi(\widetilde{D}^{m+1}_h:\hQ^{m+1}_h)\left(\hQ^{m+1}_h+\frac{1}{d}I\right) ,  \HH^{m+1}_h  \Bigg)\\
		&=  (\widetilde{W}^{m+1}_h \hQ^{m+1}_h-\hQ^{m+1}_h \widetilde{W}^{m+1}_h,\HH^{m+1}_h)+\xi(\hQ^{m+1}_h \widetilde{D}^{m+1}_h+\widetilde{D}^{m+1}_h \hQ^{m+1}_h,\HH_h^{m+1})+\frac{2\xi}{d}(\widetilde{D}^{m+1}_h,\HH_h^{m+1})\\
		&\quad-\frac{2\xi}{d^2}(\Div \tu^{m+1}_h,\tr \HH_h^{m+1}) -2\xi((\widetilde{D}^{m+1}_h:\hQ^{m+1}_h)\hQ^{m+1}_h,\HH_h^{m+1})-\frac{2\xi}{d}((\widetilde{D}^{m+1}_h:\hQ^{m+1}_h),\tr\HH^{m+1}_h )\\
		&=  (\widetilde{W}^{m+1}_h \hQ^{m+1}_h-\hQ^{m+1}_h \widetilde{W}^{m+1}_h,\HH^{m+1}_h)+\xi(\hQ^{m+1}_h \widetilde{D}^{m+1}_h+\widetilde{D}^{m+1}_h \hQ^{m+1}_h,\HH_h^{m+1})\\
		&\quad+\frac{2\xi}{d}(\widetilde{D}^{m+1}_h,\HH_h^{m+1})
		-2\xi((\widetilde{D}^{m+1}_h:\hQ^{m+1}_h)\hQ^{m+1}_h,\HH_h^{m+1})\\
		&=\int_{\dom} \tr(\widetilde{W}^{m+1}_h \hQ^{m+1}_h\HH_{h}^{m+1})-\tr(\hQ^{m+1}_h \widetilde{W}^{m+1}_h \HH_{h}^{m+1}) +\xi\tr(\hQ^{m+1}_h \widetilde{D}^{m+1}_h\HH_{h}^{m+1})dx\\ 
		&\quad+\int_{\dom}\xi\tr(\widetilde{D}^{m+1}_h \hQ^{m+1}_h\HH_{h}^{m+1})+\frac{2\xi}{d}\tr(\widetilde{D}^{m+1}_h\HH_h^{m+1})-2\xi\tr(\widetilde{D}^{m+1}_h \hQ^{m+1}_h)\tr(\hQ^{m+1}_h\HH_h^{m+1}) dx\\
		&=\frac{1}{2}\int_{\dom} \tr(\Grad \tu^{m+1}_h \hQ^{m+1}_h\HH_{h}^{m+1})-\tr((\Grad \tu^{m+1}_h)^\top \hQ^{m+1}_h\HH_{h}^{m+1})-\tr(\hQ^{m+1}_h \Grad\tu^{m+1}_h \HH_{h}^{m+1}) dx \\
		&\quad +\frac12\int_{\dom}\tr(\hQ^{m+1}_h (\Grad\tu^{m+1}_h)^\top \HH_{h}^{m+1})dx+\frac{\xi}{2}\int_{\dom}\tr(\hQ^{m+1}_h \Grad \tu^{m+1}_h\HH_{h}^{m+1})+\tr(\hQ^{m+1}_h (\Grad \tu^{m+1}_h)^\top\HH_{h}^{m+1})dx\\ &\quad+\frac{\xi}{2}\int_{\dom}\tr(\Grad\tu^{m+1}_h\hQ^{m+1}_h\HH_{h}^{m+1})+\tr((\Grad\tu^{m+1}_h)^\top \hQ^{m+1}_h\HH_{h}^{m+1}) dx\\
		&\quad +\frac{\xi}{d}\int_{\dom}\tr(\Grad\tu^{m+1}_h\HH_h^{m+1})+\tr((\Grad\tu^{m+1}_h)^\top\HH_h^{m+1}) dx\\
		&\quad-\xi\int_{\dom}\left(\tr(\Grad\tu^{m+1}_h \hQ^{m+1}_h)+\tr((\Grad\tu^{m+1}_h)^\top \hQ^{m+1}_h)\right)\tr(\hQ^{m+1}_h\HH_h^{m+1}) dx\\
		&=\int_{\dom} \tr(\HH_{h}^{m+1}\hQ^{m+1}_h(\Grad \tu^{m+1}_h)^\top)-\tr(\hQ^{m+1}_h\HH_{h}^{m+1}(\Grad \tu^{m+1}_h)^\top )dx\\
		&\quad+\xi\int_{\dom}\tr(\hQ^{m+1}_h\HH_{h}^{m+1}  (\Grad \tu^{m+1}_h)^\top)+\tr(\HH_{h}^{m+1} \hQ^{m+1}_h (\Grad \tu^{m+1}_h)^\top)dx\\
		&\quad +\frac{2\xi}{d}\int_{\dom}\tr(\HH_h^{m+1} (\Grad\tu^{m+1}_h)^\top) dx-2\xi\int_{\dom}\left( \tr( \hQ^{m+1}_h (\Grad\tu^{m+1}_h)^\top)\right)\tr(\hQ^{m+1}_h\HH_h^{m+1}) dx\\
		& = \left(\HH^{m+1}_h\hQ^{m+1}_h-\hQ^{m+1}_h\HH^{m+1}_h,\Grad \tu^{m+1}_h\right) + \xi\left(\hQ^{m+1}_h\HH_h^{m+1}+\HH^{m+1}_h\hQ^{m+1}_h,\Grad\tu^{m+1}_h\right)\\
		&\quad + \frac{2\xi}{d}(\HH^{m+1}_h,\Grad\tu^{m+1}_h)-2\xi \left((\hQ^{m+1}_h:\HH^{m+1}_h)\hQ^{m+1}_h,\Grad\tu^{m+1}_h\right)\\
		& = -(\sigma^{m+1}_h,\Grad\tu^{m+1}_h),
	\end{align*}
	where we used trace identities and the symmetry of $\hQ^{m+1}_h$ and $\HH_h^{m+1}$ for the last steps.		
\end{proof}

Next, we show that the scheme~\eqref{eq:step1fully} - \eqref{eq:projectionfullydiscrete} propagates the symmetry and trace-free constraint of $Q^m_h$ and $\HH^m_h$ through the time stepping.
\begin{lemma}
	\label{lem:tracefreediscrete}
	If $Q_h^m$ and $Q_h^{m-1}$ are trace-free and symmetric, then so are $Q^{m+1}_h$ and $\HH_h^{m+1}$ computed by scheme \eqref{eq:step1fully} - \eqref{eq:projectionfullydiscrete}. If $Q_0$ is trace-free and symmetric, then so is $Q_h^0$, and any solution $Q_h^1$ and $\HH_h^1$ of~\eqref{eq:step1fully1step} - \eqref{eq:projectionfullydiscretestep1}.
\end{lemma}
\begin{proof}
	We first assume $m\geq 1$. We choose $Y=I\tr \HH^{m+1}_h$ as a test function in~\eqref{eq:Qdisc}, where $I$ is the $d\times d$ identity matrix:
	\begin{equation*}
	\begin{split}
	&\left(\frac{\textcolor{blue}{3}\tr Q^{m+1}_h-4\tr Q^m_h+\tr Q^{m-1}_h}{2\Delta t},\tr \HH^{m+1}_h\right)+(\tr [(\tu^{m+1}_h\cdot\Grad) \hQ^{m+1}_h],\tr \HH^{m+1}_h)- (\tr s_h^{m+1},\tr \HH^{m+1}_h)\\
	&\qquad = M (\tr\HH_h^{m+1},\tr \HH^{m+1}_h).
	\end{split}
	\end{equation*}
	By assumption, we have $\tr Q^m_{h}=\tr Q^{m-1}_h=0$. Therefore, $\tr \hQ^{m+1}_h=0$, cf. Remark~\ref{rem:hQ}. Then, by Lemma~\ref{lem:propertiesofPi} (c) and (e),
	\begin{equation*}
	\tr[(\tu^{m+1}_h\cdot\Grad) \hQ^{m+1}_h]= \tr s_h^{m+1}=0.
	\end{equation*}
	Hence
	\begin{equation}\label{eq:tracepart1}
	\frac{3}{2\Delta t}(\tr Q^{m+1}_h,\tr \HH_h^{m+1}) = M (\tr\HH_h^{m+1},\tr \HH^{m+1}_h).
	\end{equation}
	Next we take $Z=I\tr Q^{m+1}_h$ as a test function in~\eqref{eq:Hdisc}:
	\begin{equation*}
		(\tr \HH_h^{m+1},\tr Q^{m+1}_h) = -L(\Grad Q_h^{m+1},\Grad (I\tr Q^{m+1}_h))-(\tr(r_h^{m+1} P(\hQ_h^{m+1})),\tr Q^{m+1}_h)_h.
	\end{equation*}
	We have
	\begin{equation*}
	\tr(r_h^{m+1} P(\hQ_h^{m+1}))= r_h^{m+1} \tr (P(\hQ_h^{m+1}))=0,
	\end{equation*}
	since $P(\hQ^{m+1}_h)$ is trace-free, cf. Lemma~\ref{lem:propertiesofPi} (d). Moreover,
	\begin{equation*}
	\begin{split}
	(\Grad Q_h^{m+1},\Grad (I\tr Q^{m+1}_h))&=\int_{\dom}\sum_{i,j,k,\ell=1}^d \partial_k (Q^{m+1}_h)_{ij}\partial_k(\delta_{ij}(Q^{m+1}_h)_{\ell\ell})dx\\
	&=\int_{\dom}\sum_{j,k,\ell=1}^d \partial_k (Q^{m+1}_h)_{jj}\partial_k(Q^{m+1}_h)_{\ell\ell}dx\\
	&=\norm{\Grad\tr Q^{m+1}_h}_{L^2}^2.
	\end{split}
	\end{equation*}
	Thus,
	\begin{equation*}
	(\tr \HH_h^{m+1},\tr Q^{m+1}_h) = -L\norm{\Grad\tr( Q^{m+1}_h)}^2_{L^2}.
	\end{equation*}
	Hence,~\eqref{eq:tracepart1} becomes
	\begin{equation*}
		-\frac{3L}{2\Delta t}\norm{\Grad\tr( Q^{m+1}_h)}^2_{L^2} = M\norm{\tr( \HH^{m+1}_h)}^2_{L^2},
	\end{equation*}
	which clearly implies that $\HH_h^{m+1}$ is trace free and by Poincar\'e inequality, $Q^{m+1}_h$ as well. 
We now show the symmetry property of $Q^{m+1}_h$. To do so, we take $V^{m+1}_h:= \HH^{m+1}_h-(\HH^{m+1}_h)^\top$ as a test function in~\eqref{eq:Qdisc} and use~\eqref{eq:skewsym}.  Hence we obtain in~\eqref{eq:Qdisc},
\begin{align*}
	\left(\frac{3Q^{m+1}_h-4Q^m_h+Q^{m-1}_h}{2\Delta t},V^{m+1}_h\right)+((\tu^{m+1}_h\cdot\Grad) \hQ^{m+1}_h,V^{m+1}_h)- (s_h^{m+1},V^{m+1}_h) &= M (\HH_h^{m+1},V^{m+1}_h),\\
\Leftrightarrow\qquad	&\\
	\frac{3}{2\Delta t}\left(Q^{m+1}_h,V^{m+1}_h\right)- (s_h^{m+1}-(s_h^{m+1})^\top,\HH^{m+1}_h) &= \frac{M}{2} (V_h^{m+1},V^{m+1}_h).
\end{align*}
Now, $s$ is symmetric since $Q=\hQ^{m+1}_h$ is by Lemma~\ref{lem:propertiesofPi}, (c). Similarly $(\tu_h^{m+1}\cdot\Grad)\hQ^{m+1}_h$ is symmetric by Lemma~\ref{lem:propertiesofPi}, (e).
Thus, we get
\begin{equation}
\label{eq:sympart1}
	\frac{3}{2\Delta t}(Q^{m+1}_h,V^{m+1}_h) = \frac{M}{2} (V_h^{m+1},V^{m+1}_h).
\end{equation}
Taking $Z=Q^{m+1}_h-(Q^{m+1}_h)^\top$ as a test function in~\eqref{eq:Hdisc} and using the same tricks and the symmetry of $P(\hQ^{m+1}_h)$, we obtain
\begin{equation*}
	(Q_h^{m+1},V^{m+1}_h) = -\frac{L}{2}(\Grad (Q_h^{m+1}-(Q_h^{m+1})^\top),\Grad (Q_h^{m+1}-(Q_h^{m+1})^\top)),
\end{equation*}
thus, plugging this into~\eqref{eq:sympart1}, we get
\begin{equation*}
	-\frac{3L}{4\Delta t}\norm{\Grad (Q_h^{m+1}-(Q_h^{m+1})^\top)}^2_{L^2} = \frac{M}{2} \norm{V_h^{m+1}}^2_{L^2},
	\end{equation*}
hence $\HH^{m+1}_h$ and $\Grad Q_h^{m+1}$ are symmetric. By the Poincar\'e inequality, $Q_h^{m+1}$ is as well.
If $Q_0$ is symmetric and trace-free, then so is $Q_h^0$ by Lemma~\ref{lem:propertiesofPi}, (f). Now it follows along the lines of the general step $m\geq 1$ above that any solution of the first step with the backward Euler  step~\eqref{eq:step1fully1step} - \eqref{eq:projectionfullydiscretestep1} is trace-free and symmetric since $Q_h^0$ is.

\end{proof}

Now we attempt to show the solvability of the scheme:
\begin{lemma}
	\label{lem:solvability}
	For every $\Delta t, h>0$ and every $m=1,2\dots$, given $(u^{m-1}_h,Q^{m-1}_h,r^{m-1}_h)\in \Yh\times\Shz\times\Xh$ and $(u^m_h,p^m_h,Q^m_h,r^m_h)\in \Yh\times\Ph\times\Shz\times\Xh$, and $\tu_h^{m-1},\tu_h^m\in \Uh$, there exists a unique $(u^{m+1}_h,p^{m+1}_h,Q^{m+1}_h,\HH^{m+1}_h,r^{m+1}_h)\in \Yh\times\Ph\times\Mhz\times\Mh\times\Xh$ solving~\eqref{eq:step1fully} -- \eqref{eq:projectionfullydiscrete}. Moreover, $Q^{m+1}_h\in \Shz$ and $\HH_h^{m+1}\in \Sh$.
\end{lemma}
\begin{proof}
	We first show that~\eqref{eq:step1fully} - \eqref{eq:projectionfullydiscrete} has a unique solution when constraining to symmetric trace-free $Q$ and $\HH$.
	So we start by showing that if $(u^{m-1}_h,Q^{m-1}_h,r^{m-1}_h)\in \Yh\times\Shz\times\Xh$ and $(u^m_h,p^m_h,Q^m_h,r^m_h)\in \Yh\times\Ph\times\Shz\times\Xh$, and $\tu_h^m,\tu_h^{m-1}\in \Uh$, then there exists a unique $(\tu^{m+1}_h,Q^{m+1}_h,\HH^{m+1}_h,r^{m+1}_h)\in \Uh\times\Shz\times\Sh\times \Xh$ satisfying~\eqref{eq:step1fully}, i.e., the first step of the algorithm is well-posed. To do so, we write~\eqref{eq:step1fully} as a linear variational problem after using that $r^{m+1}_h$ is defined via~\eqref{eq:rsimplified}: Find $(\tu,Q,\HH )\in \Uh\times\Shz\times\Sh=:\Kh $ such that for all $(v,Z,Y )\in \Kh $,
	\begin{equation}
	\label{eq:elliptic}
	a((\tu,Q,\HH),(v,Z,Y))= F^m((v,Z,Y)),
	\end{equation}
	where 
	\begin{multline}
	\label{eq:defa}
	a((\tu,Q,\HH),(v,Z,Y)) =  \frac{3}{2}	\left(\tu,v\right)+ \Delta t b(\hu^{m+1}_h,\tu ,v)+\Delta t\mu(\Grad \tu ,\Grad v)+\Delta t(\sigma_h,\Grad v)\\
	+\Delta t(\HH\Grad \hQ^{m+1}_h,v)
	-\frac32 (Q,Y)-\Delta t((\tu \cdot\Grad) \hQ^{m+1}_h,Y)	+\Delta t (s_h ,Y)\\
+\Delta t M(\HH,Y) + \frac32(\HH,Z)+\frac32 L(\Grad Q,\Grad Z)+\frac32( (P(\hQ^{m+1}_h):Q)  P(\hQ^{m+1}_h),Z)_h
	\end{multline}
	where $s_h = s(\tu,\hQ^{m+1}_h)$ and $\sigma_h =\sigma(\hQ^{m+1}_h,\HH)$ and 
	\begin{multline}
	\label{eq:deffn}
	F^m((v,Z,Y)) = 2(u^m_h,v)-\frac12(u^{m-1}_h,v)-2(Q^m_h,Y)+\frac12(Q^{m-1}_h,Y)-2(r^m_h P(
	\hQ^{m+1}_h),Z)_h +\frac12(r^{m-1}_h P(\hQ^{m+1}_h),Z)_h\\
	+2 ((P(\hQ^{m+1}_h):Q^m_h)P(\hQ^{m+1}_h),Z)_h-\frac12((P(\hQ^{m+1}_h):Q^{m-1}_h)P(\hQ^{m+1}_h),Z)_h-\Delta t (\Grad p^m_h,v)+\Delta t(\PUh f^{m+1},v).
		\end{multline}
We show that $a$ is coercive and bounded on $\Kh$ and that $F^m:\Kh\to \R$ is bounded, and then use the Lax--Milgram theorem to conclude well-posedness on $\Kh$ with the norm $\norm{(v,Z,Y)}_{\Kh}:=\left(\norm{v}_{H^1_0}^2+\norm{Z}_{H^1_0}^2+\norm{Y}_{L^2}^2\right)^{1/2}$.
	To show coercivity, we plug in $(v,Z,Y)=(\tu,Q,\HH)$:
	\begin{multline*}
		a((\tu,Q,\HH),(\tu,Q,\HH)) =  	\frac32\norm{\tu}_{L^2}^2+ \Delta t b(\hu^{m+1}_h,\tu ,\tu)+\Delta t\mu(\Grad \tu ,\Grad \tu)+\Delta t(\sigma_h,\Grad \tu)\\
	+\Delta t(\HH\Grad \hQ^{m+1}_h,\tu)
	-\frac32(Q,\HH)-\Delta t((\tu \cdot\Grad) \hQ^{m+1}_h,\HH)	+\Delta t (s_h ,\HH)\\
	+\Delta t M\norm{\HH}_{L^2}^2 +\frac32 (\HH,Q)+\frac32L(\Grad Q,\Grad Q)+\frac32( (P(\hQ^{m+1}_h):Q)  P(\hQ^{m+1}_h),Q)_h\\
= \frac32\norm{\tu}_{L^2}^2+ \Delta t\mu\norm{\Grad \tu}_{L^2}^2+\Delta t M\norm{\HH}_{L^2}^2  +\frac32L\norm{\Grad Q}_{L^2}^2+\frac32\norm{P(\hQ^{m+1}_h):Q}^2_h\\
 \geq c\left(\norm{\tu}_{H^1_0}^2+\norm{\HH}_{L^2}^2+\norm{Q}_{H^1_0}^2\right),
	\end{multline*}
	where we used the skew-symmetry of $b$ and Lemma~\ref{lem:Zhaoetalidentity}. This proves the coercivity on $\Kh$. For the boundedness, we have
	\begin{align*}
	\left| 	a((\tu,Q,\HH),(v,Z,Y))\right| \leq  &  	\frac32\norm{\tu}_{L^2}\norm{v}_{L^2}+ \frac{\Delta t}{2}\left( \norm{\hu^{m+1}_h}_{L^4}\norm{\tu}_{H^1_0}\norm{v}_{L^4}+ \norm{\hu^{m+1}_h}_{L^4}\norm{v}_{H^1_0}\norm{\tu}_{L^4} \right)\\
	&+\Delta t\mu\norm{\Grad \tu}_{L^2}\norm{\Grad v}_{L^2}+\Delta t C \left(1+\norm{\hQ_h^{m+1}}_{L^\infty}^2\right)\norm{\HH}_{L^2}\norm{\Grad v}_{L^2} \\
	&+\Delta t\norm{\HH}_{L^2}\norm{\Grad \hQ^{m+1}_h}_{L^{3}}\norm{v}_{L^6}+\frac32\norm{Q}_{L^2}\norm{Y}_{L^2}+\Delta t \norm{\tu}_{L^6}\norm{\Grad \hQ^{m+1}_h}_{L^{3}}\norm{Y}_{L^2}\\
&+\Delta t C (1+\norm{\hQ_h^{m+1}}_{L^\infty}^2)\norm{\Grad \tu}_{L^2}\norm{Y}_{L^2}	+\Delta t M \norm{\HH}_{L^2}\norm{Y}_{L^2} +\frac32 \norm{\HH}_{L^2}\norm{Z}_{L^2}\\
&+\frac32 L\norm{Q}_{H^1_0}\norm{Z}_{H^1_0}+\frac32 \norm{P(\hQ^{m+1}_h):Q}_{h} \norm{ P(\hQ^{m+1}_h):Z}_h\\
\leq & \frac32\norm{\tu}_{L^2}\norm{v}_{L^2}+ \Delta t( \norm{\hu^{m+1}_h}_{L^4}+\mu)\norm{\tu}_{H^1_0}\norm{v}_{H^1_0} \\
&+\Delta t C \left(1+\norm{\hQ_h^{m+1}}_{L^\infty}^2+  \norm{\Grad \hQ^{m+1}_h}_{L^{3}} \right)\norm{\HH}_{L^2}\norm{ v}_{H^1_0}+\frac32\norm{Q}_{L^2}\norm{Y}_{L^2} \\
& 	+\Delta t C \left(1+\norm{\hQ_h^{m+1}}_{L^\infty}^2+\norm{\Grad \hQ^{m+1}_h}_{L^{3}}\right)\norm{  \tu}_{H^1_0}\norm{Y}_{L^2}	+\Delta t M \norm{\HH}_{L^2}\norm{Y}_{L^2}\\
& + \frac32\norm{\HH}_{L^2}\norm{Z}_{L^2}+\frac32L\norm{ Q}_{H^1_0}\norm{Z}_{H^1_0} + C\norm{\hQ^{m+1}_h}_{L^\infty}^2\norm{ Q}_{L^2} \norm{  Z}_{L^2} 
	\end{align*}
	where we used the Sobolev embedding theorem, the Poincar\'e inequality, the Lipschitz continuity of $P$ (see e.g.~\cite[Thm 4.11]{Gudibanda2022}),   the properties of the discrete inner product, Lemma~\ref{lem:masslumped} and that $Q^m_h$ and $Q^{m-1}_h$ are piecewise polynomial and continuous. Thus for every fixed $h>0$,
	\begin{equation*}
	\left| 	a((\tu,Q,\HH),(v,Z,Y))\right|\leq C \left(\norm{\tu}_{H^1_0}+\norm{Q}_{H^1_0}+\norm{\HH}_{L^2}\right)\left(\norm{v}_{H^1_0}+\norm{Z}_{H^1_0}+\norm{Y}_{L^2}\right),
	\end{equation*}
	and hence $a$ is bounded on $\Kh$. 
	To show boundedness of $F^m$ (for every $m$), we estimate:
	\begin{align*}
	|F^m((v,Z,Y))|&\leq 2\norm{u^m_h}_{L^2}\norm{v}_{L^2}+\frac12\norm{u^{m-1}_h}_{L^2}\norm{v}_{L^2}+2\norm{Q^m_h}_{L^2}\norm{Y}_{L^2}+\frac12\norm{Q^{m-1}_h}_{L^2}\norm{Y}_{L^2}\\
	&\quad +2\norm{r^m_h P(\hQ^{m+1}_h)}_h\norm{Z}_h+\frac12\norm{r^{m-1}_h P(\hQ^{m+1}_h)}_h\norm{Z}_h+2\norm{P(\hQ^{m+1}_h):Q^m_h}_h\norm{P(\hQ^{m+1}_h):Z}_h\\
	&\quad +\frac12\norm{P(\hQ^{m+1}_h):Q^{m-1}_h}_h\norm{P(\hQ^{m+1}_h):Z}_h +\Delta t \norm{\Grad p^m_h}_{L^2}\norm{v}_{L^2}+\Delta t \norm{\PUh f^{m+1}}_{L^2}\norm{v}_{L^2}\\
	&\leq 2 \norm{u^m_h}_{L^2}\norm{v}_{L^2}+\frac12\norm{u^{m-1}_h}_{L^2}\norm{v}_{L^2}+2\norm{Q^m_h}_{L^2}\norm{Y}_{L^2}+\frac12\norm{Q^{m-1}_h}_{L^2}\norm{Y}_{L^2}\\
	&\quad +C\left(\norm{r^m_h}_{L^2}+\norm{r^{m-1}_h}_{L^2}\right) \norm{\hQ^{m+1}_h}_{L^\infty}\norm{Z}_{L^2}+C\norm{\hQ^{m+1}_h}^2_{L^\infty}\left(\norm{Q^m_h}_{L^2}+\norm{Q_h^{m-1}}_{L^2}\right)\norm{Z}_{L^2}\\
	&\quad +\Delta t \norm{\Grad p^m_h}_{L^2}\norm{v}_{L^2}+\Delta t\norm{ f^{m+1}}_{L^2}\norm{v}_{L^2},
	\end{align*}
	where we have used the properties of the discrete inner product, that $Q^m_h$ and $Q^{m-1}_h$ are piecewise polynomial and continuous, and that $\Grad p^m_h$ is a piecewise polynomial and bounded. Thus, with the Poincar\'e inequality, we obtain boundedness of $F^m$ on $\Kh$. 
	Using the Lax--Milgram theorem and Remark~\ref{rem:eqforr}, this implies the solvability of Step 1 of the algorithm on the constrained space $\Kh$. 
	Now let $(\tu,Q,\HH)$ be the solution of~\eqref{eq:elliptic}. We claim that it satisfies~\eqref{eq:step1fully} for any $(v,Z,Y)\in \Uh\times\Mhz\times\Mh$.  Given $Y\in \Mh$, we write $Y=\Pi Y +(I-\Pi)Y$ where $\Pi Y$ is admissible in~\eqref{eq:elliptic}. Then since $Q,Q_h^m,Q_h^{m-1}\in\Shz$ and $\HH\in \Sh$, and by Lemma~\ref{lem:propertiesofPi} (c) and (e),  $(\tu\cdot\Grad )\hQ^{m+1}_h$ and $s_h^{m+1}$ take values in $\Sym$ a.e., tested against $(I-\Pi)Y$ they are all zero due to the self-adjointness of $\Pi$, Lemma~\ref{lem:propertiesofPi}. Thus  the equation for $Q$,~\eqref{eq:Qdisc} holds for general $Y\in \Mh$. Similarly for $Z\in \Mhz$, we have $\Pi Z\in \Shz$ and $(\HH,(I-\Pi)Z)=0$ and $(\Grad Q,\Grad(I-\Pi)Z)=0$ by Lemma~\ref{lem:propertiesofPi} (a) and (b). We also have
$(r^{m+1}P(\hQ^{m+1}_h),(I-\Pi) Z)_h=0$ by Lemma~\ref{lem:propertiesofPi} (a) and (d) (and the same applies to the terms involving $r$ in $F^m$). Thus we can take arbitrary $Z\in \Mhz$ as test functions in~\eqref{eq:Hdisc}. The equations for $\tu$ and $r$ are unchanged as~\eqref{eq:elliptic} already uses test functions in $\Uh$ and the equation for $r$ was directly derived from~\eqref{eq:rdisc}, see Remark~\ref{rem:eqforr}.
	Now by Lemma~\ref{lem:tracefreediscrete}, any solution of~\eqref{eq:step1fully} has $Q^{m+1}_h$ and $\HH_h^{m+1}$ trace-free and symmetric, thus the solution of the prediction step must be unique on $\Uh\times \Mhz\times \Mh\times \Xh$.

		For Step 2, we proceed similarly. We use the formulation of~\eqref{eq:projection1}--\eqref{eq:projection2} as a Poisson problem,~\eqref{eq:poisson}. So we seek $p\in \Ph$ such that for any $q\in \Ph$,
	\begin{equation*}
		\widetilde{a}(p,q) = \widetilde{{F}}^m(q),
	\end{equation*}
	where $\widetilde{a}:\Ph\times\Ph\to \R$ is given by
	\begin{equation*}
		\widetilde{a}(p,q) = (\Grad p,\Grad q),
	\end{equation*}
	and $\widetilde{{F}}^m:\Ph\to \R$ is given by
	\begin{equation*}
		\widetilde{{F}}^m(q) = (\Grad p^m_h,\Grad q)-\left(\frac{3\Div \tu^{m+1}_h}{2\Delta t},q\right).
	\end{equation*}
 $\widetilde{a}$ is coercive and bounded on $H^1(\dom)\cap L^2_0(\dom)$ and therefore on $\Ph$, and $\widetilde{{F}}^m$ is bounded from $\Ph$ to $\R$ since $\tu^{m+1}_h \in H^1_0(\dom)$ and $\Grad p^m_h\in L^2(\dom)$ since $p_h^m\in \Ph$. Thus, applying the Lax--Milgram theorem again, $p^{m+1}_h$ exists and is unique, and therefore $u^{m+1}_h$ defined through~\eqref{eq:projection1} exists and is unique too. Moreover, the fact that $p^{m+1}_h$ satisfies~\eqref{eq:poisson} for all $q\in \Ph$ implies that $u^{m+1}_h$ satisfies the weak divergence constraint~\eqref{eq:projection2}.
\end{proof}
\begin{remark}
	\label{rem:firsteulerstep}
	 The proof of the solvability of the first (Euler) step~\eqref{eq:step1fully1step} - \eqref{eq:projectionfullydiscretestep1} follows analogously with slightly different coefficients.
\end{remark}
\subsection{Energy stability}
Next, we show the discrete energy inequality which will result in a priori estimates on the approximations that will allow us to derive pre-compactness of the approximations. 
 
\begin{lemma}
	\label{lem:discenergyestimate}
	Assume $\Delta t\leq \frac16$. 
Then the approximations computed by the scheme~\eqref{eq:step1fully}--\eqref{eq:projectionfullydiscrete} satisfy the energy estimate
		\begin{multline*}
		E^N_h + \sum_{m=1}^{N-1}\norm{u^{m+1}_h-2u^m_h+ u^{m-1}_h}_{L^2}^2+(3-2\Delta t) \sum_{m=1}^{N-1}\norm{\tu^{m+1}_h-u^{m+1}_h}_{L^2}^2 + 4\mu\Delta t \sum_{m=0}^{N-1}\norm{\Grad\tu_h^{m+1}}_{L^2}^2\\
		+ 4M\Delta t \sum_{m=0}^{N-1}\norm{\HH_h^{m+1}}_{L^2}^2 + L\sum_{m=1}^{N-1}\norm{\Grad Q_h^{m+1}-2\Grad Q^m_h+\Grad Q^{m-1}_h}_{L^2}^2+\sum_{m=1}^{N-1}\norm{r_h^{m+1}-2r^m_h+r^{m-1}_h}_{h}^2\\
		+ 7(1-2\Delta t)\norm{\tu^1_h-u^0_h}_{L^2}^2+7\norm{r_h^1 - r_h^0}_h^2 +7 L\norm{\Grad (Q^1_h-Q_h^0)}_{L^2}^2\\
		\leq C\left(\norm{f}^2_{L^2([0,N\Delta t]\times\dom)}
		+\norm{u_0}_{L^2}^2
		+\norm{Q_0}_{L^4}^4+1 +\norm{\Grad Q_0}_{L^2}^2  
		\right)\exp(3N \Delta t)
	\end{multline*}

for any $N\geq 1$,
	where
 	\begin{multline*}
	E^N_h :=  \norm{u^{N}_h}_{L^2}^2+\norm{2 u^N_h - u^{N-1}_h}_{L^2}^2+L\norm{\Grad Q_h^{N}}_{L^2}^2+L\norm{2\Grad Q_h^N-\Grad Q_h^{N-1}}^2_{L^2}\\
	+	 \norm{r_h^{N}}_h^2+\norm{2r_h^N-r_h^{N-1}}_h^2+\frac{4\Delta t^2}{3}\norm{\Grad p^N_h}_{L^2}^2.
	\end{multline*}

\end{lemma}

\begin{proof}
	We start by noting that the following identity holds (cf. proof of Lemma 3.5 in~\cite{Weber2025}):
	\begin{multline*}
		2 (3 \tu_h^{m+1}-4 u^m_h + u^{m-1}_h,\tu^{m+1}_h)\\
		= 2 (3 u^{m+1}_h-4 u^m_h + u^{m-1}_h,\tu^{m+1}_h - u^{m+1}_h) + 2 (3 u^{m+1}_h-4u^m_h + u^{m-1}_h,u^{m+1}_h)+ 6 (\tu^{m+1}_h-u^{m+1}_h,\tu^{m+1}_h).
	\end{multline*}
	Taking $v = 3u^{m+1}_h-4u^m_h+ u^{m-1}_h$ as a test function in~\eqref{eq:projection1} and using~\eqref{eq:projection2}, we see that
	\begin{equation*}
		2 (3u^{m+1}_h-4u^m_h + u^{m-1}_h,\tu^{m+1}_h-u^{m+1}_h) = 0.
	\end{equation*}
	Thus the previous identity simplifies to
	\begin{align*}
			2 (3 \tu_h^{m+1}-4 u^m_h + u^{m-1}_h,\tu^{m+1}_h)&
		=  2 (3 u^{m+1}_h-4u^m_h + u^{m-1}_h,u^{m+1}_h)+ 6 (\tu^{m+1}_h-u^{m+1}_h,\tu^{m+1}_h)\\
		& = \norm{u^{m+1}_h}_{L^2}^2 - \norm{u^m_h}_{L^2}^2 +\norm{2u^{m+1}_h-u^m_h}_{L^2}^2-\norm{2 u^m_h - u^{m-1}_h}_{L^2}^2\\
		& +\norm{u^{m+1}_h-2u^m_h+ u^{m-1}_h}_{L^2}^2+ 3\norm{\tu^{m+1}_h}_{L^2}^2 - 3\norm{u^{m+1}_h}_{L^2}^2 + 3\norm{\tu^{m+1}_h-u^{m+1}_h}_{L^2}^2.
	\end{align*}
	Therefore, taking $v= 4 \tu^{m+1}_h$ as a test function in~\eqref{eq:udisc} and using this identity as well as the skew-symmetry of $b$, we obtain
	\begin{multline}
		\label{eq:upartialenergy}
		\frac{1}{\Delta t}\bigg( \norm{u^{m+1}_h}_{L^2}^2 - \norm{u^m_h}_{L^2}^2 +\norm{2u^{m+1}_h-u^m_h}_{L^2}^2-\norm{2 u^m_h - u^{m-1}_h}_{L^2}^2\\
	  +\norm{u^{m+1}_h-2u^m_h+ u^{m-1}_h}_{L^2}^2+ 3\norm{\tu^{m+1}_h}_{L^2}^2 - 3\norm{u^{m+1}_h}_{L^2}^2 + 3\norm{\tu^{m+1}_h-u^{m+1}_h}_{L^2}^2\bigg)\\
	  + 4(\Grad p^m_h,\tu^{m+1}_h)+ 4 \mu(\Grad\tu^{m+1}_h,\Grad\tu^{m+1}_h)+4(\sigma^{m+1}_h,\Grad\tu^{m+1}_h)+4(\HH^{m+1}_h\Grad\hQ^{m+1}_h,\tu^{m+1}_h)\\
	  =4(\PUh f^{m+1},\tu^{m+1}_h).
	\end{multline}
	To find an expression for the term $4(\Grad p^m_h,\tu^{m+1}_h)$, we use the second, projection step of the scheme: We take $v = 2(u^{m+1}_h+\tu^{m+1}_h)+\frac{4\Delta t}{3}(\Grad p^{m+1}_h+\Grad p^m_h)$ as a test function in~\eqref{eq:projection1} and simplify:
	\begin{align*}
		0& = \frac{3}{4\Delta t}\left(2(u^{m+1}_h-\tu^{m+1}_h)+\frac{4\Delta t}{3}\Grad(p^{m+1}_h-p^m_h),2(u^{m+1}_h+\tu^{m+1}_h)+\frac{4\Delta t}{3}\Grad(p^{m+1}_h+p^m_h)\right)\\
		&=\frac{3}{4\Delta t}\norm{2u^{m+1}_h+\frac{4\Delta t}{3}\Grad p^{m+1}_h}_{L^2}^2 - \frac{3}{4\Delta t}\norm{2 \tu^{m+1}_h + \frac{4\Delta t}{3}\Grad p^{m}_h}_{L^2}^2\\
		&= \frac{3}{\Delta t}\norm{u^{m+1}_h}_{L^2}^2+\frac{4\Delta t}{3}\norm{\Grad p^{m+1}_h}_{L^2}^2 -  \frac{3}{\Delta t}\norm{\tu^{m+1}_h}_{L^2}^2-\frac{4\Delta t}{3}\norm{\Grad p^{m}_h}_{L^2}^2 - 4(\tu^{m+1}_h,\Grad p^m_h)
	\end{align*}
	where we also used~\eqref{eq:projection2} in the last line. Plugging this into~\eqref{eq:upartialenergy}, we obtain
	\begin{multline}
		\label{eq:upartialenergy2}
		\frac{1}{\Delta t}\bigg( \norm{u^{m+1}_h}_{L^2}^2 - \norm{u^m_h}_{L^2}^2 +\norm{2u^{m+1}_h-u^m_h}_{L^2}^2-\norm{2 u^m_h - u^{m-1}_h}_{L^2}^2\\
		+\norm{u^{m+1}_h-2u^m_h+ u^{m-1}_h}_{L^2}^2 + 3\norm{\tu^{m+1}_h-u^{m+1}_h}_{L^2}^2\bigg)\\
		+\frac{4\Delta t}{3}\left(\norm{\Grad p^{m+1}_h}_{L^2}^2-\norm{\Grad p^{m}_h}_{L^2}^2\right) \\
		+ 4 \mu(\Grad\tu^{m+1}_h,\Grad\tu^{m+1}_h)+4(\sigma^{m+1}_h,\Grad\tu^{m+1}_h)+4(\HH^{m+1}_h\Grad\hQ^{m+1}_h,\tu^{m+1}_h)\\
		=4(\PUh f^{m+1},\tu^{m+1}_h).
	\end{multline}
	Next, we take $Y = 4\HH^{m+1}_h$ as a test function in~\eqref{eq:Qdisc} and $Z = 2\frac{3Q^{m+1}_h - 4 Q^m_h+Q^{m-1}_h}{\Delta t}$ as a test function in~\eqref{eq:Hdisc} and subtract the first from the second resulting identity:
	\begin{multline}
		\label{eq:Qpartialenergy}
		\frac{2 L}{\Delta t}(\Grad Q^{m+1}_h,3 \Grad Q^{m+1}_h-4 \Grad Q^m_h +\Grad Q^{m-1}_h)+4(s_h^{m+1},\HH^{m+1}_h)+ 4 M \norm{\HH^{m+1}_h}^2_{L^2}\\
		=4((\tu^{m+1}_h\cdot\Grad)\hQ^{m+1}_h,\HH^{m+1}_h) -\frac{2}{\Delta t}(r^{m+1}_h P(\hQ^{m+1}_h),3Q^{m+1}_h - 4 Q^m_h+Q^{m-1}_h)_h.
			\end{multline} 
			Testing~\eqref{eq:rdisc} with $w=\frac{2 r^{m+1}_h}{\Delta t}$, we have
			\begin{equation}\label{eq:rpartialenergy}
				\frac{2}{\Delta t}(3r^{m+1}_h -4r^m_h+r^{m-1}_h,r^{m+1}_h)_h = \frac{2}{\Delta t}(P(\hQ^{m+1}_h):(3 Q^{m+1}_h-4 Q^m_h+Q^{m-1}_h),r^{m+1}_h)_h.
			\end{equation}
	We use, along similar lines as above, that
	\begin{multline*}
		2(\Grad Q^{m+1}_h,3 \Grad Q^{m+1}_h-4 \Grad Q^m_h +\Grad Q^{m-1}_h)\\
		 = \norm{\Grad Q^{m+1}_h}_{L^2}^2 -\norm{\Grad Q^m_h}_{L^2}^2 + \norm{2 \Grad Q^{m+1}_h-\Grad Q^m_h}_{L^2}^2- \norm{2 \Grad Q^m_h-\Grad Q^{m-1}_h}_{L^2}^2+ \norm{\Grad Q_h^{m+1}-2\Grad Q^m_h+\Grad Q^{m-1}_h}_{L^2}^2.
	\end{multline*}
	and 
	\begin{multline*}
			2(r^{m+1}_h,3r^{m+1}_h-4 r^m_h +r^{m-1}_h)_h\\
		= \norm{r^{m+1}_h}_{h}^2 -\norm{r^m_h}_{h}^2 + \norm{2 r^{m+1}_h-r^m_h}_{h}^2- \norm{2r^m_h-r^{m-1}_h}_{h}^2+ \norm{r_h^{m+1}-2  r^m_h+r^{m-1}_h}_{h}^2,
	\end{multline*}
	which we use in~\eqref{eq:Qpartialenergy} and~\eqref{eq:rpartialenergy} and add the two equations, 
	\begin{multline}
		\label{eq:Qrpartialenergy}
		\frac{L}{\Delta t}\bigg(\norm{\Grad Q^{m+1}_h}_{L^2}^2 -\norm{\Grad Q^m_h}_{L^2}^2 + \norm{2 \Grad Q^{m+1}_h-\Grad Q^m_h}_{L^2}^2\\
		- \norm{2 \Grad Q^m_h-\Grad Q^{m-1}_h}_{L^2}^2+ \norm{\Grad Q_h^{m+1}-2\Grad Q^m_h+\Grad Q^{m-1}_h}_{L^2}^2
		\bigg)\\
		+ \frac{1}{\Delta t}\left(\norm{r^{m+1}_h}_{h}^2 -\norm{r^m_h}_{h}^2 + \norm{2 r^{m+1}_h-r^m_h}_{h}^2- \norm{2r^m_h-r^{m-1}_h}_{h}^2+ \norm{r_h^{m+1}-2  r^m_h+r^{m-1}_h}_{h}^2\right)\\
		= -4 (s^{m+1}_h,\HH^{m+1}_h)- 4 M \norm{\HH^{m+1}_h}_{L^2}^2+ 4 ((\tu^{m+1}_h\cdot\Grad)\hQ^{m+1}_h,\HH^{m+1}_h).
	\end{multline}
	Combining this last identity with~\eqref{eq:upartialenergy2} and using Lemma~\ref{lem:Zhaoetalidentity}, we obtain
	\begin{multline}
		\label{eq:energypertimestep}
			\frac{1}{\Delta t}\bigg( \norm{u^{m+1}_h}_{L^2}^2 - \norm{u^m_h}_{L^2}^2 +\norm{2u^{m+1}_h-u^m_h}_{L^2}^2-\norm{2 u^m_h - u^{m-1}_h}_{L^2}^2\\
		+\norm{u^{m+1}_h-2u^m_h+ u^{m-1}_h}_{L^2}^2 + 3\norm{\tu^{m+1}_h-u^{m+1}_h}_{L^2}^2\bigg)\\
		+\frac{4\Delta t}{3}\left(\norm{\Grad p^{m+1}_h}_{L^2}^2-\norm{\Grad p^{m}_h}_{L^2}^2\right) 
		+ 4 \mu(\Grad\tu^{m+1}_h,\Grad\tu^{m+1}_h)  \\
		+ 	\frac{L}{\Delta t}\bigg(\norm{\Grad Q^{m+1}_h}_{L^2}^2 -\norm{\Grad Q^m_h}_{L^2}^2 + \norm{2 \Grad Q^{m+1}_h-\Grad Q^m_h}_{L^2}^2\\
		- \norm{2 \Grad Q^m_h-\Grad Q^{m-1}_h}_{L^2}^2+ \norm{\Grad Q_h^{m+1}-2\Grad Q^m_h+\Grad Q^{m-1}_h}_{L^2}^2
		\bigg)\\
		+ \frac{1}{\Delta t}\left(\norm{r^{m+1}_h}_{h}^2 -\norm{r^m_h}_{h}^2 + \norm{2 r^{m+1}_h-r^m_h}_{h}^2- \norm{2r^m_h-r^{m-1}_h}_{h}^2+ \norm{r_h^{m+1}-2  r^m_h+r^{m-1}_h}_{h}^2\right)\\
		=  - 4 M \norm{\HH^{m+1}_h}_{L^2}^2 +4(\PUh f^{m+1},\tu^{m+1}_h).
	\end{multline}
	Now we multiply by $\Delta t$ and sum over $m = 1,\dots, N-1$ to obtain
	\begin{multline}\label{eq:prelimenergy}
		E^N_h + \sum_{m=1}^{N-1}\norm{u^{m+1}_h-2u^m_h+ u^{m-1}_h}_{L^2}^2+3 \sum_{m=1}^{N-1}\norm{\tu^{m+1}_h-u^{m+1}_h}_{L^2}^2 + 4\mu\Delta t \sum_{m=1}^{N-1}\norm{\Grad\tu_h^{m+1}}_{L^2}^2\\
		+ 4M\Delta t \sum_{m=1}^{N-1}\norm{\HH_h^{m+1}}_{L^2}^2 + L\sum_{m=1}^{N-1}\norm{\Grad Q_h^{m+1}-2\Grad Q^m_h+\Grad Q^{m-1}_h}_{L^2}^2+\sum_{m=1}^{N-1}\norm{r_h^{m+1}-2r^m_h+r^{m-1}_h}_{h}^2\\
		= E_h^1 +4 \Delta t \sum_{m=1}^{N-1}(f^{m+1},\tu_h^{m+1}).
	\end{multline}
	In order to derive the final estimate from this, we need to estimate $E_h^1$ which is given by
	\begin{multline*}
		E^1_h =  \norm{u^{1}_h}_{L^2}^2+\norm{2 u^1_h - u^{0}_h}_{L^2}^2+L\norm{\Grad Q_h^{1}}_{L^2}^2+L\norm{2\Grad Q_h^1-\Grad Q_h^{0}}^2_{L^2}\\
		+	 \norm{r_h^{1}}_h^2+\norm{2r_h^1-r_h^{0}}_h^2+\frac{4\Delta t^2}{3}\norm{\Grad p^1_h}_{L^2}^2\\
		\leq 7\left(\norm{u^1_h}_{L^2}^2+L\norm{\Grad Q_h^{1}}_{L^2}^2 + \norm{r_h^{1}}_h^2 \right)  + 3\left(\norm{u^0_h}_{L^2}^2 + L\norm{\Grad Q_h^{0}}_{L^2}^2+\norm{r_h^{0}}_h^2 \right) +\frac{4\Delta t^2}{3}\norm{\Grad p^1_h}_{L^2}^2.
	\end{multline*}
		Thus we need an estimate on the $L^2$-norms of $u^1_h, u^0_h$, $\Grad Q^1_h$, $r_h^1$ and $\Delta t \Grad p^1_h$ which can be obtained from analyzing the first and the zeroth time step, that were computed by using a backward Euler method, cf. Remark~\ref{rem:1ststep}, and by projection~\eqref{eq:zerothstep}. We start with the zeroth step: Taking $v=\Delta t (u^0_h+\Delta t \Grad p_h^0+\tu_h^0)$ as a test function in the first equation in~\eqref{eq:zerothstep} and using that $u^0_h$ satisfies the weak divergence constraint, we find
	\begin{equation*}
		0=(u^0_h-\tu^0_h+\Delta t \Grad p_h^0,u^0_h+\tu^0_h +\Delta t \Grad p_h^0)  =\norm{u_h^0}_{L^2}^2+\Delta t^2 \norm{\Grad p_h^0}_{L^2}^2 -\norm{\tu_h^0}_{L^2}^2
	\end{equation*}
	which implies that
	\begin{equation}\label{eq:u0estimate}
		\norm{u^0_h}_{L^2}^2+\Delta t^2\norm{\Grad p_h^0}_{L^2}^2=\norm{\tu^0_h}_{L^2}^2 = \norm{\PUh u_0}_{L^2}^2 \leq \norm{u_0}_{L^2}^2 \leq C.
	\end{equation}
		Next, we obtain a uniform in $h$ and $\Delta t$ bound on $u^1_h$ and $\Delta t \Grad p^1_h$ by using the first step~\eqref{eq:step1fully1step}--\eqref{eq:projection2step1}.
	This is done in a very similar way. Taking $v=2\tu^1_h$ as a test function in~\eqref{eq:step1fully1step}, we have
	\begin{multline}\label{eq:firststepenergy}
		\frac{1}{\Delta t}\left(\norm{\tu^1_h}_{L^2}^2 + \norm{\tu^1_h-u^0_h}_{L^2}^2 - \norm{u^0_h}_{L^2}^2\right) - 2(p^0_h,\Div \tu^1_h)\\
		+2\mu \norm{\Grad \tu^1_h}_{L^2}^2  + 2(\sigma_h^1,\Grad \tu^1_h)+2(\HH^1_h\Grad Q^0_h,\tu_h^1)= 2(f^1,\tu^1_h).
	\end{multline}
	Then, taking $v=\frac{u^1_h+\tu^1_h}{\Delta t}+ \Grad(p^1_h + p^0_h)$ as a test function in~\eqref{eq:projection1step1} and using that $u^1_h$ is weakly divergence free, i.e., satisfies~\eqref{eq:projection2step1}, we have
	\begin{equation*}
		0 = \frac{1}{\Delta t^2}\left(\norm{u^1_h}_{L^2}^2 - \norm{\tu^1_h}_{L^2}^2\right)+\norm{\Grad p^1_h}_{L^2}^2 -\norm{\Grad p^0_h}_{L^2}^2 -\frac{2}{\Delta t}(\tu^1_h,\Grad p^0_h).
	\end{equation*}
	Plugging this into~\eqref{eq:firststepenergy} (after multiplying by $\Delta t$), we obtain
	\begin{multline}\label{eq:firststepenergy2}
		\frac{1}{\Delta t}\left(\norm{u^1_h}_{L^2}^2 + \norm{\tu^1_h-u^0_h}_{L^2}^2 - \norm{u^0_h}_{L^2}^2\right) +\Delta t\left(\norm{\Grad p^1_h}_{L^2}^2 -\norm{\Grad p^0_h}_{L^2}^2\right) \\
		+2\mu \norm{\Grad \tu^1_h}_{L^2}^2 + 2(\sigma_h^1,\Grad \tu^1_h)+2(\HH^1_h\Grad Q^0_h,\tu_h^1) 
		= 2(f^1,\tu^1_h).
	\end{multline}
	Next, we take $Y = 2 \HH^{1}_h$ as a test function in~\eqref{eq:Qfirststep}, $Z=2\frac{Q^1_h-Q^0_h}{\Delta t}$ as a test function in~\eqref{eq:Hfirststep} and $w=\frac{2}{\Delta t}r^1_h$ as a test function in~\eqref{eq:rfirststep}, then add the resulting~\eqref{eq:rfirststep} and~\eqref{eq:Hfirststep} and subtract~\eqref{eq:Qfirststep} to obtain
	\begin{multline}
		\label{eq:firststepenergyrQH}
		\frac{1}{\Delta t}\left(\norm{r_h^1}_h^2 - \norm{r_h^0}_h^2+\norm{r_h^1 - r_h^0}_h^2\right)+2M\norm{\HH^1_h}_{L^2}^2\\
		+\frac{L}{\Delta t}\left(\norm{\Grad Q_h^1}_{L^2}^2 -\norm{\Grad Q_h^0}_{L^2}^2+\norm{\Grad Q_h^1-\Grad Q_h^0}_{L^2}^2\right)= 2 ((\tu^1_h\cdot\Grad) Q_h^0,\HH_h^1)-2(s^1_h,\HH_h^1).
	\end{multline}
	Adding~\eqref{eq:firststepenergyrQH} and~\eqref{eq:firststepenergy2}, we obtain
	\begin{multline}
		\label{eq:totalenergyfirststep}
			\frac{1}{\Delta t}\left(\norm{u^1_h}_{L^2}^2 + \norm{\tu^1_h-u^0_h}_{L^2}^2 - \norm{u^0_h}_{L^2}^2\right) +\Delta t\left(\norm{\Grad p^1_h}_{L^2}^2 -\norm{\Grad p^0_h}_{L^2}^2\right)+2\mu \norm{\Grad \tu^1_h}_{L^2}^2 \\
		 +
			\frac{1}{\Delta t}\left(\norm{r_h^1}_h^2 - \norm{r_h^0}_h^2+\norm{r_h^1 - r_h^0}_h^2\right)+2M\norm{\HH^1_h}_{L^2}^2+\frac{L}{\Delta t}\left(\norm{\Grad Q_h^1}_{L^2}^2-\norm{\Grad Q_h^0}_{L^2}^2+\norm{\Grad (Q_h^1-Q^0_h)}_{L^2}^2 \right)\\
		= 2(f^1,\tu^1_h).
	\end{multline}
	We add this $7\Delta t$ times to~\eqref{eq:prelimenergy} to obtain
	\begin{multline}
		\label{eq:almostdone}
			E^N_h + \sum_{m=1}^{N-1}\norm{u^{m+1}_h-2u^m_h+ u^{m-1}_h}_{L^2}^2+3 \sum_{m=1}^{N-1}\norm{\tu^{m+1}_h-u^{m+1}_h}_{L^2}^2 + 4\mu\Delta t \sum_{m=0}^{N-1}\norm{\Grad\tu_h^{m+1}}_{L^2}^2\\
		+ 4M\Delta t \sum_{m=0}^{N-1}\norm{\HH_h^{m+1}}_{L^2}^2 + L\sum_{m=1}^{N-1}\norm{\Grad Q_h^{m+1}-2\Grad Q^m_h+\Grad Q^{m-1}_h}_{L^2}^2+\sum_{m=1}^{N-1}\norm{r_h^{m+1}-2r^m_h+r^{m-1}_h}_{h}^2\\
		+ 7\norm{\tu^1_h-u^0_h}_{L^2}^2+7\norm{r_h^1 - r_h^0}_h^2 +7 L\norm{\Grad (Q^1_h-Q_h^0)}_{L^2}^2\\
		\leq 4 \Delta t \sum_{m=0}^{N-1}(f^{m+1},\tu_h^{m+1}) 
	  +10 \norm{u^0_h}_{L^2}^2+7\Delta t^2\norm{\Grad p^0_h}_{L^2}^2 
+10\norm{r_h^0}_h^2+10 L\norm{\Grad Q_h^0}_{L^2}^2 
	+ 10\Delta t (f^1,\tu^1_h)\\
		\leq 4 \Delta t \sum_{m=0}^{N-1}(f^{m+1},\tu_h^{m+1})
 +10 \norm{u_0}_{L^2}^2
	+10\norm{r_h^0}_h^2 +10 L\norm{\Grad Q_h^0}_{L^2}^2  
	+ 10\Delta t (f^1,\tu^1_h).
	\end{multline}
 	It remains to estimate the source term involving $f$ which can be done using a discrete version of Gr\"onwall's inequality, Lemma~\ref{lem:discretegronwall}.
	We first estimate with Young's inequality
	\begin{equation*}
		\left|(f^{m+1},\tu_h^{m+1})\right|\leq \frac{1}{2}\left(\norm{f^{m+1}}_{L^2}^2+\norm{\tu_h^{m+1}}_{L^2}^2\right),
	\end{equation*}
	and then use $v=u^{m+1}_h$ as a test function in~\eqref{eq:projection1} which yields with the weak divergence constraint~\eqref{eq:projection2},
	\begin{equation}\label{eq:utildebound}
		0= (u^{m+1}_h - \tu^{m+1}_h,u^{m+1}_h) = \frac{1}{2}\norm{u^{m+1}_h}_{L^2}^2 -\frac{1}{2}\norm{\tu^{m+1}_h}_{L^2}^2 + \frac12\norm{u^{m+1}_h - \tu^{m+1}_h}_{L^2}^2.
	\end{equation}
	Using this in the previous estimate, we obtain,
	\begin{equation*}
		\left|(f^{m+1},\tu_h^{m+1})\right|\leq \frac{1}{2}\left(\norm{f^{m+1}}_{L^2}^2+\norm{u_h^{m+1}}_{L^2}^2+ \norm{u^{m+1}_h-\tu^{m+1}_h}_{L^2}^2\right).
	\end{equation*}
	For the $m=0$ term, we estimate instead using the triangle inequality,
	\begin{equation*}
		\left|(f^{1},\tu_h^{1})\right|\leq \frac{1}{2}\left(\norm{f^{1}}_{L^2}^2+\norm{\tu_h^{1}}_{L^2}^2\right)\leq \frac{1}{2}\norm{f^{1}}_{L^2}^2+\norm{u_h^{0}}_{L^2}^2 +\norm{u^0_h-\tu_h^1}_{L^2}^2.
	\end{equation*}
	Thus, we can estimate in~\eqref{eq:almostdone}
	\begin{multline}
		E^N_h + \sum_{m=1}^{N-1}\norm{u^{m+1}_h-2u^m_h+ u^{m-1}_h}_{L^2}^2+(3-2\Delta t) \sum_{m=1}^{N-1}\norm{\tu^{m+1}_h-u^{m+1}_h}_{L^2}^2 + 4\mu\Delta t \sum_{m=0}^{N-1}\norm{\Grad\tu_h^{m+1}}_{L^2}^2\\
	+ 4M\Delta t \sum_{m=0}^{N-1}\norm{\HH_h^{m+1}}_{L^2}^2 + L\sum_{m=1}^{N-1}\norm{\Grad Q_h^{m+1}-2\Grad Q^m_h+\Grad Q^{m-1}_h}_{L^2}^2+\sum_{m=1}^{N-1}\norm{r_h^{m+1}-2r^m_h+r^{m-1}_h}_{h}^2\\
	+ 7(1-2\Delta t)\norm{\tu^1_h-u^0_h}_{L^2}^2+7\norm{r_h^1 - r_h^0}_h^2 +7 L\norm{\Grad (Q^1_h-Q_h^0)}_{L^2}^2\\
	\leq 2\Delta t \sum_{m=1}^{N-1}(\norm{f^{m+1}}^2_{L^2}+\norm{u_h^{m+1}}_{L^2}^2) 
	+10 \norm{u_0}_{L^2}^2
	+10\norm{r_h^0}_h^2 +10 L\norm{\Grad Q_h^0}_{L^2}^2  
	+ 7\Delta t (\norm{f^1}_{L^2}^2+2\norm{u_0}_{L^2}^2)\\
		\leq 2\Delta t \sum_{m=2}^{N}(\norm{f^{m}}^2_{L^2}+E^m_h) 
	+10 \norm{u_0}_{L^2}^2
	+10\norm{r_h^0}_h^2 +10 L\norm{\Grad Q_h^0}_{L^2}^2  
	+ 7\Delta t (\norm{f^1}_{L^2}^2+2\norm{u_0}_{L^2}^2).
	\end{multline}
	Now we use the discrete Gr\"onwall inequality, Lemma~\ref{lem:discretegronwall}, with $\nu = 2$ and
	\begin{equation*}
		a_n = E^n_h,\quad b_n = 2\Delta t \sum_{j=2}^{n}\norm{f^{j}}^2_{L^2}
		+10 \norm{u_0}_{L^2}^2
		+10\norm{r_h^0}_h^2 +10 L\norm{\Grad Q_h^0}_{L^2}^2  
		+ 7\Delta t (\norm{f^1}_{L^2}^2+2\norm{u_0}_{L^2}^2)
	\end{equation*}
	to obtain
		\begin{multline}
		E^N_h + \sum_{m=1}^{N-1}\norm{u^{m+1}_h-2u^m_h+ u^{m-1}_h}_{L^2}^2+(3-2\Delta t) \sum_{m=1}^{N-1}\norm{\tu^{m+1}_h-u^{m+1}_h}_{L^2}^2 + 4\mu\Delta t \sum_{m=0}^{N-1}\norm{\Grad\tu_h^{m+1}}_{L^2}^2\\
		+ 4M\Delta t \sum_{m=0}^{N-1}\norm{\HH_h^{m+1}}_{L^2}^2 + L\sum_{m=1}^{N-1}\norm{\Grad Q_h^{m+1}-2\Grad Q^m_h+\Grad Q^{m-1}_h}_{L^2}^2+\sum_{m=1}^{N-1}\norm{r_h^{m+1}-2r^m_h+r^{m-1}_h}_{h}^2\\
		+ 7(1-2\Delta t)\norm{\tu^1_h-u^0_h}_{L^2}^2+7\norm{r_h^1 - r_h^0}_h^2 +7 L\norm{\Grad (Q^1_h-Q_h^0)}_{L^2}^2\\
		\leq \left(2\Delta t \sum_{m=2}^{N}\norm{f^{m}}^2_{L^2}
		+10 \norm{u_0}_{L^2}^2
		+10\norm{r_h^0}_h^2 +10 L\norm{\Grad Q_h^0}_{L^2}^2  
		+ 7\Delta t (\norm{f^1}_{L^2}^2+2\norm{u_0}_{L^2}^2)\right)(1-2\Delta t)^{-N}.
	\end{multline}
	For $\Delta t>0$ small enough, we have $(1-2\Delta t)^{-1}\leq 1 + 3 \Delta t$ and $2\Delta t <1$ and hence, we can bound the right-hand side (using also that $1+x\leq \exp(x)$) by
	\begin{multline}
			E^N_h + \sum_{m=1}^{N-1}\norm{u^{m+1}_h-2u^m_h+ u^{m-1}_h}_{L^2}^2+(3-2\Delta t) \sum_{m=1}^{N-1}\norm{\tu^{m+1}_h-u^{m+1}_h}_{L^2}^2 + 4\mu\Delta t \sum_{m=0}^{N-1}\norm{\Grad\tu_h^{m+1}}_{L^2}^2\\
		+ 4M\Delta t \sum_{m=0}^{N-1}\norm{\HH_h^{m+1}}_{L^2}^2 + L\sum_{m=1}^{N-1}\norm{\Grad Q_h^{m+1}-2\Grad Q^m_h+\Grad Q^{m-1}_h}_{L^2}^2+\sum_{m=1}^{N-1}\norm{r_h^{m+1}-2r^m_h+r^{m-1}_h}_{h}^2\\
		+ 7(1-2\Delta t)\norm{\tu^1_h-u^0_h}_{L^2}^2+7\norm{r_h^1 - r_h^0}_h^2 +7 L\norm{\Grad (Q^1_h-Q_h^0)}_{L^2}^2\\
		\leq C\left(1+ \norm{f}^2_{L^2([0,T]\times\dom)}
		+\norm{u_0}_{L^2}^2
		+\norm{Q_0}_{L^4}^4 +\norm{\Grad Q_0}_{L^2}^2  
	\right)\exp(3N \Delta t)
	\end{multline}
	which proves the result.
\end{proof}

\section{Convergence analysis of the scheme}\label{sec:convergence}

Next, we will show that interpolations of the approximations defined by the scheme~\eqref{eq:step1fully}--\eqref{eq:projectionfullydiscrete} converge up to a subsequence to a weak solution of~\eqref{seq:BerisEdwards}. For this purpose, we define the  piecewise constant interpolants in time:
\begin{alignat}{2}
u_h(t)& = u_h^{m+1},\qquad &t\in (t^m,t^{m+1}],\label{eq:defuh}\\
\bar{u}_h(t) &= 2u_h^m-u^{m-1}_h,\qquad &t\in (t^m,t^{m+1}],\\
\widetilde{u}_h(t) &= \widetilde{u}_h^{m+1},\qquad &t\in (t^m,t^{m+1}],\\
\widehat{u}_h(t) &= 2\tu_h^m-\tu^{m-1}_h,\qquad &t\in (t^m,t^{m+1}],\\
Q_h(t)& = Q_h^{m+1},\qquad &t\in (t^m,t^{m+1}],\\
\bar{Q}_h(t) &=2 Q_h^m-Q^{m-1}_h,\qquad &t\in (t^m,t^{m+1}],\\
r_h(t) &= r_h^{m+1},\qquad &t \in (t^m,t^{m+1}],\\
\mathcal{H}_h(t)& =\mathcal{H}_h^{m+1},\qquad &t\in (t^m,t^{m+1}],\\
p_h(t)& = p_h^{m+1},\qquad & t\in (t^m,t^{m+1}],\label{eq:defph}\\
f_h(t)& = \frac{1}{\Delta t}\int_{t^{m+1/2}}^{t^{m+3/2}}\PUh f(\tau)d\tau,\qquad &t\in (t^m,t^{m+1}],
\end{alignat}
for $m=1,2,\dots$ and for $m=0$,
\begin{align*}
	u_h(t)= u^1_h,&\quad t\in (t^0,t^1],\\
	\tu_h(t)= \tu^1_h,&\quad t\in (t^0,t^1],\\
	\hu_h(t) = \tu_h^0,&\quad t\in (t^0,t^1],\\
		\bar{u}_h(t)= u_h^0,&\quad t\in (t^0,t^1],\\
		Q_h(t) = Q_h^1,&\quad t\in (t^0,t^1],\\
				r_h(t) = r_h^1,&\quad t\in (t^0,t^1],\\
		\bar{Q}_h(t)=Q_h^0,&\quad t\in (t^0,t^1],\\
		\HH_h(t)=\HH_h^1,&\quad t\in (t^0,t^1],\\	
	p_h(t)=p^1_h,&\quad   t\in (t^0,t^1],\\
	f_h(t) = \frac{1}{\Delta t}\int_{t^{1/2}}^{t^{3/2}}\PUh f(\tau)d\tau,&\quad t\in (t^0,t^{1}],
\end{align*}
and
\begin{equation*}
\tu_h(\tau)=\hu_h(\tau)=\tu^0_h,\quad u_h(\tau)=\bar{u}_h(\tau)=u^0_h,\quad  r_h(\tau)=r^0_h,\quad Q_h(\tau)=Q^0_h,\quad p_h(\tau)=p_h^0,\quad \tau\leq 0
\end{equation*} 
From the energy inequality in the last section, Lemma~\ref{lem:discenergyestimate}, we obtain the following uniform a priori estimates for the sequences $\{u_h\}_{h>0}$, $\{\bar{u}_h\}_{h>0}$, etc.:
\begin{align*}
&\{u_h\}_{h>0},\{\bar{u}_h\}_{h>0},\{\hu_h\}_{h>0}\subset L^\infty(0,T;L^2(\dom)),\\
&\{\widetilde{u}_h\}_{h>0}\subset L^\infty(0,T;L^2(\dom))\cap L^2(0,T;H^1_0(\dom)),\\
& \{Q_h\}_{h>0}, \{\bar{Q}_h\}_{h>0}\subset L^\infty(0,T;H^1_0(\dom)),\\
&\{\mathcal{H}_h\}_{h>0}\subset L^2([0,T]\times \dom),\\
& \{r_h\}_{h>0}\subset L^\infty(0,T;L^2(\dom))
\end{align*}
These uniform estimates imply, using the Banach--Alaoglu theorem that there exist weakly convergent subsequences, which, for the ease of notation, we still denote by $h\to 0$,
\begin{align}
u_h\weakstar u,\quad \bar{u}_h\weakstar \bar{u},\quad\hu_h\weakstar \hu,&\quad  \text{in }\, L^\infty(0,T;L^2(\dom)),\label{eq:uhweakconv}\\
\widetilde{u}_h\weakstar \widetilde{u},&\quad \text{in }\, L^\infty(0,T;L^2(\dom))\cap L^2(0,T;H^1_0(\dom)),\\
Q_h\weakstar Q,\quad\bar{Q}_h\weakstar \bar{Q},&\quad \text{in }\, L^\infty(0,T;H^1_0(\dom)),\\
r_h\weakstar r,&\quad \text{in }\, L^\infty(0,T;L^2(\dom)),\\
\mathcal{H}_h\weak \mathcal{H},&\quad \text{in }\, L^2([0,T]\times\dom).\label{eq:Hhweakconv}
\end{align}
Due to the nonlinearities in system~\eqref{seq:BerisEdwards}, in order to prove convergence of the scheme, we will need to derive strong convergence of the sequences approximating $u$ and $Q$ as well as $\Grad Q$ in $L^2([0,T]\times\dom)$. To derive strong convergence of the first two, we will use an idea from~\cite{Doering1994,Berselli2021} where weak continuity in time was combined with $H^1$-spatial regularity to obtain precompactness in $L^2([0,T]\times\dom)$.
However, the approximations in~\cite{Doering1994,Berselli2021} are uniformly weakly continuous, so we still need to verify that the numerical error does not obstruct things. Furthermore, we will prove the approximate equicontinuity for the sequence $\{u_h\}_{h>0}$ whereas the spatial regularity is available for the sequence $\{\tu_h\}_{h>0}$, so we need to combine the two via the following lemma, where we show that the sequences $u_h$, $\bar{u}_h$, $\hu_h$ and $\tu_h$, and $Q_h$, $\bar{Q}_h$ respectively, have the same limits. The proof is similar to Lemma 4.1 in~\cite{Weber2025}:
\begin{lemma}
	\label{lem:samelimits}
	Assume that {$\Delta t = o_{h\to 0}(1)$}. Then, we have that $Q=\bar{Q}$ a.e. in $[0,T]\times \dom$ and
	\begin{equation}\label{eq:QQbarconv}
		\lim_{h\to 0}\norm{Q_h-\bar{Q}_h}_{L^2(0,T;H^1(\dom))} = 0,
	\end{equation}
	as well as $u=\bar{u}=\tu=\hu$ a.e. in $[0,T]\times\dom$ and in $L^2([0,T]\times\dom)$, and
	\begin{align}
		\label{eq:uubarconv}
		\lim_{h\to 0}\norm{u_h-\bar{u}_h}_{L^2([0,T]\times\dom)} &= 0,\\
	\label{eq:uhatubarconv}
\lim_{h\to 0}\norm{\hu_h-\bar{u}_h}_{L^2([0,T]\times\dom)} &= 0,\\
		\label{eq:utildeuconv}
		\lim_{h\to 0}\norm{{u}_h-\tu_h}_{L^2([0,T]\times\dom)}& = 0.
	\end{align}
	
\end{lemma} 
\begin{proof}
	We have
	\begin{align*}
		&\norm{\Grad Q_h-\Grad \bar{Q}_h}_{L^2([0,T]\times\dom)}^2\\
		&=\sum_{m=1}^{N-1}\int_{t^m}^{t^{m+1}}\!\!\int_{\dom}\left|\Grad Q^{m+1}_h-(2\Grad Q^m_h-\Grad Q^{m-1}_h) \right|^2dxdt + \int_{t^0}^{t^1}\int_{\dom} |\Grad Q^1_h -\Grad Q^0_h|^2 dx dt\\
		&\leq  \Delta t \sum_{m=1}^{N-1} \int_{\dom}\left|\Grad Q^{m+1}_h-2\Grad Q^m_h+\Grad Q^{m-1}_h \right|^2dx + \int_{t^0}^{t^1}\int_{\dom} |\Grad Q^1_h -\Grad Q^0_h|^2 dx dt\\
		&\leq C \Delta t,
	\end{align*}
	where we used the energy inequality, Lemma~\ref{lem:discenergyestimate}. Clearly, as $\Delta t, h\to 0$, this term vanishes. Hence, by the Poincar\'e inequality,~\eqref{eq:QQbarconv} holds. Using this, it then follows easily that for any test function $\varphi\in L^2([0,T]\times\dom)^{d\times d}$,
	\begin{multline*}
		\int_0^T\!\!\int_{\dom}\bar{Q}:\varphi\, dx dt = \lim_{h\to 0} \int_0^T\!\!\int_{\dom}\bar{Q}_h:\varphi\, dx dt\\
		= \lim_{h\to 0}\int_0^T\!\!\int_{\dom}(\bar{Q}_h-Q_h):\varphi\, dx dt+ \lim_{h\to 0}\int_0^T\!\!\int_{\dom}Q_h:\varphi \,dx dt =\int_0^T\!\!\int_{\dom}Q:\varphi\, dx dt,
	\end{multline*}
	hence $Q=\bar{Q}$ a.e. in $[0,T]\times\dom$. The estimates for the variables approximating $u$ follow in a similar way, see also~\cite[Lemma 4.1]{Weber2025}.
\end{proof}
As in~\cite[Lemma 4.2]{Weber2025}, it also follows that the weak limit of the sequence $\{u_h\}_{h>0}$ is weakly divergence free which implies that $u$ is divergence free a.e.\ since it lies in $L^2(0,T;H^1_0(\dom))$. Due to Lemma~\ref{lem:samelimits}, it is sufficient to prove strong convergence in $L^2$ of one of the approximations for each $u$ and $Q$ respectively, therefore we will show this for $\{Q_h\}_{h>0}$ and $\{\tu_h\}_{h>0}$.
We start by proving the approximate weak equicontinuity, for which we need the following technical lemma:
\begin{lemma}
	\label{lem:timecontinuityutilde}
	Assume that $h^\sigma\leq C {\Delta t}$ for some $2\leq  \sigma\leq k+1$, where $k$ is the polynomial degree of $\Uh$. Then the approximations $u_h$ computed by~\eqref{eq:step1fully}--\eqref{eq:projectionfullydiscrete} satisfy for any $0< \tau <T$,  any $\Delta t\leq t\leq T-\tau$ and $v\in C^\infty_{c,\Div}(\dom)$,
	\begin{equation}
		\label{eq:timecontutildeweak}
		\left| \left(3(u_h(t+\tau)-u_h(t))-(u_h(t+\tau-\Delta t)-u_h(t-\Delta t)),v\right)\right|\leq C_v( \sqrt{\Delta t}+h^\sigma + \sqrt{\tau}).
	\end{equation}
	The approximations $Q_h$ satisfy  
	\begin{equation}
		\label{eq:timecontQweak}
	\norm{3(Q_h(t+\tau)-Q_h(t))-(Q_h(t+\tau-\Delta t)-Q_h(t-\Delta t))}_{L^{6/5}(\dom)}\leq C   \left(\sqrt{\tau}+\sqrt{\Delta t}\right),
	\end{equation}
	and
	\begin{equation}\label{eq:timederQ}
	\norm{D_t^- Q_h}_{L^2(0,T;L^{6/5}(\dom))} \leq C,
	\end{equation}
	where $C$ does not depend on $h,\Delta t$, and
	where we denoted
	\begin{equation}\label{eq:timediffdef1}
		D_t^- Q_h(t): = \frac{3 Q_h(t)-4 Q_h(t-\Delta t)+ Q_h(t-2\Delta t)}{2\Delta t},
	\end{equation} 
	for $t>\Delta t$ and
	\begin{equation}\label{eq:timediffdef2}
		D_t^- Q_h(t) := \frac{Q_h(t)-Q_h(t-\Delta t)}{\Delta t},
	\end{equation}
	for $t\in [0,\Delta t]$. Finally, the approximations $r_h$ satisfy for any $w\in H^{2}(\dom)$,
	\begin{equation}\label{eq:rcontinuity}
		\left|(3(r_h(t+\tau)-r_h(t))-(r_h(t+\tau-\Delta t)-r_h(t-\Delta t)),w)\right|\leq C_w \left(\sqrt{\tau}+\sqrt{\Delta t}+ h\right).
	\end{equation}
\end{lemma}
\begin{proof}
	We let $\tau>0$ and $t\in [\Delta t,T-\tau]$.
	Then we let $m_1\in \N$ such that $u_h(t) = u_h^{m_1+1}$ and $m_2\in \N$ such that $u_h(t+\tau) = u_h^{m_2+1}$.  Since $t+\tau>\Delta t$, we must have $m_2\geq 1$. Since $t\geq \Delta t$, we have $m_1\geq  0$. We then write
	\begin{align*}
		3(u_h(t+\tau)-u_h(t))-(u_h(t+\tau-\Delta t)-u_h(t-\Delta t)) &= 3 u_h^{m_2+1}-u_h^{m_2}-3 u_h^{m_1+1}+ u_h^{m_1} \\
		&= \sum_{m=m_1+1}^{m_2}(3 u^{m+1}_h-4u_h^m+ u_h^{m-1}).
	\end{align*}
	And similarly,
	\begin{equation*}
		3(Q_h(t+\tau)-Q_h(t))-(Q_h(t+\tau-\Delta t)-Q_h(t-\Delta t)) = \sum_{m=m_1+1}^{m_2}(3 Q^{m+1}_h-4 Q_h^m+ Q_h^{m-1}),
	\end{equation*}
	and
		\begin{equation*}
		3(r_h(t+\tau)-r_h(t))-(r_h(t+\tau-\Delta t)-r_h(t-\Delta t)) = \sum_{m=m_1+1}^{m_2}(3 r^{m+1}_h-4 r_h^m+ r_h^{m-1}).
	\end{equation*}
	Thus, we can take a test function $v\in C^\infty_{c,\Div}(\dom)$ and plug in the scheme~\eqref{eq:step1fully}--\eqref{eq:projectionfullydiscrete},
	\begin{align*}
		&(3(u_h(t+\tau)-u_h(t))-(u_h(t+\tau-\Delta t)-u_h(t-\Delta t)) ,v) \\
		& = (3(u_h(t+\tau)-u_h(t))-(u_h(t+\tau-\Delta t)-u_h(t-\Delta t)) ,\Ih^u v)\\ 
		&\qquad +(3(u_h(t+\tau)-u_h(t))-(u_h(t+\tau-\Delta t)-u_h(t-\Delta t)) ,v-\Ih^u v) \\
		&\quad  =  \sum_{m=m_1+1}^{m_2}(3 u^{m+1}_h-4u_h^m+ u_h^{m-1},\Ih^u v)\\
		&\qquad +(3(u_h(t+\tau)-u_h(t))-(u_h(t+\tau-\Delta t)-u_h(t-\Delta t)) ,v-\Ih^u v) \\
		&\quad  = -2\Delta t   \sum_{m=m_1+1}^{m_2}\underbrace{b(\hu_h^{m+1},\tu_h^{m+1},\Ih^u v)}_{\text{I}} - 2\Delta t\sum_{m=m_1+1}^{m_2}\underbrace{\mu(\Grad \tu_h^{m+1},\Grad \Ih^u v)}_{\text{II}}  \\
		&\quad  \quad -2\Delta t\sum_{m=m_1+1}^{m_2} \underbrace{\left(\Grad p^{m+1}_h,\Ih^u v\right)}_{\text{III}}+2\Delta t \sum_{m=m_1+1}^{m_2}\underbrace{(f^{m+1},\Ih^u v)}_{\text{IV}}\\
		&\qquad \quad -2\Delta t \sum_{m=m_1+1}^{m_2}\underbrace{(\sigma_h^{m+1},\Grad \Ih^u v)}_{\text{V}} - 2\Delta t \sum_{m=m_1+1}^{m_2}\underbrace{(\HH_h^{m+1}\Grad \hQ^{m+1}_h, \Ih^u v)}_{\text{VI}} \\
		&\qquad \quad +\underbrace{(3(u_h(t+\tau)-u_h(t))-(u_h(t+\tau-\Delta t)-u_h(t-\Delta t)) ,v-\Ih^u v) }_{\text{VII}},
	\end{align*}
	where $\Ih^u$ is the interpolation operator on $\Uh$, as for example constructed in~\cite[Section 9.1]{Ern2021}.
	We estimate each of the terms I -- VII:
	For the first, we use the stability of the interpolation operator~\cite[Theorem 11.13]{Ern2021}:
	\begin{align*}
		|\text{I}|&=\left|b(\hu_h^{m+1},\tu_h^{m+1},\Ih^u v)\right|\\
		& = \left|\frac12\int_{\dom}\left((\hu^{m+1}_h\cdot\Grad )\tu_h^{m+1}\cdot \Ih^u v -(\hu_h^{m+1}\cdot\Grad)\Ih^u v \cdot\tu^{m+1}_h\right) dx \right|\\
		& \leq\frac12 \Big(\norm{\hu^{m+1}_h}_{L^2} \norm{\Grad\tu_h^{m+1}}_{L^2}\norm{\Ih^u v}_{L^\infty}+   \norm{\hu^{m+1}_h}_{L^2}\norm{\tu_h^{m+1}}_{L^2}\norm{\Grad \Ih v}_{L^\infty}\Big)\\
		& \leq C\Big(\norm{\hu^{m+1}_h}_{L^2} \norm{\Grad\tu_h^{m+1}}_{L^2}\norm{ v}_{L^\infty}+   \norm{\hu^{m+1}_h}_{L^2}\norm{\tu_h^{m+1}}_{L^2}\norm{v}_{W^{1,\infty}}\Big).
	\end{align*}
	For the second term, we have
	\begin{align*}
		|\text{II}|
		& \leq\mu \norm{\Grad \tu^{m+1}_h}_{L^2(\dom)}\norm{\Grad\Ih^u v}_{L^2(\dom)}\\
		& \leq C \norm{\Grad \tu^{m+1}_h}_{L^2(\dom)}\norm{ v}_{H^2(\dom)},
	\end{align*}
	using the stability of the interpolation operator $\Ih^u$.
	For the third term, III, we have, using that $v$ is divergence free, the approximation property of the interpolation operator~\cite[Theorem 11.13]{Ern2021}, and that by the energy estimate $\norm{\Grad p}_{L^\infty(0,T;L^2(\dom))}\leq C \Delta t^{-1}$,
	\begin{align}\label{eq:pressuretimecont}
		|\text{III}|&\leq C \norm{\Grad p^{m+1}_h}_{L^2(\dom)}\norm{\Ih^u v-v}_{L^2( \dom)}\\
		& \leq C\norm{\Grad p^{m+1}_h}_{L^2(\dom)}h^\sigma\norm{ v}_{H^\sigma( \dom)}\notag\\
		& \leq C \frac{h^\sigma}{\Delta t}\norm{ v}_{H^\sigma( \dom)}\notag\\
		& \leq  C \norm{ v}_{H^\sigma( \dom)},\notag
	\end{align}
	under the condition that $h^\sigma\leq C {\Delta t}$.
	For term IV, we have
	\begin{equation*}
		|\text{IV}|  \leq C\norm{f^{m+1}}_{ L^2(\dom)}\norm{v}_{ H^2(\dom)} .
	\end{equation*}
	For the fifth term, we estimate using the Sobolev inequality and the stability of the interpolation operator $\Ih^u$,
	\begin{align*}
		|\text{V}|& \leq C\left(\norm{\hQ^{m+1}_h}_{L^2}\norm{\HH^{m+1}_h}_{L^2}\norm{\Grad \Ih v}_{L^\infty}+ \norm{\hQ_h^{m+1}}_{L^4}^2 \norm{\HH_h^{m+1}}_{L^2} \norm{\Grad \Ih v}_{L^\infty}+ \norm{\HH^{m+1}_h}_{L^2}\norm{\Grad \Ih v}_{L^2}\right) \\
		& \leq C \left(\norm{\hQ_h^{m+1}}_{H^1}^2+1\right)\norm{\HH^{m+1}_h}_{L^2}\norm{ v}_{W^{1,\infty}}.
	\end{align*}
	For the sixth term, we have using the discrete energy estimate
	\begin{align*}
		|\text{VI}|& \leq \norm{\HH^{m+1}_h}_{L^2(\dom)}\norm{\Grad\hQ_h^{m+1}}_{L^2(\dom)} \norm{\Ih^u v}_{L^\infty}\\
		& \leq C \norm{\HH^{m+1}_h}_{L^2(\dom)}\norm{\hQ_h^{m+1}}_{H^1(\dom)} \norm{ v}_{L^\infty}.
	\end{align*}
	For term VII, we use the discrete energy estimate, Lemma~\ref{lem:discenergyestimate}, and interpolation estimates:
	\begin{equation*}
		|\text{VII}|\leq C \norm{u_h}_{L^\infty(0,T;L^2(\dom))} h^\sigma \norm{v}_{H^\sigma(\dom)}\leq C h^\sigma \norm{v}_{H^\sigma(\dom)}
	\end{equation*}
	
	Combining the estimates I -- VII, and using that $\sigma\geq 2$, $d\leq 3$, we have, 
	\begin{align*}
		&\left|\left(3(u_h(t+\tau)-u_h(t))-(u_h(t+\tau-\Delta t)-u_h(t-\Delta t)),v\right)\right|\\
		&\quad  \leq C \Delta t  \sum_{m=m_1+1}^{m_2} \Bigg( \norm{\hu^{m+1}_h}_{L^2} \norm{\tu_h^{m+1}}_{H^1}+ \norm{\Grad \tu^{m+1}_h}_{L^2(\dom)}\\
		&\qquad \hphantom{\leq C \Delta t  \sum_{m=m_1+1}^{m_2} \Big(}+ 1+ \norm{f^{m+1}}_{ L^2(\dom)}+ \left(\norm{\hQ_h^{m+1}}_{H^1}^2+1\right)\norm{\HH^{m+1}_h}_{L^2} \Bigg)\norm{v}_{H^\sigma\cap W^{1,\infty}}\\
		&\qquad + C h^\sigma \norm{v}_{H^\sigma}\\
		& \leq C \Bigg(\left(1+\norm{\tu_h}_{L^\infty(0,T;L^2(\dom))}\right)\Delta t\sum_{m=m_1+1}^{m_2} \norm{\Grad \tu_h^{m+1}}_{L^2} + C \Delta t (m_2-m_1)+Ch^\sigma\\
		&\qquad + C\Delta t \sum_{m=m_1+1}^{m_2}\norm{f^{m+1}}_{L^2} + C\left(1+\norm{Q_h}_{L^\infty(0,T;H^1(\dom))}^2\right)\Delta t\sum_{m=m_1+1}^{m_2}\norm{\HH^{m+1}_h}_{L^2}		\Bigg) \norm{v}_{H^\sigma\cap W^{1,\infty}}.
	\end{align*}
Using that $\Delta t (m_2-m_1) \leq \tau +\Delta t$, we can estimate this, using the discrete energy estimate:
			\begin{align*}
			&\left|\left(3(u_h(t+\tau)-u_h(t))-(u_h(t+\tau-\Delta t)-u_h(t-\Delta t)),v\right)\right|\\
			& \leq C \Bigg(\left(1+\norm{\tu_h}_{L^\infty(0,T;L^2(\dom))}\right)\sqrt{(m_2-m_1)\Delta t}\left(\Delta t\sum_{m=m_1+1}^{m_2} \norm{\Grad \tu_h^{m+1}}_{L^2}^2\right)^{1/2} \\
			&\qquad + C (\tau+\Delta t)+C h^\sigma\\
			&\qquad + C\sqrt{(m_2-m_1)\Delta t}\left(\Delta t \sum_{m=m_1+1}^{m_2}\norm{f^{m+1}}_{L^2}^2\right)^{1/2} \\
			&\qquad + C\left(1+\norm{Q_h}_{L^\infty(0,T;H^1(\dom))}^2\right)\sqrt{(m_2-m_1)\Delta t}\left(\Delta t\sum_{m=m_1+1}^{m_2}\norm{\HH^{m+1}_h}_{L^2}^2\right)^{1/2}		\Bigg) \norm{v}_{H^\sigma\cap W^{1,\infty}}\\
			& \leq C (\sqrt{\Delta t} + \sqrt{\tau}+h^\sigma) \norm{v}_{H^\sigma\cap W^{1,\infty}},
	\end{align*}
	which proves the estimate for $u_h$.

Next, we prove the estimate for the variable $Q_h$. This is similar, except that we use the $L^2$-orthogonal projection $\PMh:L^2(\dom)\to \Mh$ to project the test function onto the finite element space. The projection $\PMh$ is defined via
\begin{equation}\label{eq:L2projectionMh}
	(\PMh Y,Z) =(Y,Z),\quad \forall \, Z\in \Mh,
\end{equation} 
and satisfies
\begin{equation}
	\label{eq:L2stability}
	\norm{\PMh Y}_{L^r(\dom)}\leq C \norm{Y}_{L^r(\dom)},\quad \forall Y\in L^r(\dom).\quad r\in [1,\infty].
\end{equation}
This was proved for example in~\cite{Douglas1975} for quasi-uniform meshes, and see also~\cite{Diening2021} for a more recent result with fewer assumptions.
So we take a test function $Y\in L^6(\dom)$ and compute, using~\eqref{eq:L2projectionMh},
	\begin{align*}
	&(3(Q_h(t+\tau)-Q_h(t))-(Q_h(t+\tau-\Delta t)-Q_h(t-\Delta t)) ,Y) \\
	& \quad = (3(Q_h(t+\tau)-Q_h(t))-(Q_h(t+\tau-\Delta t)-Q_h(t-\Delta t)) ,\PMh Y)\\
	&\quad  =  \sum_{m=m_1+1}^{m_2}(3 Q^{m+1}_h-4Q_h^m+ Q_h^{m-1},\PMh Y)\\
	& = -2\Delta t \sum_{m=m_1+1}^{m_2}\underbrace{((\tu^{m+1}_h\cdot\Grad)\hQ_h^{m+1},\PMh Y)}_{\text{I}}+2\Delta t  \sum_{m=m_1+1}^{m_2}\underbrace{(s_h^{m+1},\PMh Y)}_{\text{II}}\\
	&\qquad\quad +2M\Delta t \sum_{m=m_1+1}^{m_2}\underbrace{(\HH_h^{m+1},\PMh Y)}_{\text{III}}.
\end{align*}
	We estimate the three terms:
		\begin{align}\label{eq:termeis}
	\begin{split}
			|\text{I}|& \leq C\norm{\tu_h^{m+1}}_{L^6}\norm{\Grad \hQ_h^{m+1}}_{L^2}\norm{\PMh Y}_{L^3}\\
		& \leq C \norm{\tu_h^{m+1}}_{H^1}\norm{Q_h}_{L^\infty(0,T;H^1(\dom))}\norm{Y}_{L^3},
	\end{split}
	\end{align}
	using~\eqref{eq:L2stability}.
	For the second term,
		\begin{align}\label{eq:termzwei}
	\begin{split}
			|\text{II}|& \leq C \left(\norm{\Grad\tu_h^{m+1}}_{L^2}\left(1+  \norm{\hQ_h^{m+1}}_{L^6}^2\right)\right)\norm{\PMh Y}_{L^6}\\
		&\leq C \norm{\Grad\tu_h^{m+1}}_{L^2}\left(1 + \norm{ Q_h}_{L^\infty(0,T;H^1(\dom))}^2\right)\norm{Y}_{L^6},
	\end{split}
	\end{align}
	again using~\eqref{eq:L2stability}.
	The third term, we estimate using~\eqref{eq:L2projectionMh},
		\begin{equation}\label{eq:termdrue}
		|\text{III}| =\left|\left(\HH_h^{m+1},Y\right)\right| \leq  \norm{\HH_h^{m+1}}_{L^2}\norm{ Y}_{L^2}.
	\end{equation}
	Thus, we obtain after combining these estimates
	\begin{align*}
		& \norm{3(Q_h(t+\tau)-Q_h(t))-(Q_h(t+\tau-\Delta t)-Q_h(t-\Delta t))}_{L^{6/5}(\dom)}\\
		& = \sup_{0\neq Y\in L^{6}(\dom)}\frac{\left|\left( 3(Q_h(t+\tau)-Q_h(t))-(Q_h(t+\tau-\Delta t)-Q_h(t-\Delta t)) ,Y \right)\right|}{\norm{Y}_{L^{6}(\dom)}}\\
	& \leq C \Delta t \sum_{m=m_1+1}^{m_2} \left(\norm{\tu_h^{m+1}}_{H^1}\norm{Q_h}_{L^\infty(0,T;H^1(\dom))} + \norm{\Grad \tu_h^{m+1}}_{L^2}\left(1+\norm{Q_h}_{L^\infty(0,T;H^1(\dom))}^2\right)+\norm{\HH^{m+1}_h}_{L^2} \right) \\
	&\leq C \left(\sqrt{\tau}+\sqrt{\Delta t}\right), 
\end{align*}
	which proves the second estimate. The estimate on the time difference $D_t^-Q_h$,~\eqref{eq:timederQ} is proved in a very similar way: We first observe
	\begin{align*}
		\norm{D_t^- Q_h}_{L^2(0,T;L^{6/5}(\dom))}^2 =\Delta t \sum_{m=0}^{N-1} \norm{D_t^- Q_h^{m+1}}_{L^{6/5}(\dom)}^2=\Delta t \sum_{m=0}^{N-1}  \left(\sup_{0\neq Y\in L^{6}(\dom)}\frac{\left|(D_t^- Q_h^{m+1},Y)\right|}{\norm{Y}_{L^6(\dom)}}\right)^2.
	\end{align*}
	We use the orthogonality property of the $L^2$-projection \eqref{eq:L2projectionMh} and plug in the scheme for $D_t^- Q_h^{m+1}$ and then use estimates~\eqref{eq:termeis},~\eqref{eq:termzwei} and~\eqref{eq:termdrue} at level $m+1$ for $m\geq 1$ (and almost identical estimates for $m=0$, plugging in~\eqref{eq:Qfirststep}), to obtain
	\begin{align*}
			&\norm{D_t^- Q_h}_{L^2(0,T;L^{6/5}(\dom))}^2 \\
			&\leq C \Delta t  \sum_{m=0}^{N-1} \left(\norm{\tu_h^{m+1}}_{H^1}^2\norm{Q_h}^2_{L^\infty(0,T;H^1(\dom))} + \norm{\Grad \tu_h^{m+1}}^2_{L^2}\left(1+\norm{Q_h}_{L^\infty(0,T;H^1(\dom))}^4\right)+\norm{\HH^{m+1}_h}^2_{L^2} \right)\\
			& \leq C,
	\end{align*}
	which is bounded uniformly in $h,\Delta t$ thanks to the discrete energy estimate, Lemma~\ref{lem:discenergyestimate}. We continue to proving the estimate for $r_h$. We have for $w\in H^2(\dom)$, using Lemmas~\ref{lem:masslumped} and~\ref{lem:masslumpedLp}, the Lipschitz continuity of $P$, Sobolev inequalities, and the energy estimate, Lemma~\ref{lem:discenergyestimate},
\begin{align*}
	&\left|(3(r_h(t+\tau)-r_h(t))-(r_h(t+\tau-\Delta t)-r_h(t-\Delta t)),w)\right|\\
	& \leq \big|(3(r_h(t+\tau)-r_h(t))-(r_h(t+\tau-\Delta t)-r_h(t-\Delta t)),w)\\
	&\qquad \qquad -(3(r_h(t+\tau)-r_h(t))-(r_h(t+\tau-\Delta t)-r_h(t-\Delta t)),w)_h\big|\\
	& \quad + \left|(3(r_h(t+\tau)-r_h(t))-(r_h(t+\tau-\Delta t)-r_h(t-\Delta t)),w)_h\right|\\
	&\leq  C h \norm{r_h}_{L^\infty(0,T;L^2(\dom))}\norm{  w}_{H^2}+ 2\Delta t \sum_{m=m_1+1}^{m_2} |(D_t^- r_h^{m+1},w)_h |\\
	& = C h \norm{r_h}_{L^\infty(0,T;L^2(\dom))}\norm{w}_{H^2}+ 2\Delta t \sum_{m=m_1+1}^{m_2} |(P(\hQ_h^{m+1}):D_t^- Q_h^{m+1},w)_h |\\
	& \leq C  h \norm{r_h}_{L^\infty(0,T;L^2(\dom))}\norm{w}_{H^2}+ 2\Delta t \sum_{m=m_1+1}^{m_2} \norm{\Ih (P(\hQ_h^{m+1}):D_t^- Q_h^{m+1})}_{h,1}\norm{\Ih w}_{h,\infty}\\
	& \leq C  h \norm{r_h}_{L^\infty(0,T;L^2(\dom))}\norm{w}_{H^2}+ C \Delta t \sum_{m=m_1+1}^{m_2}\norm{\Ih (P(\hQ_h^{m+1}):D_t^- Q_h^{m+1})}_{h,1}\norm{ w}_{L^\infty}\\
	& \leq C  h \norm{r_h}_{L^\infty(0,T;L^2(\dom))}\norm{w}_{H^2}+ C \Delta t \sum_{m=m_1+1}^{m_2}\left(\int_{\dom}\sum_{z\in \mathcal{N}_h} |P(\hQ_h^{m+1}(z))|^6\varphi_z dx\right)^{1/6}\norm{D_t^- Q_h^{m+1}}_{L^{6/5}}\norm{ w}_{L^\infty}\\
		& \leq C  h \norm{r_h}_{L^\infty(0,T;L^2(\dom))}\norm{w}_{H^2}+ C \Delta t \sum_{m=m_1+1}^{m_2}\left(\int_{\dom}\sum_{z\in \mathcal{N}_h} |\hQ_h^{m+1}(z)|^6\varphi_z dx\right)^{1/6}\norm{D_t^- Q_h^{m+1}}_{L^{6/5}}\norm{ w}_{L^\infty}\\
	&  \leq C  h \norm{r_h}_{L^\infty(0,T;L^2(\dom))}\norm{w}_{H^2}+ C \Delta t \sum_{m=m_1+1}^{m_2}\norm{\hQ_h^{m+1}}_{L^6}\norm{D_t^- Q_h^{m+1}}_{L^{6/5}}\norm{ w}_{L^\infty}\\
&  \leq C  h \norm{r_h}_{L^\infty(0,T;L^2(\dom))}\norm{w}_{H^2}+ C \Delta t \norm{Q_h}_{L^\infty(0,T;H^1(\dom))} \sum_{m=m_1+1}^{m_2}\norm{D_t^- Q_h^{m+1}}_{L^{6/5}}\norm{ w}_{L^\infty}\\
&  \leq C  h \norm{r_h}_{L^\infty(0,T;L^2(\dom))}\norm{w}_{H^2}+ C \sqrt{\Delta t (m_2-m_1)} \norm{Q_h}_{L^\infty(0,T;H^1(\dom))} \norm{D_t^- Q_h}_{L^2(0,T;L^{6/5}(\dom))}\norm{ w}_{L^\infty}\\
& \leq C(h +\sqrt{\Delta t}+\sqrt{\tau})\norm{w}_{H^2}.
		\end{align*} 
This proves the last estimate,~\eqref{eq:rcontinuity}.
\end{proof}

Next, we use this lemma to show uniform in $\Delta t$ and $h$ weak equicontinuity in time for $Q_h$ and $u_h$. Specifically, we show:
\begin{lemma}
	\label{lem:weakapproxcontinuity}
	Assume that $h^\sigma\leq C \Delta t$, for some $2\leq \sigma\leq k+1$ where $k$ is the maximal polynomial degree of $\Uh$. Let $\ell$ be the maximal polynomial degree of the elements in $\Ph$.
	Then the approximations $u_h$ computed  by~\eqref{eq:step1fully} - \eqref{eq:projectionfullydiscrete} satisfy for any $v\in C^\infty_{c,\Div}(\dom)$ and any $\tau>0$ and any $t\in [0,T-\tau]$
	\begin{equation}
		\label{eq:weakcontu}
		\left|\left(u_h(t+\tau)-u_h(t),v\right)\right| \leq C_v\left(\sqrt{\tau}+ \sqrt{\Delta t}+h^\sigma\right)
	\end{equation}
	and for any $g\in C^\infty(\overline{\dom})$,
		\begin{equation}
		\label{eq:weakcontu2}
		\left|\left(u_h(t+\tau)-u_h(t),\Grad g\right)\right| \leq C_g h^\ell
	\end{equation}
	where $C_v, C_g$ are constants independent of $h,\Delta t>0$. The approximations $Q_h$ satisfy,
	\begin{equation}
		\label{eq:weakcontQ}
			\norm{Q_h(t+\tau)-Q_h(t)}_{L^{6/5}(\dom)} \leq C\left(\sqrt{\tau}+ \sqrt{\Delta t}\right),
	\end{equation}
	and the approximations $r_h$ satisfy for any $w\in H^2(\dom)$,
		\begin{equation}
		\label{eq:weakcontr}
		\left|\left(r_h(t+\tau)-r_h(t),w\right)\right| \leq C_w\left(\sqrt{\tau}+ \sqrt{\Delta t}+h\right),
	\end{equation}
	In particular, the maps $t\mapsto (u_h(t),v)$ and $t\mapsto (Q_h(t),Y)$ (for $Y\in L^6(\dom)$), and $t\mapsto (r_h(t),w)$ are equicontinuous up to errors in $\Delta t$ and $h$ that vanish as $h,\Delta t\to 0$.
\end{lemma}
\begin{proof}
	If $t,t+\tau\leq \Delta t$, the proof of these facts follows exactly along the lines of Lemma~\ref{lem:timecontinuityutilde} with the scheme~\eqref{eq:step1fully} - \eqref{eq:projectionfullydiscrete} replaced by the backward Euler step~\eqref{eq:step1fully1step} - \eqref{eq:projectionfullydiscretestep1}. So we can assume without loss of generality that $t+\tau>\Delta t$. We start by proving~\eqref{eq:weakcontu}. Using the triangle inequality, we write 
	\begin{align}\label{eq:firsttriangleineq}
	\begin{split}
			\left|\left(u_h(t+\tau)-u_h(t),v\right)\right|& \leq \frac12 \left|\left(3(u_h(t+\tau)-u_h(t))-(u_h(t+\tau-\Delta t)-u_h(t-\Delta t)),v \right)\right| \\
		&\quad +\frac12 \left|(u_h(t+\tau)-u_h(t+\tau-\Delta t),v)\right|+ \frac12 \left|\left(u_h(t)-u_h(t-\Delta t),v \right)\right|.
	\end{split}
	\end{align}
	To estimate the first term on the right-hand side, we use Lemma~\ref{lem:timecontinuityutilde}.
	The second and the third term on the right-hand side can be estimated using the triangle inequality once more as follows (replace $t$ by $t+\tau$ for the second term):
	\begin{align}\label{eq:triangle2}
		\begin{split}
			\left|\left(u_h(t)-u_h(t-\Delta t),v \right)\right|&\leq \frac13\left|\left(3(u_h(t)-u_h(t-\Delta t))-(u_h(t-\Delta t)-u_h(t-2\Delta t)),v\right)\right|\\
			&\quad +\frac13 \left|\left(u_h(t-\Delta t)-u_h(t-2\Delta t),v\right)\right|.
		\end{split}
	\end{align}
	Writing $t=m\Delta t +\theta$ for some $m\in \N$ and $\theta\in [0,\Delta t)$ (and similarly $t+\tau=m'\Delta t+\theta'$ for the second term on the right-hand side of~\eqref{eq:firsttriangleineq}), we can iteratively use~\eqref{eq:triangle2} to write
	\begin{multline*}
			\left|\left(u_h(t)-u_h(t-\Delta t),v \right)\right|\\
			\leq \sum_{i=1}^m\frac{1}{3^i}\left|\left(3(u_h((m-i+1)\Delta t +\theta)-u_h((m-i)\Delta t +\theta))-(u_h((m-i)\Delta t +\theta)-u_h((m-i-1)\Delta t +\theta)),v\right)\right|\\
			\quad + \frac{1}{3^m} \left|\left(u_h(\theta)-u_h(0),v\right)\right|.
	\end{multline*}
	Using Lemma~\ref{lem:timecontinuityutilde}, we can estimate the right-hand side by
	\begin{equation}\label{eq:onestepbound}
			\left|\left(u_h(t)-u_h(t-\Delta t),v \right)\right|
		\leq C_v(\sqrt{\Delta t}+ h^\sigma)\sum_{i=1}^m\frac{1}{3^i} +\frac{1}{3^m}\left|\left(u_h^1-u_h^0,v\right)\right|\leq C_v (\sqrt{\Delta t}+ h^\sigma).
	\end{equation}
	Here we also used,
	\begin{equation*}
		\left|\left(u_h(\theta)-u_h(0),v\right)\right| =	\left|\left(u_h^1-u_h^0,v\right)\right|\leq C_v(\sqrt{\Delta t}+h^\sigma).
	\end{equation*}
	This is proved in the same way as~\eqref{eq:timecontutildeweak} by replacing the scheme~\eqref{eq:step1fully} - \eqref{eq:projectionfullydiscrete} by the backward Euler step~\eqref{eq:step1fully1step} - \eqref{eq:projectionfullydiscretestep1} and estimating the terms, therefore we omit the proof here. Using~\eqref{eq:onestepbound} for the second and third term and~\eqref{eq:timecontutildeweak} for the first term in~\eqref{eq:firsttriangleineq}, we obtain
	\begin{equation*}
		\left|\left(u_h(t+\tau)-u_h(t),v \right)\right| \leq C_v\left(\sqrt{\tau}+\sqrt{\Delta t}+ h^\sigma\right),
	\end{equation*}
	which proves the first estimate~\eqref{eq:weakcontu}.
	To prove~\eqref{eq:weakcontu2}, we take $g\in C^\infty(\overline{\dom})$ and use that $u_h$ is weakly discretely divergence free by~\eqref{eq:projection2}. Then we have, denoting $\Ih^p$ the interpolation operator on $\Ph$ as constructed, e.g., in~\cite[Section 9.1]{Ern2021},
	\begin{align*}
		\left|\left(u_h(t+\tau)-u_h(t),\Grad g\right)\right|& \leq 	\left|\left(u_h(t+\tau)-u_h(t),\Grad g - \Grad \Ih^p g\right)\right| + 	\left|\left(u_h(t+\tau)-u_h(t),\Grad \Ih^p g\right)\right|\\
		& = 	\left|\left(u_h(t+\tau)-u_h(t),\Grad g - \Grad \Ih^p g\right)\right| \\
		& \leq C \norm{u_h}_{L^\infty(0,T;L^2(\dom))} \norm{\Grad g - \Grad \Ih^p g}_{L^2}\\
		& \leq C \norm{u_h}_{L^\infty(0,T;L^2(\dom))} h^\ell \norm{g }_{H^{\ell+1}}\\
		& \leq C_g h^\ell,
	\end{align*} 
	by the discrete energy estimate, Lemma~\ref{lem:discenergyestimate} and interpolation estimates~\cite[Theorem 11.13]{Ern2021}. The proof of the third and the fourth estimates,~\eqref{eq:weakcontQ} and~\eqref{eq:weakcontr}, is identical to the proof of~\eqref{eq:weakcontu}, replacing $u_h$ by $Q_h$ and $r_h$ respectively, and using~\eqref{eq:timecontQweak} and~\eqref{eq:rcontinuity} instead, and $w\in H^2(\dom)$ for~\eqref{eq:rcontinuity}.	
\end{proof}

Before proving precompactness of the sequences $\{\tu_h\}_{h>0}$ and $\{Q_h\}_{h>0}$, we recall the following standard Lions lemma:
\begin{lemma}\label{lem:lionslike}
	Let $X_1\subset X_2\subset X_3$ be Banach spaces, such that the embedding $X_1\subset X_2$ is compact and the embedding $X_2\subset X_3$ is continuous. Then for every $\eta>0$ there exists a constant $C_\eta$ depending on $\eta$ and the spaces $X_1,X_2,X_3$ such that
	\begin{equation}
		\label{eq:compactembedding}
		\norm{v}_{X_2}\leq \eta \norm{v}_{X_1}+C_\eta \norm{v}_{X_3},\quad \forall\, v\in X_1.
	\end{equation}
\end{lemma}	
\begin{proof}
A proof of this result can be found for example in~\cite[Ch.III, Lem. 2.1]{Temam1977}.
\end{proof}
We also need the following related technical lemma which follows from the compact embedding $H^1_0(\dom)\subset L^2(\dom)\subset H^{-1}(\dom)$:
\begin{restatable}{lemma}{rellich}
	\label{lem:Rellich2}
	For any $\eta>0$ there exists $N\in \N$ and $\phi_1,\dots, \phi_N\in L^2(\dom)$ such that for any $v\in L^2(\dom)$, we have
	\begin{equation}
		\label{eq:rellich}
		\norm{v}_{H^{-1}(\dom)}\leq \sqrt{\sum_{k=1}^N (v,\phi_k)^2} + \eta \norm{v}_{L^2(\dom)}.
	\end{equation}	
\end{restatable}
The proof of this result is standard and postponed to the Appendix~\ref{app:rellich}.

	Using these results, we can now show the strong convergence in $L^2([0,T]\times\dom)$ of a subsequence of $\{\tu_h,Q_h\}_{h>0}$ to the limits $(u,Q)$:
\begin{lemma}
	\label{lem:convofutilde}
	
	Under the assumptions of Lemma~\ref{lem:timecontinuityutilde},
	we have that a subsequence of $\{\tu_h\}_{h>0}$ converges strongly in $L^2([0,T]\times\dom)$ to $u$ and  a subsequence of $\{Q_h\}_{h>0}$ converges strongly in $L^2([0,T]\times\dom)$ to $Q$. Furthermore, the limits are weakly continuous, i.e., $u,Q\in C([0,T];L^2(\dom)_w)$.
	
\end{lemma}
\begin{proof}
	We prove the statement for the sequence $\{\tu_h\}_{h>0}$, the proof of the statement for $\{Q_h\}_{h>0}$ is similar, and simpler because the complication of having the spatial regularity estimates only for $\tu_h$ and the time continuity estimates only for $u_h$ is removed.
	We use an idea from~\cite{Doering1994} combining uniform bounds in $C(0,T;L^2(\dom)_w)\cap L^2(0,T;H^1(\dom))$. We cannot use it directly since, on one hand, we only have equicontinuity in time up to an error depending on $\Delta t$ and $h$ for $u_h$ and only with respect to smooth test functions, and on the other hand, we only have spatial regularity for the variable $\tu_h$.
	We divide the proof into the following steps:
	
	\smallskip
	
	{\bf Step 1:} We show that for any given $\psi\in L^2(\dom)$, $(u_h(t),\psi)\to  (u(t),\psi)$ uniformly with respect to $t\in [0,T]$ for a subsequence.
	
	\smallskip
	
	First, the bound $\norm{u_h(t)}_{L^2(\dom)}\leq C_0$ for all $t\in [0,T]$ by Lemma~\ref{lem:discenergyestimate} implies that $u_h\weakstar u$ in $L^\infty(0,T;L^2(\dom))$ hence $\norm{u}_{L^\infty(0,T;L^2(\dom))}\leq C_0$.  Next, we fix a dense set $\mathcal{T}:= ([0,T]\cap \Q)\cup \{T\}:= \{t_j\}_{j\in \N}$. By the Banach--Alaoglu theorem, and the uniform bound $\norm{u_h(t)}_{L^2(\dom)}\leq C_0$ for all $t\in [0,T]$, we have $u_h(t_j)\weak w_j$ in $L^2(\dom)$ for all $j\in \N$. In other words, $(u_h(t_j),\varphi)\to (w_j,\varphi)$ for all $j\in \N$ and any $\varphi\in L^2(\dom)$.  We pick a diagonal subsequence, for convenience still denoted by $h$ such that $u_h(t_j)\weak w_j$ for all $t_j\in \mathcal{T}$. 
	
	Next, we let $v\in C^\infty_{c,\Div}(\dom) \oplus \Grad C^\infty(\overline{\dom})$ and define $f^v_h(t): = (u_h(t),v)$. We claim that $f^v_h$ is uniformly Cauchy on $[0,T]$: Given $0<\epsilon <1$ we pick a finite subset $\mathcal{T}_\epsilon= \{t_j\}_{j=1}^{N_\epsilon}$ such that for any $t\in [0,T]$, there exists $t_j\in \mathcal{T}_\epsilon$ such that $|t-t_j|<\epsilon^2/(9 C_v)^2$.  Then we choose $\overline{h},\overline{\Delta t}>0$ small enough such that for any $h_1,h_2<\overline{h}$ and any corresponding $\Delta t_1,\Delta t_2 <\overline{\Delta t}$, we have
	\begin{equation*}
		h_i^{\min\{\sigma,\ell\}} < \frac{\epsilon}{9 C_v},\quad \Delta t_i^{1/2}<\frac{\epsilon}{9 C_v},\quad i=1,2;\quad \max_{t_j\in \mathcal{T}_\epsilon} |f^v_{h_1}(t_j)-f_{h_2}^v(t_j)|<\frac{\epsilon}{9}.
	\end{equation*}
	Then we have for any $t\in [0,T]$ and any $h_1,h_2<\overline{h}$, and $t_j$ such that $|t-t_j|<\epsilon^2/(9 C_v)$ by Lemma~\ref{lem:weakapproxcontinuity},
	\begin{align*}
		|f^v_{h_1}(t)-f^v_{h_2}(t)| &\leq |(u_{h_1}(t)-u_{h_1}(t_j),v)| + |(u_{h_1}(t_j)-u_{h_2}(t_j),v)|+|(u_{h_2}(t_j)-u_{h_2}(t),v)| \\
		& \leq C_v\left(\sqrt{|t-t_j|}+\Delta t_1^{1/2} + h_1^\sigma+h_1^\ell\right) + \frac{\epsilon}{9} + C_v\left(\sqrt{|t-t_j|}+\Delta t_2^{1/2} + h_2^\sigma+h_2^\ell\right) \\
		& \leq \epsilon.
	\end{align*}
	Thus $f^v_h(t)$ is Cauchy, and we have $f^v_h(t)\to f^v(t)$ uniformly in $t\in [0,T]$. Moreover, passing to the limit in~\eqref{eq:weakcontu} and~\eqref{eq:weakcontu2}, we obtain that the limit $f_v$ satisfies
	\begin{equation*}
		|f^v(t)-f^v(s)|\leq C_v \sqrt{|t-s|}.
	\end{equation*} 
	Now we fix $t_0\in [0,T]$ arbitrary. We have now shown that $(u_h(t_0),v)\to f^v(t_0)$ for any $v\in C^\infty_{c,\Div}(\dom) \oplus \Grad C^\infty(\overline{\dom})$. For a general $\psi\in L^2(\dom)$, we have
	\begin{equation}\label{eq:outofideas}
		\limsup_{h_1,h_2}|(u_{h_1}(t_0)-u_{h_2}(t_0),\psi)|\leq 	\limsup_{h_1,h_2}|(u_{h_1}(t_0)-u_{h_2}(t_0),v)|+2C_0 \norm{v-\psi}_{L^2} = 2 C_0 \norm{v-\psi}_{L^2}. 
			\end{equation}
			$C^\infty_{c,\Div}(\dom)$ is dense in $L^2_{\Div}(\dom)$ and $\Grad C^\infty(\overline{\dom})$ is dense in $\Grad H^1(\dom) = \{u\in L^2(\dom),\, u=\Grad p,\, p\in H^1(\dom)\}$ which together span $L^2(\dom)$, i.e., $\overline{\Grad H^1(\dom)}^{L^2}  = L^2_{\Div}(\dom)^\perp$ (cf. Theorem 1.4 in Chapter 1 of~\cite{Temam1977}). Thus the right-hand side of~\eqref{eq:outofideas} can be made arbitrarily small. Thus $f^\psi(t)= \lim_{h\to 0}(u_h(t),\psi)$ exists for any $t\in [0,T]$, is linear in $\psi$ and satisfies $|f^\psi(t)|\leq C_0 \norm{\psi}_{L^2}$. By the Riesz representation theorem, there exists a unique $\check{u}(t)\in L^2(\dom)$ such that $u_h(t)\weak \check{u}(t)$ and $\norm{\check{u}(t)}_{L^2}\leq C_0$ for all $t\in [0,T]$. Moreover, $t\mapsto (\check{u}(t),v)$ is weakly continuous in time for all $v\in C^\infty_{c,\Div}(\dom) \oplus \Grad C^\infty(\overline{\dom})$, and by the same argument as in~\eqref{eq:outofideas} for any $v\in L^2(\dom)$. Thus $\check{u}\in C([0,T];L^2(\dom)_w)$. On the other hand we have, passing to a further subsequence if needed, $u_h\weakstar u$ in $L^\infty(0,T;L^2(\dom))$, i.e., $\int_0^T (u_h(t),v(t)) dt \to \int_0^T (u(t),v(t))dt$ for any $v\in L^1(0,T;L^2(\dom))$. By the previously proved uniform in $t$ convergence, we must also have $\int_0^T (u_h(t),v(t)) dt \to \int_0^T (\check{u}(t),v(t))dt$. Since $v\in L^1(0,T;L^2(\dom))$ is arbitrary, the two limits must agree, i.e., $\check{u} = u$ a.e. in $[0,T]\times \dom$. So $(u_h(t),\psi)\to (u(t),\psi)$ for any $\psi\in L^2(\dom)$ uniformly in $t\in [0,T]$ along the chosen subsequence, and $u\in C([0,T];L^2(\dom)_w)$.
	
	\smallskip
	
{\bf Step 2:} We show that $\tu_h\to u$ in $L^2([0,T]\times\dom)$.

	\smallskip
	
	Due to the compact embedding $L^2(\dom)\subset \subset H^{-1}(\dom)$, we can upgrade the weak convergence $u_h(t)\weak u(t)$ in $L^2(\dom)$ to strong convergence in $H^{-1}(\dom)$. We use Lemma~\ref{lem:Rellich2}. Let $\eta>0$ arbitrary, then there exists $N\in \N$ and $\phi_1,\dots, \phi_N$ such that
	\begin{equation*}
		\norm{u_h(t)-u(t)}_{H^{-1}}\leq \sqrt{\sum_{k=1}^N (u_h(t)-u(t),\phi_k)^2}	+ \eta \norm{u_h(t)-u(t)}_{L^2(\dom)}.
	\end{equation*}
	We take the supremum over $t\in [0,T]$ and use that $\norm{u(t)}_{L^2}\leq C_0$ and $\norm{u_h(t)}_{L^2}\leq C_0$ for all $t\in [0,T]$ and $(u_h(t),\psi)\to (u(t),\psi)$ for any $\psi\in L^2(\dom)$ uniformly in $t\in [0,T]$ by Step 1, thus
	\begin{equation*}
		\sup_{t\in [0,T]} \norm{u_h(t)-u(t)}_{H^{-1}} \leq \sup_{t\in [0,T]} \sqrt{\sum_{k=1}^N (u_h(t)-u(t),\phi_k)^2}	+ \eta 2C_0.
	\end{equation*} 
	Letting $h\to 0$, we obtain 
		\begin{equation*}
		\lim_{h\to 0}\sup_{t\in [0,T]} \norm{u_h(t)-u(t)}_{H^{-1}} \leq  \eta 2C_0. 
	\end{equation*} 
	Since $\eta>0$ was arbitrary, this proves that $\sup_{t\in [0,T]} \norm{u_h(t)-u(t)}_{H^{-1}} \to 0$ as $h\to 0$ and therefore also 
$u_h\to u$ in $L^2(0,T;H^{-1}(\dom))$. Now we have by the classical Lemma~\ref{lem:lionslike} by Lions with $X_2=L^2(\dom)$, $X_1=H^1(\dom)$ and $X_3 = H^{-1}(\dom)$, for any $\eta>0$,
\begin{align*}
	&\int_0^T \norm{\tu_h-u}^2_{L^2(\dom)}dt \\
	&\leq \eta^2 \int_0^T \norm{\tu_h-u}^2_{H^1(\dom)} dt + C_\eta^2\int_0^T \norm{\tu_h-u}^2_{H^{-1}(\dom)} dt\\
	&\leq \eta^2 C  + 2C_\eta^2\int_0^T \norm{\tu_h-u_h}^2_{H^{-1}(\dom)} dt+ 2C_\eta^2\int_0^T \norm{u-u_h}^2_{H^{-1}(\dom)} dt\\
	&\leq \eta^2 C  + 2C_\eta^2\int_0^T \norm{\tu_h-u_h}^2_{L^2(\dom)} dt+ 2C_\eta^2\int_0^T \norm{u-u_h}^2_{H^{-1}(\dom)} dt,
\end{align*}
where we used the discrete energy estimate, Lemma~\ref{lem:discenergyestimate} in the second line and Lemma~\ref{lem:samelimits} for the last line.
Thus,
\begin{equation*}
\limsup_{h,\Delta t \to 0}\int_0^T \norm{\tu_h-u}^2_{L^2(\dom)}dt \leq C\eta^2.
\end{equation*}
Since $\eta>0$ was arbitrary, this implies the result.
\end{proof}
\begin{remark}
	\label{rem:continuityr}
	From Step 1 of the proof of Lemma~\ref{lem:convofutilde}, the a priori energy estimate, Lemma~\ref{lem:discenergyestimate}, and the approximate equicontinuity, Lemma~\ref{lem:weakapproxcontinuity}, it follows that $(r_h(t),w)\to (r(t),w)$ for $w\in L^2(\dom)$ and any $t\in [0,T]$, and that the limit satisfies $r\in C([0,T];L^2(\dom)_w)$.
\end{remark}

Using these preliminary estimates, we can now prove: 
\begin{theorem}
	The sequences $\{u_h,\bar{u}_h,\widetilde{u}_h,\hu_h,Q_h,\bar{Q}_h,\mathcal{H}_h,r_h\}_{h>0}$ converge, up to a subsequence, as $\Delta t, h \to 0$ to a weak solution $(u,Q,\mathcal{H})$ of \eqref{seq:BerisEdwards} as in Definition~\ref{def:weaksol} under the condition that    {$h^{2}=o(\Delta t)$}.
\end{theorem}
\begin{proof}
	Using the definitions of the interpolations~\eqref{eq:defuh} -- \eqref{eq:defph}, we can rewrite the numerical scheme~\eqref{eq:step1fully} -- \eqref{eq:projectionfullydiscrete} as
	\begin{subequations}\label{eq:step1inter}
		\begin{align}
		\label{eq:udiscinterp}
		\left(D_t^- u_h,v\right)+ b(\hu_h,\tu_h,v)+\mu(\Grad \tu_h,\Grad v) &= -(\sigma(\bar{Q}_h,\HH_h),\Grad v)-(\HH_h\Grad \bar{Q}_h,v)-(\Grad p_h, v)+(f_h,v),\\
		\label{eq:Qdiscinterp}
		\left(D_t^- Q_h,Y\right)+((\tu_h\cdot\Grad) \bar{Q}_h,Y) &=  (s(\tu_h,\bar{Q}_h),Y)+M (\HH_h,Y),\\
			\label{eq:Hdiscinterp}
		(\HH_h,Z) &= -L(\Grad Q_h,\Grad Z)-(r_h P(\bar{Q}_h),Z)_h,\\
		\label{eq:rdiscinterp}
		\left(D_t^-r_h,w\right)_h & = (P(\bar{Q}_h):D^-_t Q_h,w)_h,	
		\end{align}
	\end{subequations}
	and 
	\begin{equation}
	\label{eq:divconstraintinterp}
	( u_h,\Grad q)=0,
	\end{equation}
	with the test functions in the same space as in~\eqref{eq:step1fully} -- \eqref{eq:projectionfullydiscrete}.
	Here, we used the notation~\eqref{eq:timediffdef1} and~\eqref{eq:timediffdef2} for the time differences again.
	By the previous considerations, weak convergences~\eqref{eq:uhweakconv} -- \eqref{eq:Hhweakconv}, Lemma~\ref{lem:discenergyestimate}, Lemma~\ref{lem:weakapproxcontinuity}, Lemma~\ref{lem:samelimits}, and Lemma~\ref{lem:convofutilde}, we can improve~\eqref{eq:uhweakconv} -- \eqref{eq:Hhweakconv} to
	\begin{align}
	u_h,\bar{u}_h,\widetilde{u}_h,\hu_h\to u,&\quad \text{in }\, L^2([0,T]\times\dom),\label{eq:uhstrongconv}\\
	\widetilde{u}_h\weak u,&\quad \text{in }\, L^2(0,T;H^1_0(\dom)),\label{eq:tuhweakconvagain}\\
	Q_h,\bar{Q}_h\to Q,&\quad \text{in }\, L^2([0,T]\times\dom),\\
	Q_h,\bar{Q}_h\weakstar Q,&\quad \text{in }\, L^\infty(0,T;H^1_0(\dom)),\\
	r_h\weakstar r,&\quad \text{in }\, L^\infty(0,T;L^2(\dom)),\\
	\mathcal{H}_h\weak \mathcal{H},&\quad \text{in }\, L^2([0,T]\times\dom).\label{eq:Hhweakconvagain}
	\end{align}
	Moreover, as noted previously, the limit $u$ is divergence free almost everywhere in $[0,T]\times\dom$ and $u,Q,r\in C(0,T;L^2(\dom)_w)$. Therefore, we also have for any $v\in L^2(\dom)$
	\begin{equation}\label{eq:uattainmentofinitiladata}
		(u_0,v)=\lim_{h,\Delta t\to 0}(u_0^h,v) = \lim_{h,\Delta t\to 0} (u_h(0),v) = (u(0),v),
	\end{equation}
	where the last equality is due to the weak convergence of $u_h(t)$ in $L^2(\dom)$ for any $t\in [0,T]$ (cf. Lemma~\ref{lem:weakapproxcontinuity}), and similarly for any $Y\in L^2(\dom)$ and any $w\in L^2(\dom)$,
	\begin{equation}\label{eq:Qattainmentinitialdata}
		(Q_0,Y)=\lim_{h,\Delta t\to 0}(Q_0^h,Y) = \lim_{h,\Delta t\to 0} (Q_h(0),Y) = (Q(0),Y),
	\end{equation}
	and
	\begin{equation}\label{eq:rattainmentinitialdata}
		(r_0,w)=\lim_{h,\Delta t\to 0}(r_0^h,w) = \lim_{h,\Delta t\to 0} (r_h(0),w) = (r(0),w).
	\end{equation}
	
	Now we take test functions $v, Y, Z,w\in C^\infty_c((0,T)\times \dom)$ with $v$ divergence free, integrate in time and rewrite~\eqref{eq:udiscinterp} as
		\begin{subequations}\label{eq:step1app}
		\begin{align}
		\label{eq:uapp}
		&\int_0^T\left[\left(D^-_t u_h,v\right)+b(\hu_h,\tu_h,v)+\mu(\Grad \tu_h,\Grad v) +(\sigma(\bar{Q}_h,\HH_h),\Grad v)+(\HH_h\Grad \bar{Q}_h,v)-(f_h,v)\right] dt\\
		&=\int_0^T\left[\left(D^-_t u_h,v-\Ih^u v\right)+ b(\hu_h,\tu_h,v-\Ih^u v)+\mu(\Grad \tu_h,\Grad (v-\Ih^u v))\right] dt\notag\\
		&\quad +\int_0^T\left[(\sigma(\bar{Q}_h,\HH_h),\Grad (v-\Ih^u v))+(\HH_h\Grad \bar{Q}_h,v-\Ih^u v)+(\Grad p_h, v-\Ih^u v)-(f_h,v-\Ih^u v)\right] dt,\notag\\
		\label{eq:Qapp}
		&\int_0^T\left[\left(D^-_t Q_h,Y\right)+((\tu_h\cdot\Grad) \bar{Q}_h,Y)-(s(\tu_h,\bar{Q}_h),Y)-M (\HH_h,Y)\right] dt\\
		&=\int_0^T\left[\left(D^-_t Q_h,Y-\Ih Y\right) + ((\tu_h\cdot\Grad) \bar{Q}_h,Y-\Ih Y) \right] dt\notag\\
		&\quad -\int_0^T\left[ (s(\tu_h,\bar{Q}_h),Y-\Ih Y) + M (\HH_h,Y-\Ih Y) \right] dt ,\notag\\
			\label{eq:Happ}
		&\int_0^T\left[ (\HH_h,Z)  + L(\Grad Q_h,\Grad Z) + (r_h P(\bar{Q}_h),\Ih Z)_h \right] dt\\
		&=\int_0^T\left[ (\HH_h,Z-\Ih Z)+L(\Grad Q_h,\Grad (Z-\Ih Z))\right] dt,	\notag\\
		\label{eq:rapp}
		&\int_0^T\left[ \left(D^-_t r_h,\Ih w\right)_h  -  (P(\bar{Q}_h):D^-_t Q_h,\Ih w)_h \right] dt=0, 
		\end{align}
	\end{subequations}
where $\Ih$ is the Lagrangian interpolation on $\Xh$ (and $\Mh$ and $\Mhz$ respectively),  and $\Ih^u$ is the interpolation operator on the finite element space $\Uh$, as defined for example in~\cite[Section 9.1]{Ern2021}. We choose $\Delta t$ and $h$ small enough such that $v(t,\cdot)=Y(t,\cdot)=w(t,\cdot)=0$ on $[T-2\Delta t,T]$ and on $[0,2\Delta t]$ and that $Z(\cdot,x)=0$ for $x\in K$ where $K$ is an element adjacent to the boundary.
We consider the limits $h,\Delta t\to 0$ in each of these equations. We start by rewriting the time differences as follows:
\begin{align*}
	\int_0^T\left(D^-_t u_h,v\right) dt & = \frac{1}{2\Delta t}\int_{2\Delta t}^T (3u_h(t)-4u_h(t-\Delta t)+ u_h(t-2\Delta t),v(t)) dt  \\
	& = \frac{1}{2\Delta t}\left( \int_{2\Delta t}^T 3 (u_h(t),v(t)) dt - \int_{\Delta t}^{T-\Delta t}4 (u_h(t),v(t+\Delta t)) dt +\int_{0}^{T-2\Delta t}(u_h(t),v(t+2\Delta t)) dt \right) \\
	&  = \frac{1}{2\Delta t}\left( \int_{0}^T 3 (u_h(t),v(t)) dt - \int_{0}^{T}4 (u_h(t),v(t+\Delta t)) dt +\int_{0}^{T}(u_h(t),v(t+2\Delta t)) dt \right),
\end{align*}
where we used that $v$ is compactly supported in time and zero on $[0,2\Delta t]$ and $[T-2\Delta t,T]$. Letting 
\begin{equation*}
	D_t^+ v(t) := \frac{-v(t+2\Delta t)+4v(t+\Delta t)-3 v(t)}{2\Delta t},
\end{equation*}
we can rewrite this as
\begin{equation*}
	\int_0^T\left(D^-_t u_h,v(t)\right) dt  = -\int_0^T (u_h(t),D_t^+ v(t)) dt. 
\end{equation*}
Similarly,
\begin{equation*}
		\int_0^T\left(D^-_t Q_h,Y(t)\right) dt = - \int_0^T (Q_h(t),D_t^+ Y(t)) dt,\quad 	\int_0^T\left(D^-_t r_h,\Ih w(t)\right)_h dt = - \int_0^T (r_h(t),D_t^+ \Ih w(t))_h dt.
\end{equation*}
Since $D^+_t v(t)\to \partial_t v(t)$ pointwise due to the smoothness of $v$ and $u_h\weak u$ in $L^\infty(0,T;L^2(\dom))$, we have
\begin{equation}\label{eq:timederconv}
	\begin{split}
	&\int_0^T\left(D^-_t u_h,v\right) dt  \stackrel{h,\Delta t\to 0}{\longrightarrow} -\int_0^T (u(t),\partial_t v(t)) dt, \\
	&\int_0^T\left(D^-_t Q_h,Y\right) dt  \stackrel{h,\Delta t\to 0}{\longrightarrow} -\int_0^T (Q(t),\partial_t Y(t)) dt= \int_0^T(\partial_t Q(t),Y(t)) dt,
	\end{split}
	\end{equation}
	where the last identity follows since $\norm{D_t^- Q_h}_{L^2(0,T;L^{6/5}(\dom))}\leq C$ uniformly in $h,\Delta t$ by~\eqref{eq:timederQ} and thus it is true for the weak limit also.
Using this, we proceed to studying the limit of~\eqref{eq:Qapp} as $h,\Delta t\to 0
$. Convergences~\eqref{eq:uhstrongconv} -- \eqref{eq:Hhweakconvagain} and~\eqref{eq:timederconv}  allow us to pass to the limit in all the terms on the left-hand side and replace the approximations by their corresponding limits. Notice that the correction term $\frac{2\xi}{d^2}\Div u I$ in $s$ vanishes since the limit $u$ is divergence-free. So we only have to show that the right-hand side vanishes. 
	For the first term on the right-hand side, we have, after summing by parts, and using interpolation estimates,
	\begin{equation*}
	\begin{split}
			\left|\int_0^T\left(D^-_t Q_h,Y-\Ih Y\right) dt \right|& =\left|\int_0^T\left( Q_h,D_t^+Y-\Ih D_t^+Y\right) dt \right| \\
			& \leq C h^2\norm{\partial_t Y}_{L^2(0,T;H^2(\dom))},
	\end{split}
		\end{equation*}
	which converges to zero as $h,\Delta t\to 0$. 
The second term on the right-hand side we can estimate as follows:
\begin{align*}
	\left|\int_0^T((\tu_h\cdot\Grad)\bar{Q}_h,Y-\Ih Y)dt\right|	&\leq\int_0^T\norm{\tu_h}_{L^2}\norm{\Grad \bar{Q}_h}_{L^2}\norm{\Ih Y-Y}_{L^\infty} dt\\
	& \leq C\norm{\tu_h}_{L^\infty(0,T;L^2(\dom))}\norm{\Grad \bar{Q}_h}_{L^\infty(0,T;L^2(\dom))} h \norm{Y}_{L^1(0,T;W^{1,\infty}(\dom))},
\end{align*}
where we used the discrete energy estimate, Lemma~\ref{lem:discenergyestimate}, and interpolation estimates.
Clearly, this term converges to zero as $h,\Delta t\to 0$. 
We continue to the third term: Using Sobolev embeddings,
\begin{align*}
	&	\left|\int_0^T(s(\tu_h,\bar{Q}_h),Y-\Ih Y) dt\right| \\
	& \leq C  \int_0^T\!\! \int_{\dom} \left(|\Grad\tu_h||\bar{Q}_h||\Ih Y-Y|+|\Grad\tu_h||\Ih Y-Y| +|\Grad\tu_h||\bar{Q}_h|^2 |\Ih Y-Y| \right)dx dt\\
	& \leq C \int_{0}^{T}\!\!\left(\norm{\Grad\tu_h}_{L^2}\norm{\bar{Q}_h}_{L^2}\norm{\Ih Y-Y}_{L^\infty}+ \norm{\Grad\tu_h}_{L^2}\norm{\Ih Y-Y}_{L^2}+\norm{\Grad\tu_h}_{L^2}\norm{\bar{Q}_h}_{L^4}^2\norm{\Ih Y-Y}_{L^\infty}    \right) dt\\
	& \leq C h\int_{0}^{T}\!\!\bigg(\norm{\Grad\tu_h}_{L^2}\norm{\bar{Q}_h}_{L^2}\norm{Y}_{W^{1,\infty}} + \norm{\Grad\tu_h}_{L^2}\norm{Y}_{W^{1,\infty}}\\
	&\qquad +\norm{\Grad\tu_h}_{L^2}\norm{\bar{Q}_h}_{L^4}^2\norm{Y}_{W^{1,\infty}} \bigg) dt\\
		& \leq C h \norm{\Grad \tu_h}_{L^2([0,T]\times\dom)}\left(1+\norm{\bar{Q}_h}_{L^\infty(0,T;H^1(\dom))}^2\right) \norm{Y}_{L^2(0,T;W^{1,\infty}(\dom))}\stackrel{h,\Delta t\to 0}{\longrightarrow} 0.
\end{align*}
Finally, for the fourth term, we have
 \begin{equation*}
	 \left|M\int_0^T(\HH_h,Y-\Ih Y)dt\right|\leq C h^2\norm{\HH_h}_{L^2([0,T]\times \dom)}\norm{Y}_{L^2(0,T;H^2(\dom))}\stackrel{h,\Delta t\to 0}{\longrightarrow} 0.
	 \end{equation*}
Thus all terms on the right-hand side of~\eqref{eq:Qapp} vanish or converge to zero and we obtain in the limit
\begin{equation}
	\label{eq:Qlim}
	\int_0^T\left[\left(\partial_t Q ,Y\right)+((u\cdot\Grad) Q,Y)-(s(u,Q),Y)-M (\HH,Y)\right] dt = 0.
\end{equation}
	Next, we consider the equation for $\HH_h$,~\eqref{eq:Happ}. Using the convergence properties~\eqref{eq:uhstrongconv} -- \eqref{eq:Hhweakconvagain}, we can pass to the limit in the first two terms on the left-hand side:
\begin{equation*}
	\lim_{h,\Delta t\to 0}\left[\int_0^T\left[(\HH_h,Z) +L(\Grad {Q}_h,\Grad Z)\right]dt \right] = \int_0^T\left[(\HH,Z) +L(\Grad Q,\Grad Z)\right]dt.
\end{equation*}
For the third term on the left-hand side of~\eqref{eq:Happ}, we use the properties of the discrete inner product $(\cdot,\cdot)_h$ stated in Lemmas~\ref{lem:masslumped},~\ref{lem:masslumpedLp} and~\ref{lem:masslump2}:
\begin{align*}
	&\lim_{h,\Delta t\to 0}\left|\int_0^T(r_h P(\bar{Q}_h),\Ih Z)_h dt-\int_0^T(r P(Q),Z)dt\right|\\
	&\quad \leq \lim_{h,\Delta t\to 0}\left|\int_0^T\left[(r_h P(\bar{Q}_h),\Ih Z)_h -(\Ih(r_h  P(\bar{Q}_h)), \Ih Z)\right] dt\right|\\
	&\qquad +\lim_{h,\Delta t\to 0}\left|\int_0^T\left[(\Ih(r_h P(\bar{Q}_h)),\Ih Z) -(r_h P(\bar{Q}_h), \Ih Z)\right]dt\right|\\
		&\qquad +\lim_{h,\Delta t\to 0}\left|\int_0^T\left[(r_h P(\bar{Q}_h),\Ih Z) -(r_h P(\bar{Q}_h),  Z)\right]dt\right|\\
	&\qquad +\lim_{h,\Delta t\to 0}\left|\int_0^T\left[ (r_h P(\bar{Q}_h), Z) -(r P(Q),Z)\right] dt\right|\\
		& \leq \lim_{h,\Delta t\to 0} C h \norm{\Ih(r_hP(\bar{Q}_h))}_{L^1([0,T]\times\dom)}\norm{\Grad\Ih Z}_{L^\infty([0,T]\times\dom)}\\
	&\qquad 	+ \lim_{h,\Delta t\to 0} \norm{r_h P(\bar{Q}_h) - \Ih(r_h P(\bar{Q}_h))}_{L^1([0,T]\times\dom)}\norm{\Ih Z}_{L^\infty([0,T]\times\dom)}\\
	&	\qquad + \lim_{h,\Delta t\to 0} \norm{r_h P(\bar{Q}_h)}_{L^1([0,T]\times\dom)}\norm{Z - \Ih Z}_{L^\infty([0,T]\times\dom)}\\
	&\qquad +\lim_{h,\Delta t\to 0}\left|\int_0^T\left[ (r_h P(\bar{Q}_h), Z) -(r P(Q),Z)\right] dt\right|\\
	& \leq \lim_{h,\Delta t\to 0} C h \norm{r_h}_{L^2([0,T]\times\dom)}\norm{ \bar{Q}_h}_{L^2([0,T]\times\dom)}\norm{\Grad Z}_{L^\infty([0,T]\times\dom)}\\
	&\qquad +\lim_{h,\Delta t\to 0} Ch \norm{r_h}_{L^2([0,T]\times\dom)}\norm{\Grad \bar{Q}_h}_{L^2([0,T]\times\dom)}\norm{Z}_{L^\infty([0,T]\times\dom)} \\
		&\qquad +\lim_{h,\Delta t\to 0} Ch \norm{r_h}_{L^2([0,T]\times\dom)}\norm{ \bar{Q}_h}_{L^2([0,T]\times\dom)}\norm{\Grad Z}_{L^\infty([0,T]\times\dom)} \\
	&\qquad +\lim_{h,\Delta t\to 0}\left|\int_0^T\left[ (r_h P(\bar{Q}_h), Z) -(r P(Q),Z)\right] dt\right|.
\end{align*}
The first three terms are zero thanks to the energy estimate and the last one is zero thanks to the weak convergence of $r_h$ and the strong convergence of $\bar{Q}_h$ in $L^2$.
For the two terms on the right-hand side of~\eqref{eq:Happ}, we have using interpolation estimates and the discrete energy estimate
\begin{equation*}
	\left|\int_0^T\left[ (\HH_h,Z-\Ih Z) +L(\Grad {Q}_h ,\Grad (Z-\Ih Z))\right] dt\right| \leq C h \norm{Z}_{L^2(0,T;H^2(\dom))},
\end{equation*}
which goes to zero as $h,\Delta t \to 0$.	
Thus we conclude that the limit $\HH$ satisfies
\begin{equation}\label{eq:Hprelimlimit}
	\int_0^T\left[(\HH,Z)+L(\Grad Q,\Grad Z)+(rP(Q), Z) \right] dt =0.
\end{equation}
Using this, we will show that $\Grad Q_h$ converges strongly in $L^2([0,T]\times\dom)$ up to a subsequence.
To do so, we first note that we can use a density argument to extend identity~\eqref{eq:Hprelimlimit} to test functions $Z\in L^2(0,T;H_0^1(\dom))$. Hence, we can take the limit $Q$ as a test function:
\begin{equation}
	\label{eq:QHidentity1}
	\int_0^T\left[(\HH,Q)+L(\Grad Q,\Grad Q)+(rP(Q),Q)\right] dt =0.
\end{equation}
On the other hand, we can take $Z ={Q}_h$ as a test function in~\eqref{eq:Happ}, which is in $\Mhz$ and hence the right-hand side vanishes:
\begin{equation}\label{eq:QhHhidentity}
	\int_0^T\left[(\HH_h,{Q}_h ) +L(\Grad {Q}_h,\Grad {Q}_h)+(r_h P(\bar{Q}_h),{Q}_h)_h\right] dt = 0.
\end{equation}
We take the limit $h,\Delta t\to 0$ in this identity. For the first term, due to the weak convergence of $\HH_h$ and the strong convergence of ${Q}_h$, we have
\begin{equation*}
	\lim_{h,\Delta t\to 0} \int_0^T(\HH_h,{Q}_h ) dt = \int_0^T(\HH,Q ) dt.
\end{equation*}
For the third term, we have using Lemmas~\ref{lem:masslumped},~\ref{lem:masslumpedLp},~\ref{lem:masslump2}, the discrete energy estimate, Lemma~\ref{lem:discenergyestimate}, the inverse estimate~\cite[Lemma 12.1, Exercise 12.4]{Ern2021},
\begin{equation}\label{eq:inverseestimate}
	\norm{v_h}_{L^p(\dom)}\leq C h^{d\left(\frac{1}{p}-\frac{1}{r}\right)}\norm{v_h}_{L^r},
\end{equation}
for $p=\infty$ and $r=6$, and the Sobolev inequality,
\begin{align*}
	&\left|\int_0^T  \left[(r_h P(\bar{Q}_h),{Q}_h)_h - (rP(Q),Q)\right] dt \right|\\
	&\leq  \left|\int_0^T  \left[(r_h P(\bar{Q}_h),{Q}_h)_h - (\Ih(r_h P(\bar{Q}_h)),{Q}_h)\right] dt \right|+ \left|\int_0^T  \left[(\Ih(r_h P(\bar{Q}_h)),{Q}_h) - (r_h P(\bar{Q}_h),Q_h)\right] dt \right|\\
	&\qquad +  \left|\int_0^T  \left[(r_h P(\bar{Q}_h),{Q}_h) - (r_h P(\bar{Q}_h),Q)\right] dt \right|+ \left|\int_0^T  \left[(r_h P(\bar{Q}_h),Q) - (r_h P(Q),Q)\right] dt \right|\\
	&\qquad + \left|\int_0^T  \left[(r_h P(Q),Q) - (r P(Q),Q)\right] dt \right|\\
	&\leq C h \norm{\Ih(r_h P(\bar{Q}_h))}_{L^2([0,T]\times\dom)}\norm{\Grad Q_h}_{L^2([0,T]\times\dom)}\\
	&\qquad + \norm{\Ih(r_hP(\bar{Q}_h))-r_h P(\bar{Q}_h)}_{L^1([0,T]\times\dom)}\norm{Q_h}_{L^\infty([0,T]\times\dom)}\\
	& \qquad +\norm{r_h}_{L^2([0,T]\times\dom)}\norm{P(\bar{Q}_h)}_{L^6([0,T]\times\dom)}\norm{Q_h-Q}_{L^3([0,T]\times\dom)}\\
	& \qquad + C\norm{r_h}_{L^2([0,T]\times\dom)}\norm{\bar{Q}_h-Q}_{L^3([0,T]\times\dom)}\norm{Q}_{L^6([0,T]\times\dom)} \\
	& \qquad +  \left|\int_0^T  \left[(r_h P(Q),Q) - (r P(Q),Q)\right] dt \right|\\
		&\leq C h \norm{\norm{\Ih(r_h P(\bar{Q}_h))}_{h,2}}_{L^2([0,T])}\norm{\Grad Q_h}_{L^2([0,T]\times\dom)}\\
	&\qquad + C h^{1-d/6} \norm{r_h}_{L^2([0,T]\times\dom)}\norm{\Grad \bar{Q}_h}_{L^2([0,T]\times\dom)}\norm{Q_h}_{L^\infty(0,T;H^1(\dom))}\\
	& \qquad +C\norm{r_h}_{L^2([0,T]\times\dom)}\norm{{Q}_h}_{L^\infty(0,T;H^1(\dom))}\norm{Q_h-Q}_{L^2([0,T]\times\dom)}^{1/3}\norm{Q_h-Q}_{L^4([0,T]\times\dom)}^{2/3}\\
	& \qquad + C\norm{r_h}_{L^2([0,T]\times\dom)}\norm{\bar{Q}_h-Q}_{L^2([0,T]\times\dom)}^{1/3}\norm{\bar{Q}_h-Q}_{L^4([0,T]\times\dom)}^{2/3}\norm{Q}_{L^\infty(0,T;H^1(\dom))} \\
	& \qquad +  \left|\int_0^T  \left[(r_h P(Q),Q) - (r P(Q),Q)\right] dt \right|\\
		&\leq C h \norm{r_h}_{L^2([0,T]\times\dom)}\norm{Q_h}_{L^\infty([0,T]\times\dom)}\\
	&\qquad + C h^{1-d/6} +C\norm{Q_h-Q}_{L^2([0,T]\times\dom)}^{1/3}+ C\norm{\bar{Q}_h-Q}_{L^2([0,T]\times\dom)}^{1/3}\\
	& \qquad +  \left|\int_0^T  \left[(r_h P(Q),Q) - (r P(Q),Q)\right] dt \right|\\
		&\leq C h^{1-d/6} \norm{r_h}_{L^2([0,T]\times\dom)}\norm{Q_h}_{L^\infty(0,T;H^1(\dom))}\\
	&\qquad + C h^{1-d/6} +C\norm{Q_h-Q}_{L^2([0,T]\times\dom)}^{1/3}+ C\norm{\bar{Q}_h-Q}_{L^2([0,T]\times\dom)}^{1/3}\\
	& \qquad +  \left|\int_0^T  \left[(r_h P(Q),Q) - (r P(Q),Q)\right] dt \right|.
\end{align*}
The third and the fourth term go to zero by the strong convergence of $Q_h$ and $\bar{Q}_h$ in $L^2$ and the last term goes to zero by the weak convergence of $r_h$.
Thus in the limit $h,\Delta t\to 0$, identity~\eqref{eq:QhHhidentity} becomes
\begin{equation}\label{eq:QhHhidentitylimit}
	\lim_{h,\Delta t\to 0} \int_0^T L(\Grad {Q}_h,\Grad {Q}_h) dt +\int_0^T\left[(\HH,Q )+(r P(Q),Q)\right] dt = 0.
\end{equation}
Subtracting this from~\eqref{eq:QHidentity1}, we obtain
\begin{equation*}
	\lim_{h,\Delta t\to 0} \int_0^T L(\Grad  {Q}_h,\Grad {Q}_h) dt = L \norm{\Grad Q}_{L^2([0,T]\times\dom)}^2,
\end{equation*}
i.e., the $L^2$-norm of  $\Grad {Q}_h$ converges. Combining it with the weak convergence of $\Grad{Q}_h$, this implies that  $\Grad {Q}_h$ converges strongly in $L^2([0,T]\times\dom)$ up to a subsequence. Combining it with Lemma~\ref{lem:samelimits}, we also obtain the strong convergence of $\bar{Q}_h$ in $L^2(0,T;H^1(\dom))$.

Next, we consider equation~\eqref{eq:uapp}. We have already shown the convergence of the first term in~\eqref{eq:timederconv}.
		 For the second term, we  use that $\hu_h,\tu_h\to u$ in $L^2([0,T]\times\dom)$ and $\tu_h\weak u$ in $L^2([0,T];H^1_0(\dom))$:
\begin{multline*}
	\int_0^T b(\hu_h,\tu_h,v ) dt =  \frac12\int_0^T\int_{\dom}(\hu_h\cdot \Grad)\tu_h\cdot v-( \hu_h\cdot\Grad)v\cdot \tu_h  dxdt\\
	\stackrel{h,\Delta t\to 0}{\longrightarrow}\frac12\int_0^T\int_{\dom} (u\cdot\Grad)u \cdot v-(u\cdot\Grad) v\cdot u dx dt =\int_0^T\int_{\dom} (u\cdot\Grad)u \cdot vdx dt,
\end{multline*}
using that $\Div u=0$ a.e. in $[0,T]\times\dom$.
For the third term, we use that $\Grad \tu_h\weak \Grad u$ in $L^2([0,T]\times\dom)$ and hence
\begin{equation*}
	\int_0^T(\Grad\tu_h,\Grad v) dt \stackrel{h,\Delta t\to 0}{\longrightarrow} \int_0^T (\Grad u,\Grad v) dt.
\end{equation*}
For the fourth term, we have 
\begin{equation*}
	\int_0^T(\sigma(\bar{Q}_h,\HH_h) ,\Grad v) dt \stackrel{h,\Delta t\to 0}{\longrightarrow}\int_0^T (\sigma(Q,\HH), \Grad v) dt
\end{equation*}	
using that $\bar{Q}_h$ converges strongly in $L^2([0,T];H^1(\dom))$ and $\HH_h$ converges weakly in $L^2([0,T]\times\dom)$.
For the fifth term we also use the strong convergence of $\bar{Q}_h$ in $L^2([0,T];H^1(\dom))$ and the weak convergence of $\HH_h$ in $L^2([0,T]\times\dom)$ to conclude that
\begin{equation*}
	\int_0^T(\HH_h\Grad\bar{Q}_h , v) dt \stackrel{h,\Delta t\to 0}{\longrightarrow}\int_0^T (\HH\Grad Q, v) dt.
\end{equation*}	
We consider the sixth term: We have
\begin{equation*}
	\begin{split}
		\int_0^T (f_h,v) dt &= \sum_{m=0}^{N-1}\int_{t^m}^{t^{m+1}} (\PUh f^{m+1},v) dt\\
		& = \sum_{m=0}^{N-1}\int_{t^m}^{t^{m+1}}(  f^{m+1},\PUh v) dt\\
		& = \sum_{m=0}^{N-1}\int_{t^m}^{t^{m+1}} \frac{1}{\Delta t}\int_{t^{m+1/2}}^{t^{m+3/2}}(  f(\tau),\PUh v(t)-v(t))d\tau dt + \sum_{m=0}^{N-1}\int_{t^m}^{t^{m+1}} \frac{1}{\Delta t}\int_{t^{m+1/2}}^{t^{m+3/2}}(  f(\tau), v(t)) d\tau dt\\
		& \stackrel{h,\Delta t\to 0}{\longrightarrow} \int_0^T (f,v) dt,
	\end{split}
\end{equation*}
since time averages converge in $L^2$ and $\norm{v-\PUh v}_{L^2(\dom)}\to 0$ as $h\to 0$. Hence the left-hand side of~\eqref{eq:uapp} converges to
\begin{equation*}
	\int_0^T\left[-( u ,\partial_t v)+ ((u\cdot \Grad) u, v)+\mu(\Grad u,\Grad v) +(\sigma(Q,\HH),\Grad v)+(\HH\Grad Q,v)-(f,v)\right] dt.
\end{equation*}
	In order to conclude that $u$ satisfies the weak formulation, we thus need to show that the right-hand side converges to zero.
	So, we consider the terms on the right-hand side: For  the first term, we rewrite it using summation by parts and then  interpolation estimates (e.g.~\cite[Theorem 11.13]{Ern2021}) 
\begin{align*}
	\left|\int_0^T (D_t^- u_h, v-\Ih^u  v ) dt \right| & = \left|\int_0^T (u_h,D_t^+ v-\Ih^u D_t^+ v ) dt\right|\\
	& \leq C h^{k+1}\norm{u_h}_{L^\infty(0,T;L^2(\dom))} \norm{D_t^+ v}_{L^1(0,T;H^{k+1}(\dom))}\\
	& \leq C h^{k+1} \norm{u_h}_{L^\infty(0,T;L^2(\dom))} \norm{\partial_t  v}_{L^1(0,T;H^{k+1}(\dom))},
\end{align*}
which vanishes as $\Delta t, h \to 0$.
Next,
\begin{equation*}
\begin{split}
&\left|\int_0^Tb(\hu_h,\tu_h,v-\Ih^u v) dt\right|\\
&\leq C\norm{\hu_h}_{L^\infty(0,T;L^2)}\norm{ \tu_h}_{L^2(0,T;H^1(\dom))}\norm{v-\Ih^u v}_{L^2([0,T];W^{1,\infty}(\dom))}\\
&\leq C h \norm{v }_{L^2(0,T;W^{2,\infty}(\dom))},
\end{split}
\end{equation*}
where we used interpolation estimates~\cite[Theorem 11.13]{Ern2021} and the energy estimate, Lemma~\ref{lem:discenergyestimate}. Thus also this term vanishes as $h,\Delta t \to 0$.
For the third term, we have
\begin{equation*}
\left|\mu\int_0^T(\Grad\tu_h,\Grad(v-\Ih^u v))dt\right|\leq \mu \norm{\Grad\tu_h}_{L^2([0,T]\times\dom)}\norm{\Grad(v-\Ih^u v)}_{L^2([0,T]\times\dom)}\leq C h\norm{v}_{L^2(0,T;H^2)},
\end{equation*}
where we used again interpolation estimates and the discrete energy estimate. For the fourth term, we have,
\begin{equation*}
\begin{split}
&\left|\int_0^T(\sigma(\bar{Q}_h,\HH_h),\Grad(v-\Ih^u v)) dt\right|  	\leq C \int_0^T\int_{\dom}\left[|\bar{Q}_h| |\HH_h|+|\HH_h|+|\bar{Q}_h|^2 |\HH_h|\right]|\Grad(v-\Ih^u v)| dt\\
&\qquad\leq C \int_0^T\left[
\norm{\bar{Q}_h}_{L^3}\norm{\HH_h}_{L^2}\norm{\Grad(v-\Ih^u v)}_{L^6} + \norm{\HH_h}_{L^2}\norm{\Grad(v-\Ih^u v)}_{L^2}+ \norm{\HH_h}_{L^2}\norm{\bar{Q}_h}_{L^6}^2\norm{\Grad(v-\Ih^u v)}_{L^6} \right]dt\\
&\qquad\leq C \int_0^T\left[
\norm{\bar{Q}_h}_{H^1}\norm{\HH_h}_{L^2}\norm{\Grad(v-\Ih^u v)}_{L^6} + \norm{\HH_h}_{L^2}\norm{\Grad(v-\Ih^u v)}_{L^2}+ \norm{\HH_h}_{L^2}\norm{\bar{Q}_h}_{H^1}^2\norm{\Grad(v-\Ih^u v)}_{L^6} \right]dt\\
&\qquad\leq C h \int_0^T\left[
\norm{\bar{Q}_h}_{H^1}\norm{\HH_h}_{L^2}\norm{v}_{W^{2,6}} + \norm{\HH_h}_{L^2}\norm{v}_{H^2}+ \norm{\HH_h}_{L^2}\norm{\bar{Q}_h}_{H^1}^2\norm{v}_{W^{2,6}} \right]dt\\
&\qquad\leq C h \Big[
\norm{\bar{Q}_h}_{L^\infty(0,T;H^1(\dom))}\norm{\HH_h}_{L^2([0,T]\times\dom)}\norm{v}_{L^2(0,T;W^{2,6}(\dom))} + \norm{\HH_h}_{L^2([0,T]\times\dom)}\norm{v}_{L^2(0,T;H^2(\dom))}\\
&\qquad\qquad+ \norm{\HH_h}_{L^2([0,T]\times\dom)}\norm{\bar{Q}_h}_{L^\infty(0,T;H^1(\dom))}^2\norm{v}_{L^2(0,T;W^{2,6}(\dom))} \Big]\\
&\qquad\leq C h\norm{v}_{L^2(0,T;W^{2,6}(\dom))}\stackrel{h,\Delta t \to 0}{\longrightarrow} 0,
\end{split}
\end{equation*}
	For the fifth term on the right-hand side of~\eqref{eq:uapp}, we have with~\cite[Theorem 11.13]{Ern2021}
	\begin{equation*}
	\begin{split}
	\left|\int_0^T(\HH_h\Grad\bar{Q}_h,v-\Ih^u v)dt\right| &\leq \int_0^T\norm{\HH_h}_{L^2}\norm{\Grad \bar{Q}_h}_{L^2}\norm{\Ih^u v-v}_{L^\infty} dt  \\
	&\leq C h \norm{\mathcal{H}_h}_{L^2([0,T]\times\dom)}\norm{ Q_h}_{L^\infty(0,T;H^1(\dom))}\norm{v}_{L^2(0,T;W^{1,\infty}(\dom))}  \\
&\leq C h \norm{v}_{L^2(0,T;W^{1,\infty}(\dom))}\stackrel{h,\Delta t\to 0}{\longrightarrow} 0.
	\end{split}
	\end{equation*}
	For the sixth term, we have
	\begin{equation}\label{eq:pressureestimate2}
	\begin{split}
	\left|\int_0^T(\Grad p_h,\Ih^u v-v) dt\right|&\leq  \norm{\Grad p_h}_{L^2([0,T]\times \dom)}\norm{ \Ih^u v-v}_{L^2([0,T]\times\dom)}\\
	&\leq  C h^{k+1}\norm{\Grad p_h}_{L^2([0,T]\times \dom)}\norm{v}_{L^2(0,T;H^{k+1}(\dom))}\\
	& \leq C \frac{h^{k+1}}{\Delta t}\norm{v}_{{L^2(0,T;H^{k+1}(\dom))}}
	\end{split}
	\end{equation}
	similar to the estimate~\eqref{eq:pressuretimecont}.  Clearly, also this term goes to zero if $h^{k+1}=o(\Delta t)$. 
	Finally, for the seventh term, we have, 
	\begin{equation*}
		\left|\int_0^T(f_h,v-\Ih^u v) dt\right|\leq \norm{f_h}_{L^2([0,T]\times\dom)}\norm{v-\Ih^u v}_{L^2([0,T]\times\dom)}\stackrel{h,\Delta t\to 0}{\longrightarrow} 0.
	\end{equation*}
	Thus we obtain in the limit
	\begin{equation}
	\label{eq:uprelimlimit}
	\int_0^T\left[-\left( u ,\partial_t v\right)+ b(u,u,v)+\mu(\Grad u,\Grad v) +(\sigma(Q,\HH),\Grad v)+({\HH\Grad Q},v)-(f,v)\right] dt = 0,
	\end{equation}
	as desired.

	It remains to consider the equation for $r_h$, i.e.,~\eqref{eq:rapp}. 
	For the first term on the left-hand side, we change the integration variable to get the time difference to be on the test function $w$:
	\begin{equation*}
	\int_0^T\left(D_t^- r_h,\Ih w\right)_h dt = 	-\int_0^T\left(r_h,D_t^+\Ih w\right)_h dt
		\end{equation*} 
	Next, we study its limit in a similar way as for the third term on the left-hand side of~\eqref{eq:Happ}:
	\begin{align*}
	&\lim_{h,\Delta t\to 0}\left|\int_0^T\left(D_t^- r_h,\Ih w\right)_h dt+ \int_0^T(r,\partial_t w)dt \right|\\
	&\quad \leq  \lim_{h,\Delta t\to 0}\underbrace{\left|\int_0^T \left[\left(r_h,D_t^+\Ih  w\right)_h - \left(r_h,D_t^+\Ih  w\right)\right] dt\right|}_{\textrm{I}}\\
	&\qquad +\lim_{h,\Delta t\to 0}\underbrace{\left|\int_0^T \left[\left(r_h,D_t^+\Ih w\right) - \left(r_h,D_t^+ w\right)\right] dt\right|}_{\textrm{II}}\\
	&\qquad +\lim_{h,\Delta t\to 0}\underbrace{\left|\int_0^T \left[\left(r_h,D_t^+w\right) - \left(r,\partial_t w\right)\right] dt\right|}_{\textrm{III}}. 
	\end{align*}
	For the first term, we use the properties of the discrete inner product, Lemma~\ref{lem:masslumped}, Lemma~\ref{lem:masslumpedLp}, the properties of the  interpolation operator $\Ih$ and the discrete energy estimate:
	\begin{equation*}
	\textrm{I}\leq C h \norm{r_h}_{L^\infty(0,T;L^2(\dom))}\norm{\Grad\left(D_t^+\Ih w\right)}_{L^1([0,T];L^2(\dom))}\leq C h \norm{\partial_t  w}_{L^1([0,T];H^2(\dom))}.	
	\end{equation*}
	For the second term, we use    interpolation estimates  and the discrete energy estimate:
	\begin{equation*}
	\textrm{II}\leq C h \norm{r_h}_{L^\infty(0,T;L^2(\dom))}\norm{ D_t^+ w}_{L^2([0,T];H^2(\dom))}\leq C h \norm{\partial_t  w}_{L^2([0,T];H^2(\dom))}.
	\end{equation*}
	The third term $\textrm{III}$ converges to zero thanks to the weak star convergence of $r_h$ and the strong convergence $ D_t^+ w \to \partial_t w$ almost everywhere and the regularity of $w$. 
	Next, let us consider the second term on the left-hand side of~\eqref{eq:rapp}:
	\begin{align*}
	& \left|\int_0^T (P(\bar{Q}_h):D^-_t Q_h,\Ih w)_h dt-\int_0^T(P(Q):\partial_t Q, w) dt\right|\\
	&\quad \leq \underbrace{ \left|\int_0^T \left[ (P(\bar{Q}_h):D^-_t Q_h,\Ih w)_h -  (\Ih(P(\bar{Q}_h):D^-_t Q_h),\Ih w) \right]dt\right|}_{\text{I}}\\
	&\qquad + \underbrace{ \left|\int_0^T \left[ (\Ih(P(\bar{Q}_h):D^-_t Q_h),\Ih w)  -  (\Ih(P(\bar{Q}_h):D^-_t Q_h),  w) \right]dt\right|}_{\text{II}}\\
	&\qquad + \underbrace{ \left|\int_0^T \left[ (P(\bar{Q}_h):D^-_t Q_h, w)  -  (\Ih(P(\bar{Q}_h):D^-_t Q_h),  w) \right]dt\right|}_{\text{III}}\\
	&\qquad + \underbrace{ \left|\int_0^T \left[ (P(\bar{Q}_h):D^-_t Q_h, w)  -  (P(Q):\partial_t Q,  w) \right]dt\right|}_{\text{IV}}.
	\end{align*}
	We show that each of the four terms on the right-hand side converges to zero:
	For the first term I, we use the properties of the discrete inner product once more, Lemma~\ref{lem:masslumped} and~\ref{lem:masslumpedLp}, the Lipschitz continuity of $P$, and the properties of the interpolation operator $\Ih$, and the estimate on the time difference of $Q_h$,~\eqref{eq:timederQ}:
\begin{align*}
		\text{I}& \leq Ch \norm{\Ih(P(\bar{Q}_h):D_t^- Q_h)}_{L^1}\norm{\Grad \Ih w}_{L^\infty}\\
		& \leq Ch \norm{\bar{Q}_h}_{L^6}\norm{D_t^- Q_h}_{L^{6/5}}\norm{w}_{L^\infty(0,T;W^{1,\infty}(\dom))}\\
		&\leq C h\norm{w}_{L^\infty(0,T;W^{1,\infty}(\dom))}.
	\end{align*}
	Clearly, this vanishes as $h\to 0$.
	For the second term II, we again use interpolation estimates, Lemma~\ref{lem:masslumpedLp}, the Lipschitz continuity of $P$, and the discrete energy estimate and~\eqref{eq:timederQ},
	\begin{align*}
	\text{II}& \leq Ch \norm{\bar{Q}_h}_{L^6}\norm{D_t^- Q_h}_{L^{6/5}}\norm{w}_{L^\infty(0,T;W^{1,\infty}(\dom))}\\
	&\leq C h\norm{w}_{L^\infty(0,T;W^{1,\infty}(\dom))}.
	\end{align*}
	To estimate term III, we first note that, by the energy estimate and~\eqref{eq:timederQ},
	\begin{align*}
		\norm{D_t^- Q_h}_{L^2([0,T]\times\dom)}^2& \leq \int_0^T \norm{D_t^- Q_h}_{L^{6/5}}\norm{D_t^- Q_h}_{L^6} dt \\
		& \leq C\norm{D_t^- Q_h}_{L^2([0,T];L^{6/5}(\dom))}\norm{D_t^-  Q_h}_{L^2(0,T;H^1(\dom))} \\
		& \leq \frac{C}{\Delta t}\norm{D_t^- Q_h}_{L^2([0,T];L^{6/5}(\dom))}\norm{ Q_h}_{L^\infty(0,T;H^1(\dom))}\\
			& \leq \frac{C}{\Delta t}.
	\end{align*}
	Using this, and  Lemmas~\ref{lem:masslumped},~\ref{lem:masslumpedLp},~\ref{lem:masslump2}, and the Lipschitz continuity of $P$, we get
\begin{align*}
	\text{III}& \leq  \norm{P(\bar{Q}_h):D_t^- Q_h - \Ih(P(\bar{Q}_h): D_t^- Q_h)}_{L^1}\norm{w}_{L^\infty}\\
	& \leq C h \norm{\Grad \bar{Q}_h}_{L^2}\norm{D_t^- Q_h}_{L^2}\norm{w}_{L^\infty}\\
	& \leq C \frac{h}{\sqrt{\Delta t}} \norm{ \Grad\bar{Q}_h}_{L^\infty(0,T;L^2(\dom))}\norm{w}_{L^\infty}.
\end{align*}
The terms on the right-hand side are bounded thanks to the energy inequality. Under the condition $h^2 = o(\Delta t)$, it converges to zero.
Finally, the fourth term IV converges to zero, since $D^-_t Q_h$ converges weakly in $L^2(0,T;L^{6/5}(\dom))$ by the uniform bound~\eqref{eq:timederQ} and  $\bar{Q}_h$ converges strongly in $L^2(0,T;H^1(\dom))$, as shown previously, which implies strong convergence in $L^2(0,T;L^6(\dom))$ by the continuous embedding of $H^1$ into $L^6$.
	Thus in the limit,
	\begin{equation}\label{eq:rprelimlimit}
	\int_0^T\left[\left(r,\partial_t w\right)  + (P(Q):\partial_t Q, w)\right] dt=0.
	\end{equation}
	Combining \eqref{eq:Qlim}, \eqref{eq:Hprelimlimit},~\eqref{eq:uprelimlimit}, and~\eqref{eq:rprelimlimit}, we see that $(u,Q,\mathcal{H},r)$ satisfies the distributional version of~\eqref{seq:BerisEdwardsIEQ}. Using Lemma 5.2 from~\cite{Gudibanda2022} about the equivalence of~\eqref{seq:BerisEdwards} and~\eqref{seq:BerisEdwardsIEQ}, we also obtain that $(u,Q,\HH)$ satisfies the distributional form of~\eqref{seq:BerisEdwards} stated in Definition~\ref{def:weaksol}. In~\cite{Gudibanda2022}, the time derivative of $Q$ satisfies $\partial_t Q\in L^2([0,T]\times\dom)$, whereas here we only have $\partial_t Q\in L^2(0,T;L^{6/5}(\dom))$, however, going through the proof of that lemma, one can see that the precise exponent of the spatial regularity does not matter, and the proof goes through with  $\partial_t Q\in L^2(0,T;L^{6/5}(\dom))$. The spatial regularity stated in~\eqref{eq:regularity} follows from the uniform bounds in the discrete energy estimate, Lemma~\ref{lem:discenergyestimate}. The time regularity for $Q$ follows from the uniform bound for the approximations in~\eqref{eq:timederQ}. The time regularity for $u$ follows from the density of smooth test functions in $L^2(0,T;W^{1,6}_0(\dom))$ and checking that~\eqref{eq:uprelimlimit} is valid for test functions in $L^2(0,T;X_6)$ due to the available integrability for all the terms appearing. 
	
			Next, we reconsider the discrete energy balance~\eqref{eq:prelimenergy}. Adding $2\Delta t$ times~\eqref{eq:totalenergyfirststep} to it and using~\eqref{eq:u0estimate}, we obtain
			
				\begin{multline}\label{eq:prelimlimitenergy}
				E^N_h + \sum_{m=1}^{N-1}\norm{u^{m+1}_h-2u^m_h+ u^{m-1}_h}_{L^2}^2+3 \sum_{m=1}^{N-1}\norm{\tu^{m+1}_h-u^{m+1}_h}_{L^2}^2 + 4\mu\Delta t \sum_{m=0}^{N-1}\norm{\Grad\tu_h^{m+1}}_{L^2}^2\\
				+ 4M\Delta t \sum_{m=0}^{N-1}\norm{\HH_h^{m+1}}_{L^2}^2 + L\sum_{m=1}^{N-1}\norm{\Grad Q_h^{m+1}-2\Grad Q^m_h+\Grad Q^{m-1}_h}_{L^2}^2+\sum_{m=1}^{N-1}\norm{r_h^{m+1}-2r^m_h+r^{m-1}_h}_{h}^2\\
				=  \norm{u_h^1}^2_{L^2}+\norm{2 u_h^1-u_h^0}_{L^2}^2 + L \norm{\Grad Q_h^1}_{L^2}^2+L\norm{2\Grad Q_h^1-\Grad Q_h^0}_{L^2}^2 \\
				+ \norm{r_h^1}_{h}^2 + \norm{2r_h^1 - r_h^0}^2_h + \frac{4\Delta t^2}{3}\norm{\Grad p_h^1}_{L^2}^2
				 +4 \Delta t \sum_{m=0}^{N-1}(f^{m+1},\tu_h^{m+1})\\
				- 2 \norm{u^1_h}_{L^2}^2 -2\norm{\tu^1_h-u^0_h}_{L^2}^2  + 2\norm{u^0_h}_{L^2}^2 
			 - 	2 \Delta t^2\norm{\Grad p^1_h}_{L^2}^2 +2 \Delta t^2\norm{\Grad p^0_h}_{L^2}^2\\
				-2\norm{r_h^1}_h^2 + 2 \norm{r_h^0}_h^2 - 2 \norm{r_h^1 - r_h^0}_h^2 
			-2L\norm{\Grad Q_h^1}_{L^2}^2 - 2 L \norm{\Grad (Q_h^1-Q^0_h)}_{L^2}^2 + 2 L \norm{\Grad Q_h^0}_{L^2}^2\\
			\leq 
			-\norm{u_h^1}^2_{L^2}+\norm{2 u_h^1-u_h^0}_{L^2}^2- L \norm{\Grad Q_h^1}_{L^2}^2+L\norm{2\Grad Q_h^1-\Grad Q_h^0}_{L^2}^2 \\
			- \norm{r_h^1}_{h}^2 + \norm{2r_h^1 - r_h^0}^2_h 
			+4 \Delta t \sum_{m=0}^{N-1}(f^{m+1},\tu_h^{m+1})
			+ 2\norm{u_0}_{L^2}^2 
			+ 2 \norm{r^0_h}_{h}^2
		 + 2 L \norm{\Grad Q_h^0}_{L^2}^2.
			\end{multline}			
	Rewriting this in terms of the piecewise constant functions $u_h$, $\bar{u}_h$, $\tu_h$ etc. and
	passing to the limit on the left-hand side and using Lemma~\ref{lem:weakapproxcontinuity} together with weak lower semi-continuity of the norm, we have for  every $t\in [0,T]$,
	\begin{multline}
		2\norm{u(t)}_{L^2}^2 + 4\mu \int_0^t \norm{\Grad u(\tau)}_{L^2}^2 d\tau+2 L\norm{\Grad Q(t)}^2_{L^2}+ 4 M \int_0^t\norm{\HH(\tau)}_{L^2}^2 d\tau +2\norm{r(t)}^2_{L^2}\\
		\leq \liminf_{h,\Delta t\to 0}\bigg(E_h(t)  	+\Delta t^{-1}\int_{\Delta t}^t\norm{u_h(\tau)-\bar{u}_h(\tau)}_{L^2}^2 d\tau + 3\Delta t^{-1}\int_{\Delta t}^t\norm{\tu_h(\tau)-u_h(\tau)}_{L^2}^2 d\tau
		+ 4 \mu\int_0^t\norm{\Grad \tu_h(\tau)}_{L^2}^2d\tau\\
		+ 4M \int_0^t \norm{\HH_h(\tau)}_{L^2}^2 d\tau
		+L \Delta t^{-1} \int_{\Delta t}^t \norm{\Grad Q_h(\tau) - \Grad \bar{Q}_h(\tau)}_{L^2}^2 d\tau + \Delta t^{-1} \int_{\Delta t}^t \norm{r_h(\tau) -2r_h(\tau-\Delta t)+r_h(\tau-2\Delta t) }^2_{h}d\tau \bigg).
	\end{multline}
	Let us consider the terms on the right-hand side. For the term involving $f$, we have
	\begin{equation*}
		4\int_0^t (f_h,\tu_h) d\tau \stackrel{\Delta t,h\to 0}{\longrightarrow} 4 \int_0^t (f,u) d\tau,
	\end{equation*}
	using the strong convergence of $\tu_h$ and $f_h$ in $L^2([0,T]\times\dom)$. We rewrite the terms involving approximations of $u$ as follows:
	\begin{equation*}
		-\norm{u_h^1}_{L^2}^2 + \norm{2 u_h^1 - u_h^0}_{L^2}^2 + 2 \norm{u_0}_{L^2}^2 = 3 \norm{u_h^1}_{L^2}^2 +\norm{u_h^0}_{L^2}^2 - 4(u_h^1,u_h^0)+ 2 \norm{u_0}_{L^2}^2. 
	\end{equation*}
	We have $u_h^0\to u_0$ strongly in $L^2(\dom)$ and $u_h^1\weak u_0$ in $L^2(\dom)$. Therefore
	\begin{equation*}
		\lim_{h,\Delta t\to 0}\left(	\norm{u_h^0}_{L^2}^2 - 4(u_h^1,u_h^0)+ 2 \norm{u_0}_{L^2}^2 \right) = -\norm{u_0}_{L^2}^2.
	\end{equation*}
	Similarly, 
		\begin{equation*}
		\lim_{h,\Delta t\to 0}\left(	\norm{\Grad Q_h^0}_{L^2}^2 - 4(\Grad Q_h^1,\Grad Q_h^0)+ 2 \norm{\Grad Q_0}_{L^2}^2 \right) = -\norm{\Grad Q_0}_{L^2}^2.
	\end{equation*}
	Next, we need to show
		\begin{equation}\label{eq:rlimitinitial}
		\lim_{h,\Delta t\to 0}\left(	3\norm{r_h^0}_{h}^2 - 4(r_h^1,r_h^0)_h \right) = -\norm{r_0}_{L^2}^2.
	\end{equation}
	To do so, we first recall that 
	\begin{equation*}
		\norm{r_0}_{L^2}^2 = 2\int_{\dom}\!\! (\mathcal{F}_B(Q_0)+A_0) dx,\quad \text{and}\quad \norm{r_h^0}_h^2 = 2\sum_{z\in \mathcal{N}_h} (\mathcal{F}_B(Q_h^0(z))+A_0)\!\!\int_{\dom}\varphi_z dx = 2\int_{\dom}\!\! (\Ih(\mathcal{F}_B(Q_h^0))+A_0) dx.
	\end{equation*}
	Therefore
	\begin{align*}
		\left|\norm{r_0}_{L^2}^2 - \norm{r_h^0}_h^2\right|& \leq 2\left| \int_{\dom} \mathcal{F}_B(Q_0)-\mathcal{F}_B(Q_h^0) dx\right| + 2\left|\int_{\dom} \mathcal{F}_B(Q_h^0) - \Ih(\mathcal{F}_B(Q_h^0)) dx\right|\\
		& \leq C\left(1+\norm{Q_0}^3_{L^4}+\norm{Q_h^0}^3_{L^4}  \right) \norm{Q_0-Q_h^0}_{L^4}  + 2\sum_{K\in \mathcal{T}_h}\int_K \left|\mathcal{F}_B(Q_h^0)-\Ih(\mathcal{F}_B(Q_h^0))\right| dx\\
		& \leq C \norm{Q_0-Q_h^0}_{L^4}  + C h^2\sum_{K\in \mathcal{T}_h}\int_K \left|\mathcal{F}''_B(Q_h^0)\right| |\Grad Q_h^0|^2 dx\\
		& \leq C \norm{Q_0-Q_h^0}_{L^4}  + C h^2\int_{\dom} \left(1+\left|Q_h^0\right|^2\right) |\Grad Q_h^0|^2 dx\\
		& \leq C \norm{Q_0-Q_h^0}_{L^4}  + C h^2 \left(1+\norm{Q_h^0}_{L^\infty}^2\right) \norm{\Grad Q_h^0}_{L^2}^2\\
		& \leq C \norm{Q_0-Q_h^0}_{L^4}  + C h^2 \left(1+\norm{Q_h^0}_{L^\infty}^2\right).
	\end{align*}
	where we used that $Q_0\in H^1(\dom)$, interpolation estimates for the Lagrangian interpolation, and that $Q_h^0$ is piecewise linear and therefore its element-wise second derivative is zero. Now since $Q_0\in H^1_0(\dom)$, the first term converges to zero, and for the second term, we use the inverse estimate~\eqref{eq:inverseestimate} to obtain
	\begin{equation}\label{eq:r0conv}
\left|\norm{r_0}_{L^2}^2 - \norm{r_h^0}_h^2\right| \leq  C \norm{Q_0-Q_h^0}_{L^4}  + C h^{2-d/3} \left(1+\norm{Q_h^0}_{H^1}^2\right)\stackrel{h\to 0}{\longrightarrow} 0.
	\end{equation}
	Similarly, we have
	\begin{align*}
		\norm{r_h^0-r_0}_{L^2 }& \leq \norm{\Ih(r(Q_h^0)) -r(Q_h^0)}_{L^2} + \norm{r(Q_h^0)-r(Q_0)}_{L^2}\\
		& \leq C h^2\left(\sum_{K\in \mathcal{T}_h} \int_K\left|\Grad^2 r(Q_h^0)\right|^2 dx\right)^{1/2} +\left(1+ \norm{Q_0}_{L^4} + \norm{Q_h^0}_{L^4}\right) \norm{Q_h^0-Q_0}_{L^4}\\
		& \leq C h^2\left(\sum_{K\in \mathcal{T}_h} \int_K\left| r''(Q_h^0)\right|^2|\Grad Q_h^0|^4\right)^{1/2} +C\norm{Q_h^0-Q_0}_{L^4}\\
		& \leq C h^2\left(\sum_{K\in \mathcal{T}_h} \int_K|\Grad Q_h^0|^4\right)^{1/2} +C\norm{Q_h^0-Q_0}_{L^4}.
			\end{align*}
			Here we used that $|r''(Q)|$ is bounded. Using inverse estimates, we can bound this by
			\begin{equation}\label{eq:initialrconv}
				\norm{r_h^0-r_0}_{L^2 }\leq C h^{2-d/2} \norm{\Grad Q_h^0}_{L^2}^2 +C\norm{Q_h^0-Q_0}_{L^4}\stackrel{h\to 0}{\longrightarrow} 0.
			\end{equation}
	Next, we show that $(r_h^1-r_h^0,r_h^0)_h\to 0$ which together with~\eqref{eq:r0conv} and~\eqref{eq:initialrconv} will imply~\eqref{eq:rlimitinitial}. Using the first step of the scheme,~\eqref{eq:rfirststep}, for a sufficiently smooth test function $w$, we have, using Lemmas~\ref{lem:masslumped} and~\ref{lem:masslumpedLp}, interpolation estimates, and Lemma~\ref{lem:weakapproxcontinuity},
\begin{align*}
		\left|(r_h^1-r_h^0,w)_h\right|& = \left|(P(Q_h^0):(Q_h^1-Q_h^0),w)_h\right|\\
		& \leq \left|(P(Q_h^0):(Q_h^1-Q_h^0),w) \right| + \left|(P(Q_h^0):(Q_h^1-Q_h^0),w)_h-(\Ih(P(Q_h^0):(Q_h^1-Q_h^0)),\Ih w)\right|\\
		&\qquad + \left|(\Ih(P(Q_h^0):(Q_h^1-Q_h^0)),\Ih w)-(P(Q_h^0):(Q_h^1-Q_h^0),\Ih w)\right|\\
	&\qquad + \left|(P(Q_h^0):(Q_h^1-Q_h^0),\Ih w)-(P(Q_h^0):(Q_h^1-Q_h^0), w)\right|\\
		&  \leq \norm{P(Q_h^0)}_{L^6}\norm{Q_h^1-Q_h^0}_{L^{6/5}}\norm{w}_{L^\infty} + C h \norm{\Ih(P(Q_h^0):(Q_h^1-Q_h^0))}_{L^1}\norm{w}_{W^{1,\infty}}\\
		& \qquad + \norm{\Ih(P(Q_h^0):(Q_h^1-Q_h^0)) -P(Q_h^0):(Q_h^1-Q_h^0)}_{L^1}\norm{w}_{L^\infty}\\
		&\qquad + Ch \norm{P(Q_h^0):(Q_h^1-Q_h^0)}_{L^1}\norm{w}_{W^{1,\infty}} \\
		& \leq C \sqrt{\Delta t} \norm{Q_h^0}_{L^6}\norm{w}_{L^\infty}+ C h \norm{Q_h^0}_{L^2}\norm{Q_h^1-Q_h^0}_{L^2}\norm{w}_{W^{1,\infty}}\\
		&\qquad + C h \norm{\Grad Q_h^0}_{L^2}\norm{Q_h^1-Q_h^0}_{L^2}\norm{w}_{L^\infty}+ C h \norm{Q_h^0}_{L^2}\norm{Q_h^1-Q_h^0}_{L^2}\norm{w}_{W^{1,\infty}}\\
		& \leq C(h+\sqrt{\Delta t})\norm{w}_{W^{1,\infty}}.
		\end{align*}
	Thus,
		\begin{equation}\label{eq:smoothw}
				\left|(r_h^1-r_h^0,w)_h\right| \leq C (\sqrt{\Delta t}+ h)\norm{w}_{W^{1,\infty}},
		\end{equation}
		which goes to zero as $h,\Delta t\to 0$. 
		Now let $r_0^\delta$, $\delta>0$, be a smooth approximation of the initial data $r_0$ which satisfies $\norm{r_0^\delta- r_0}_{L^2} \leq  \delta$.
	Then given $\epsilon>0$ arbitrary, we have using~\eqref{eq:smoothw},~\eqref{eq:initialrconv} and Lemma~\ref{lem:masslumped},
	\begin{align*}
			\left|(r_h^1-r_h^0,r^0_h)_h\right| &\leq 	\left|(r_h^1-r_h^0,r^0_h-r_0^\epsilon)_h\right| +	\left|(r_h^1-r_h^0,r_0^\epsilon)_h\right| \\
			& \leq \norm{r_h^1-r_h^0}_h\norm{r_h^0-r_0^\epsilon}_h + C (\sqrt{\Delta t}+ h)\norm{r_0^\epsilon}_{W^{1,\infty}}\\
			& \leq \sqrt{d+2}\norm{r_h^1-r_h^0}_h\left(\norm{r_h^0-r_0^\epsilon}_{L^2} +\norm{r_0^\epsilon-\Ih(r_0^\epsilon)}_{L^2}\right) + C (\sqrt{\Delta t}+ h)\norm{r_0^\epsilon}_{W^{1,\infty}}\\
			& \leq \sqrt{d+2}\norm{r_h^1-r_h^0}_h\left(\norm{r_h^0-r_0}_{L^2}+ \norm{r_0-r_0^\epsilon}_{L^2}+\norm{r_0^\epsilon-\Ih(r_0^\epsilon)}_{L^2}\right) + C (\sqrt{\Delta t}+ h)\norm{r_0^\epsilon}_{W^{1,\infty}}\\
				& \leq \sqrt{d+2}\norm{r_h^1-r_h^0}_h\left(\norm{r_h^0-r_0}_{L^2}+ \epsilon+ C h^2 \norm{r_0^\epsilon}_{H^2}\right) + C (\sqrt{\Delta t}+ h)\norm{r_0^\epsilon}_{W^{1,\infty}}.
	\end{align*}
	Sending $h,\Delta t\to 0$, $\norm{r_h^0-r_0}_{L^2}\to 0$ by~\eqref{eq:initialrconv}, and so we obtain
	\begin{equation*}
	\limsup_{h,\Delta t \to 0}	\left|(r_h^1-r_h^0,r_h^0)_h\right| \leq C \epsilon,
	\end{equation*}
	where $\epsilon>0$ was arbitrary. Thus~\eqref{eq:rlimitinitial}  follows.
	Next, noticing that $2\Delta t |(f^1, \tilde u_h^1)| \leq \Delta t \|f^1\|_{L^2}^2 + \Delta t \|\tilde u_h^1\|_{L^2}^2 \to 0$  as $\Delta t \to 0$, we pass to the limit in $\Delta t$ times equation~\eqref{eq:totalenergyfirststep}, to obtain
	\begin{equation*}
		\limsup_{h,\Delta t\to 0}\left(\norm{u_h^1}_{L^2}^2+L\norm{\Grad Q_h^1}_{L^2}^2 + \norm{r_h^1}_{h}^2 \right)\leq \norm{u_0}_{L^2}^2+L\norm{\Grad Q_0}_{L^2}^2+\norm{r_0}_{L^2}^2.
	\end{equation*}
	Combining these estimates, we obtain that the right-hand side of~\eqref{eq:prelimlimitenergy} satisfies
	\begin{multline*}
		\limsup_{h,\Delta t\to 0}\bigg( 	-\norm{u_h^1}^2_{L^2}+\norm{2 u_h^1-u_h^0}_{L^2}^2- L \norm{\Grad Q_h^1}_{L^2}^2+L\norm{2\Grad Q_h^1-\Grad Q_h^0}_{L^2}^2 \\
		- \norm{r_h^1}_{h}^2 + \norm{2r_h^1 - r_h^0}^2_h 
		+4 \Delta t \sum_{m=0}^{N-1}(f^{m+1},\tu_h^{m+1})
		+ 2\norm{u_0}_{L^2}^2 
		+ 2 \norm{r^0_h}_{h}^2
	 + 2 L \norm{\Grad Q_0}_{L^2}^2\bigg) \\
		\leq 2 \norm{u_0}_{L^2}^2+2L\norm{\Grad Q_0}_{L^2}^2+2\norm{r_0}_{L^2}^2 + 4 \int_0^t (f,u)d\tau, 
	\end{multline*}
	which implies that 
	the limit $(u,Q,r,\HH)$ satisfies for every $t\in [0,T]$,
	\begin{multline}\label{eq:energylimit}
		\norm{u(t)}_{L^2(\dom)}^2 +L\norm{\Grad Q(t)}_{L^2(\dom)}^2+\norm{r(t)}_{L^2}^2 +2\mu\int_0^t\norm{\Grad u(\tau)}_{L^2(\dom)}^2 d\tau+2M\int_0^t\norm{\HH(\tau)}_{L^2}^2 d\tau \\
		 \leq \norm{u_0}_{L^2(\dom)}^2 + L\norm{\Grad Q_0}_{L^2(\dom)}^2 + \norm{r_0}_{L^2(\dom)}^2 + 2\int_0^t(f(\tau),u(\tau)) d\tau,
	\end{multline}
	which is the energy inequality~\eqref{eq:energyineq} after using that $r=r(Q)$. 
	
	Now the continuity at zero will follow. We write
	\begin{equation*}
		\norm{u(t)-u_0}_{L^2(\dom)}^2 = \norm{u(t)}_{L^2(\dom)}^2 + \norm{u_0}_{L^2(\dom)}^2 - 2 (u(t),u_0).
	\end{equation*}
	Since $u$ is weakly continuous in time with values in $L^2_{\Div}(\dom)$, i.e., the mapping
	\begin{equation*}
		t\mapsto (u(t),v),
	\end{equation*}
	is continuous for any $v\in L^2_{\Div}(\dom)$ (see Lemma~\ref{lem:convofutilde}), we have that
	\begin{equation*}
		\lim_{t\to 0} (u(t),u_0) = \norm{u_0}_{L^2(\dom)}^2.
	\end{equation*}
	Similarly, $t\mapsto (r(t),w)$ is continuous for $w\in L^2(\dom)$ (see Remark~\ref{rem:continuityr}), and so we have
	\begin{equation*}
		\lim_{t\to 0} (r(t),r_0) = \norm{r_0}_{L^2}^2.
	\end{equation*}
	For the variable $Q$, we observe that since $\partial_t Q\in L^2([0,T];L^{6/5}(\dom))$, we have that $\partial_t \Grad Q\in L^2([0,T],(W^{1,6}(\dom))^*)$, and since $\Grad Q \in L^\infty(0,T;L^2(\dom))$, we obtain with Lemma II.5.9 in~\cite{Boyer2013} that $t\mapsto (\Grad Q(t),Y)$ is continuous for all $Y\in L^2(\dom)$. Thus
		\begin{equation*}
	\lim_{t\to 0} (\Grad Q(t),\Grad Q_0) = \norm{\Grad Q_0}_{L^2}^2.
	\end{equation*}
	Furthermore, using weak lower semi-continuity of the $L^2$-norm and that $u,\Grad Q,r\in C(0,T;L^2(\dom)_w)$, we obtain that
	\begin{equation*}
		\norm{u_0}_{L^2(\dom)}^2 \leq \liminf_{t\to 0}\norm{u(t)}_{L^2(\dom)}^2, 	\quad\norm{\Grad Q_0}_{L^2(\dom)}^2 \leq \liminf_{t\to 0}\norm{\Grad Q(t)}_{L^2(\dom)}^2,\quad 	\norm{r_0}_{L^2(\dom)}^2 \leq \liminf_{t\to 0}\norm{r(t)}_{L^2(\dom)}^2. 
	\end{equation*}
	On the other hand, sending $t\to 0$ in the energy inequality~\eqref{eq:energylimit}, we obtain that
	\begin{equation*}
		\limsup_{t\to 0}\left(\norm{u(t)}_{L^2(\dom)}^2+\norm{r(t)}_{L^2}^2+L\norm{\Grad Q(t)}_{L^2}^2 \right)\leq \norm{u_0}_{L^2(\dom)}^2+\norm{r_0}_{L^2}^2 + L \norm{\Grad Q_0}_{L^2}^2.
	\end{equation*}
	Thus
	\begin{equation*}
		\lim_{t\to 0}\left( \norm{u(t)}_{L^2(\dom)}^2+ L\norm{\Grad Q(t)}_{L^2(\dom)}^2+ \norm{r(t)}_{L^2(\dom)}^2\right) = \norm{u_0}_{L^2(\dom)}^2 + L\norm{\Grad Q_0}_{L^2(\dom)}^2+ \norm{r_0}_{L^2(\dom)}^2
	\end{equation*}
	and we obtain that (using the Poincar\'e inequality)
	\begin{equation*}
		\lim_{t\to 0}\norm{u(t)-u_0}_{L^2(\dom)}^2 = 0,\quad \lim_{t\to 0}\norm{Q(t)-Q_0}_{H^1(\dom)}^2 = 0,\quad \lim_{t\to 0}\norm{r(t)-r_0}_{L^2(\dom)}^2 = 0.
	\end{equation*}
	Furthermore, by Lemma~\ref{lem:tracefreediscrete}, $\HH_h$, $Q_h$ and $\bar{Q}_h$ are trace-free and symmetric, therefore, using the self-adjointness of $\Pi$ (cf. Lemma~\ref{lem:propertiesofPi} (a)), we have
	\begin{equation*}
	\int_{0}^T(\HH,Z) dt=	\lim_{h,\Delta t\to 0}\int_{0}^T(\HH_h,Z) dt = 	\lim_{h,\Delta t\to 0}\int_{0}^T(\HH_h,\Pi Z) dt =  \int_{0}^T(\HH,\Pi Z) dt = \int_{0}^T(\Pi \HH,Z) dt, 
	\end{equation*} 
	i.e., $\HH$ is symmetric and trace-free almost everywhere. In the same way it follows that $Q$ is symmetric and trace-free.
	We conclude that $(u,Q,\HH)$ is a weak solution of~\eqref{seq:BerisEdwards} as in Definition~\ref{def:weaksol}.
\end{proof}

\section{Computational results}\label{sec:computations}
In this section we illustrate the performance of our numerical scheme presented in \autoref{sec:spatial-discretization} in 2D domains.
Our implementation utilizes the open-source finite element library \texttt{Firedrake} \cite{FiredrakeUserManual}, and can be retrieved at \url{https://github.com/gobenavides/BDF2-Chorin-BerisEdwards}.
For all experiments, we consider the lowest-order Taylor--Hood element $\mathrm{P}^2$-$\mathrm{P}^1$ for the velocity-pressure pair $(\Uh,\Ph)$, the lowest-order Lagrange finite element space in its tensor- and scalar-valued versions for, respectively, $\Mh$ and $\Xh$.
We strongly enforce the symmetry of the Q-tensor $Q$ and molecular field $\HH$ on the space $\Mh$.
Recall that $\Yh = \Uh + \nabla \Ph$.
For computational simplicity, we utilize the Lagrange interpolant for the initialization of the discrete Q-tensor $Q$ in lieu of the Scott--Zhang operator (cf.~\eqref{eq:initdataapprox}).
Also, since the source terms $f$ we consider are smooth, we utilize pointwise in-time evaluation instead of local averages \eqref{eq:smSigmamsemi};
this is justified as $f(t) - \frac{1}{\Delta t} \int_{t-\frac{\Delta t}{2}}^{t+\frac{\Delta t}{2}} f(\tau) d\tau = \mathcal{O}((\Delta t)^2)$, which is, formally, of the same order as the scheme \eqref{eq:step1fully}--\eqref{eq:projectionfullydiscrete}.
Finally, we mention that non-homogenous Dirichlet conditions are implemented in the traditional way.

Recall that a Q-tensor $Q$ is said to be \emph{uniaxial} if it can be expressed as $Q = S(m \otimes m - \frac{1}{d}I)$, where the unit-length \emph{director} field $m$ represents the average orientation of the liquid crystal particles, and the \emph{order parameter} $S$ measures the anisotropy of the liquid crystal, with $S=0$ representing the isotropic state.
In the bi-dimensional case (i.e.~$d=2$), all Q-tensors $Q$ are inherently uniaxial and satisfy $S = 2\sqrt{Q_{11}^2 + Q_{12}^2}$.
In contrast, three-dimensional Q-tensors are allowed to be \emph{biaxial} (i.e.~not uniaxial).
A reasonable representation of the average orientation of the liquid crystal molecules is then given by the leading eigenvector of $Q$.

\subsection{Spatial convergence test}\label{sec:spatial-test}
In this test, we investigate the convergence behavior of the scheme as the spatial parameter $h \to 0$.
For simplicity, we consider modeling parameters $a=1$, $b=0$, $c=1$, $A_0 = 1$, $\mu=1$, $\xi = 1$, $M=1$ and $L=1$.
The computational time is $T=1$ and we fix $\Delta t = 1/1000$.
Our computational domain $\Omega = (0,2)^2$ is triangulated with uniform triangles of size $h$.
We consider manufactured solutions
\begin{gather*}
u(t,x,y) = \pi e^{-t} \Big( (1-\cos(\pi x))\sin(\pi y), -(1-\cos(\pi y))\sin(\pi x) \Big), \qquad
p(t,x,y) = (x-1)^3 \sin(2\pi t\, y),\\
Q = m \otimes m - \frac{1}{2} |m|^2 I, \qquad
r \stackrel{\eqref{eq:defr}}{=} r(Q), \qquad
\HH \stackrel{\eqref{eq:defHIEQ}}{=} L \Delta Q - r P(Q)
\end{gather*}
where $m(t,x,y) = \big( \cos(\pi t x)\cos(0.5 \pi y t), \sin(\pi t x)\sin(0.5 \pi t y)  \big)$.
We modify the momentum and Q-tensor kinematic equations with appropriate forcing terms so that the aforementioned functions satisfy the resulting equations.
The non-vanishing Dirichlet condition for the Q-tensor $Q$ is incorporated in the discrete scheme following customary techniques.

For two generic error norms $\|\mathrm{e}(h)\|$ and $\|\mathrm{e}(\hat h)\|$ corresponding to two consecutive mesh sizes $h$ and $\hat h$, the experimental spatial rate of convergence is given by $\log(\|\mathrm{e}(h)\|/\|\mathrm{e}(\hat h)\|) / \log(h / \hat h)$.
We present our numerical results in \autoref{tab:spatial_convergence}.

\begin{table}
    \setlength{\tabcolsep}{3pt}
    \begin{tabular}{c c c c c c c c c c c}
    \toprule
    $h$ & $\|u - u_h\|_{L^2}$ & Rate & $\|p - p_h\|_{L^2}$ & Rate & $\|Q - Q_h\|_{L^2}$ & Rate & $\|\HH - \HH_h\|_{L^2}$ & Rate & $\|r - r_h\|_{L^2}$ & Rate\\
    \midrule
    1.4142 & 8.0187\texttt{E}-01 & --    & 7.6283\texttt{E}+00 & --    & 8.6521\texttt{E}-01 & --    & 2.7176\texttt{E}+01 & --    & 1.8965\texttt{E}-01 & --    \\
    0.7071 & 5.1675\texttt{E}-01 &  0.63 & 7.9748\texttt{E}+00 & -0.06 & 3.9236\texttt{E}-01 &  1.14 & 1.4500\texttt{E}+01 &  0.91 & 6.6628\texttt{E}-02 &  1.51 \\
    0.3536 & 1.2973\texttt{E}-01 &  1.99 & 2.8067\texttt{E}+00 &  1.51 & 1.6783\texttt{E}-01 &  1.23 & 3.9975\texttt{E}+00 &  1.86 & 2.8783\texttt{E}-02 &  1.21 \\
    0.1768 & 3.1450\texttt{E}-02 &  2.04 & 9.6550\texttt{E}-01 &  1.54 & 4.5949\texttt{E}-02 &  1.87 & 8.7308\texttt{E}-01 &  2.19 & 1.0026\texttt{E}-02 &  1.52 \\
    0.0884 & 8.3653\texttt{E}-03 &  1.91 & 2.5794\texttt{E}-01 &  1.90 & 1.1784\texttt{E}-02 &  1.96 & 2.0982\texttt{E}-01 &  2.06 & 2.6346\texttt{E}-03 &  1.93 \\
    0.0442 & 2.1252\texttt{E}-03 &  1.98 & 6.5026\texttt{E}-02 &  1.99 & 2.9633\texttt{E}-03 &  1.99 & 5.1864\texttt{E}-02 &  2.02 & 6.6658\texttt{E}-04 &  1.98 \\
    \bottomrule
\end{tabular}
    \centering
    \caption{Spatial convergence rates in $L^2$ at final time $T=1$.
    The results suggest quadratic convergence behavior.}
    \label{tab:spatial_convergence}
\end{table}

\subsection{Temporal convergence test}

In this test, we study the convergence behavior of the scheme as the time step $\Delta t \to 0$.
We consider a fixed uniform triangulation of $\Omega = (0,2)^2$ of size $h=\sqrt{2}/64$, and the same modeling parameters and manufactured solutions as in \autoref{sec:spatial-test}.
We partition the time interval $(0,1)$ with a uniform time step $\Delta t = (5k-1)^{-1}$, $k=2,\dotsc,7$.

For two generic error norms $\|\mathrm{e}(\Delta t)\|$ and $\|\mathrm{e}(\widehat {\Delta t})\|$ corresponding to two consecutive time steps $\Delta t$ and $\widehat{\Delta t}$, the experimental temporal rate of convergence is given by $\log(\|\mathrm{e}(\Delta t)\|/\|\mathrm{e}(\widehat{\Delta t})\|) / \log(\Delta t / \widehat{\Delta t})$.
We summarize our findings in \autoref{tab:temporal_convergence}.

\begin{table}
    \setlength{\tabcolsep}{3pt}
    \begin{tabular}{c c c c c c c c c c c}
    \toprule
    $\Delta t$ & $\|u - u_h\|_{L^2}$ & Rate & $\|p - p_h\|_{L^2}$ & Rate & $\|Q - Q_h\|_{L^2}$ & Rate & $\|\HH - \HH_h\|_{L^2}$ & Rate & $\|r - r_h\|_{L^2}$ & Rate\\
    \midrule
    0.1111 & 4.5726\texttt{E}-01 & --   & 1.1521\texttt{E}+01 & --   & 1.6107\texttt{E}-01 & --   & 4.6537\texttt{E}+00 & --   & 1.4151\texttt{E}+00 & --   \\
    0.0714 & 1.6563\texttt{E}-01 & 2.30 & 7.7953\texttt{E}+00 & 0.88 & 8.7020\texttt{E}-02 & 1.39 & 1.9398\texttt{E}+00 & 1.98 & 3.3523\texttt{E}-01 & 3.26 \\
    0.0526 & 9.2796\texttt{E}-02 & 1.90 & 4.3173\texttt{E}+00 & 1.93 & 4.2043\texttt{E}-02 & 2.38 & 1.0600\texttt{E}+00 & 1.98 & 1.4672\texttt{E}-01 & 2.71 \\
    0.0417 & 6.0012\texttt{E}-02 & 1.87 & 2.5195\texttt{E}+00 & 2.31 & 2.3711\texttt{E}-02 & 2.45 & 6.6199\texttt{E}-01 & 2.02 & 7.7651\texttt{E}-02 & 2.72 \\
    0.0345 & 4.2064\texttt{E}-02 & 1.88 & 1.6457\texttt{E}+00 & 2.25 & 1.5534\texttt{E}-02 & 2.23 & 4.5413\texttt{E}-01 & 1.99 & 4.5893\texttt{E}-02 & 2.78 \\
    0.0294 & 3.1314\texttt{E}-02 & 1.86 & 1.1751\texttt{E}+00 & 2.12 & 1.1217\texttt{E}-02 & 2.05 & 3.3469\texttt{E}-01 & 1.92 & 2.9354\texttt{E}-02 & 2.81 \\
    \bottomrule
\end{tabular}
    \centering
    \caption{Temporal convergence rates in $L^2$ at final time $T=1$.
    The results suggest quadratic convergence for $p$, $Q$ and $\HH$.
    In turn, $u$ and $r$ seem to converge at rates $\approx 1.8$ and $\approx 2.8$, respectively.}
    \label{tab:temporal_convergence}
\end{table}

\subsection{Simulation: defect dynamics}
It is well-known that liquid crystal defects can travel in time, split, merge or annihilate based on their topological charges (defect degrees).
In this experiment we demonstrate the capability of our scheme of replicating this phenomenon.
Highly inspired by \cite[Section 4.2]{Zhao2017}, we consider a computational domain $\Omega = (0,2)^2$ uniformly triangulated with triangles of size $h=\sqrt{2}/64$, a computational time $T=250$ and uniform time step $\Delta t = 1/500$.
Physical parameters are $a=-0.2$, $b=1$, $c=1$, $A_0 = 500$, $\mu = 1$, $\xi = 0$, $M = 10$ and $L = 10^{-3}$.
We consider the no-slip boundary condition $u = 0$ on $\partial\Omega$ and the non-zero Dirichlet condition $Q = Q_{\text{bc}}$ on $\partial\Omega$ for the Q-tensor $Q$, where
\begin{equation*}
Q_{\text{bc}} = \tilde Q_{\text{bc}} - \frac{\tr(\tilde Q_{\text{bc}})}{d} I, \qquad
\tilde Q_{\text{bc}} = \frac{m_{\text{bc}} \otimes m_{\text{bc}}}{|m_{\text{bc}}|^2 + \varepsilon}, \qquad
m_{\text{bc}}(x,y) = (x-1, y-1),
\end{equation*}
and $\varepsilon := 10^{-10}$ is a regularization parameter to avoid division by $0$.
Regarding initial conditions, we consider $u = 0$, $Q = Q_{\text{in}}$ and $r = r(Q_{\text{in}})$ at $t = 0$, where $Q_{\text{in}}$ is the uniaxial Q-tensor given by
\begin{gather*}
Q_{\text{in}} = \tilde Q_{\text{in}} - \frac{\tr(\tilde Q_{\text{in}})}{d} I, \qquad
\tilde Q_{\text{in}} = \frac{m_{\text{in}} \otimes m_{\text{in}}}{|m_{\text{in}}|^2 + \varepsilon}, \qquad
m_{\text{in}}(x,y) = (x-0.5, y-0.5).
\end{gather*}
Discrete initial conditions are computed according to \eqref{eq:initdataapprox} (with the Scott--Zhang operator replaced by the Lagrange interpolant) and \eqref{eq:zerothstep}, whereas discrete variables at $t = \Delta t$ are determined following Remark \ref{rem:1ststep}.
We report some of our findings in \autoref{fig:exp3-1} and \autoref{fig:exp3-2}, which qualitatively agree with \cite[Section 4.2]{Zhao2017}.

\begin{figure}
    \centering
    \begin{subfigure}[b]{0.3\textwidth}
        \includegraphics[width=\textwidth]{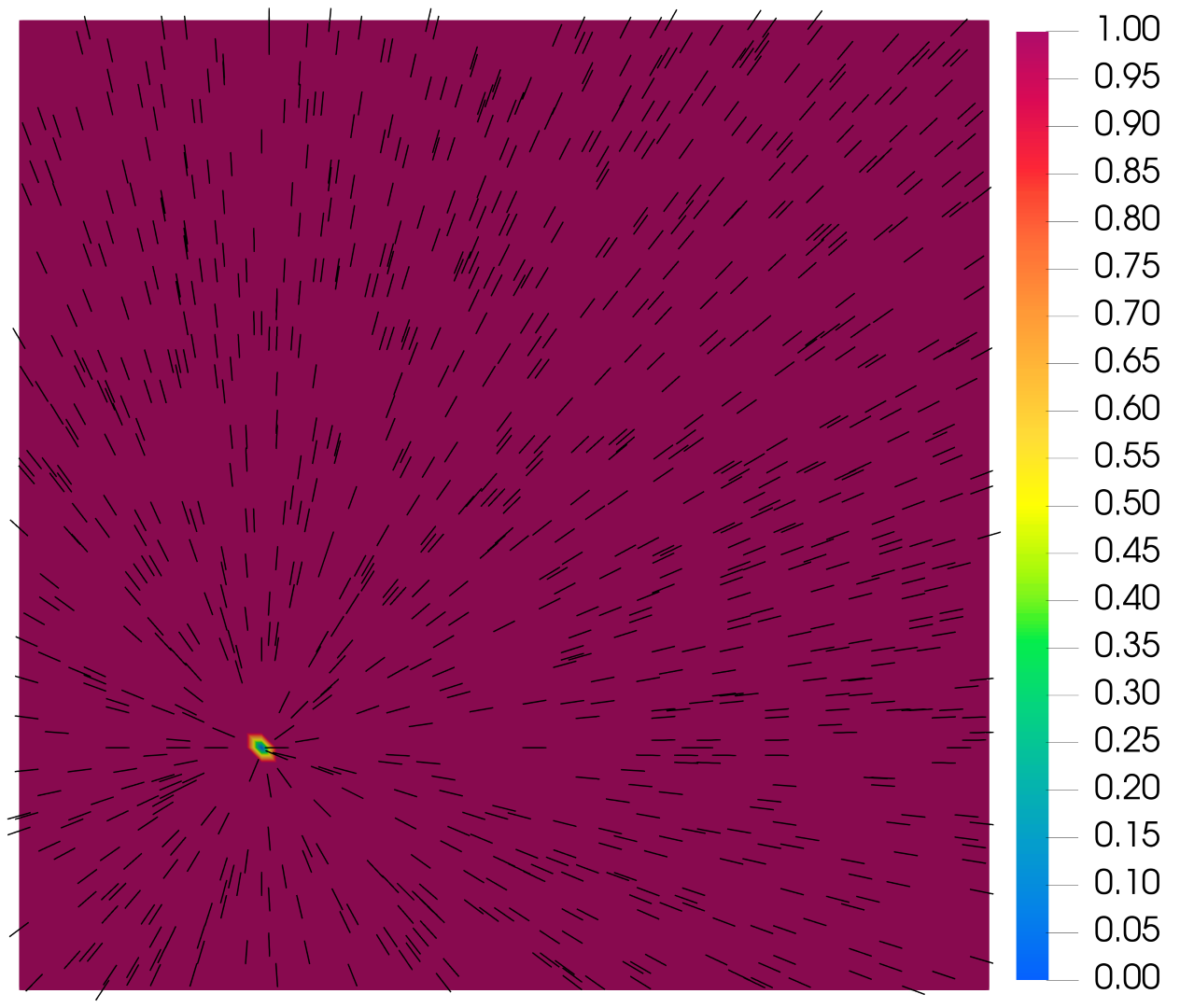}
        \caption{$t = 0$}
    \end{subfigure}
    \hfill
    \begin{subfigure}[b]{0.3\textwidth}
        \includegraphics[width=\textwidth]{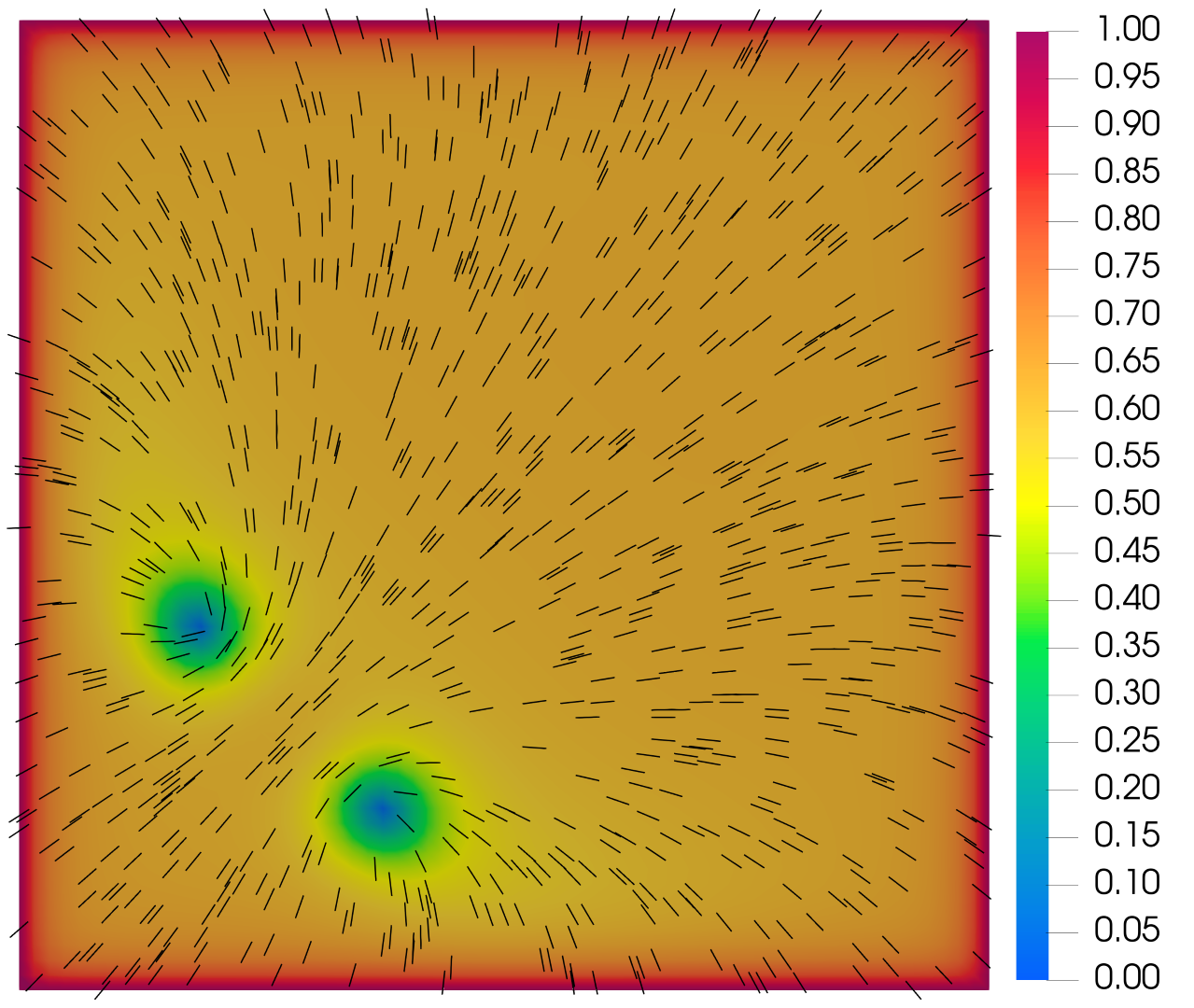}
        \caption{$t = 5.602$}
    \end{subfigure}
    \hfill
    \begin{subfigure}[b]{0.3\textwidth}
        \includegraphics[width=\textwidth]{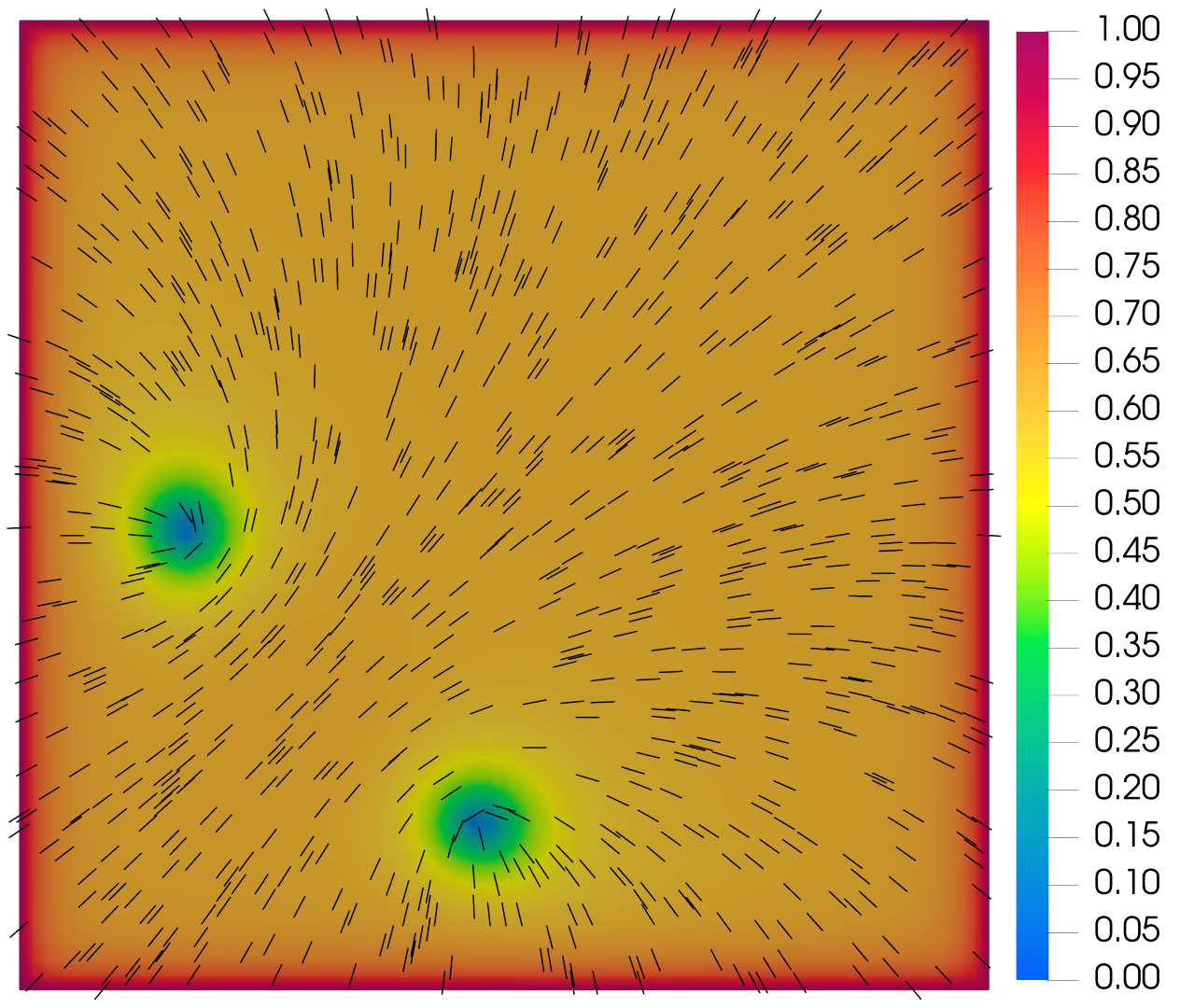}
        \caption{$t = 10.402$}
    \end{subfigure}
    
    \vspace{0.2cm}
    
    \begin{subfigure}[b]{0.3\textwidth}
        \includegraphics[width=\textwidth]{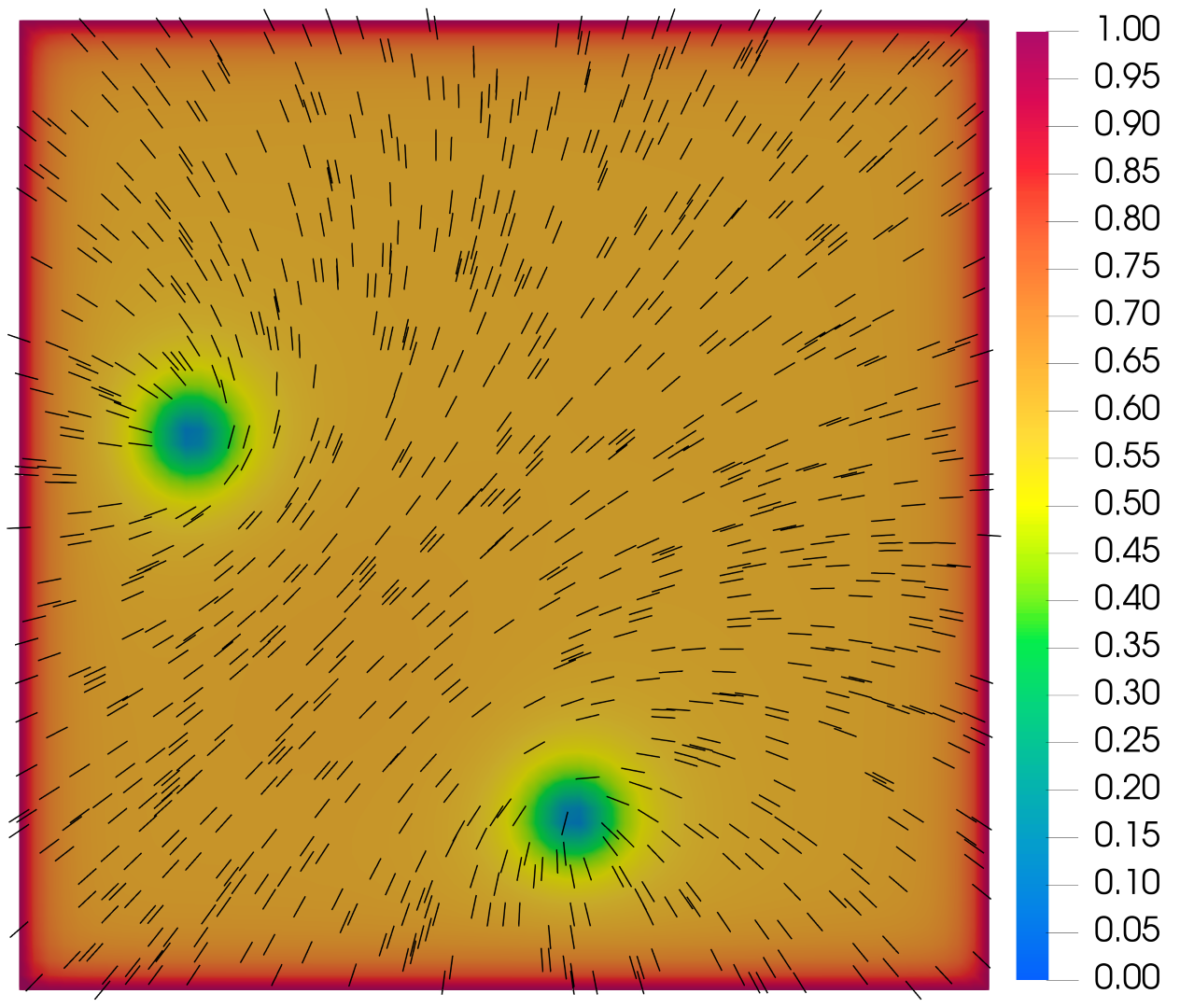}
        \caption{$t = 20.002$}
    \end{subfigure}
    \hfill
    \begin{subfigure}[b]{0.3\textwidth}
        \includegraphics[width=\textwidth]{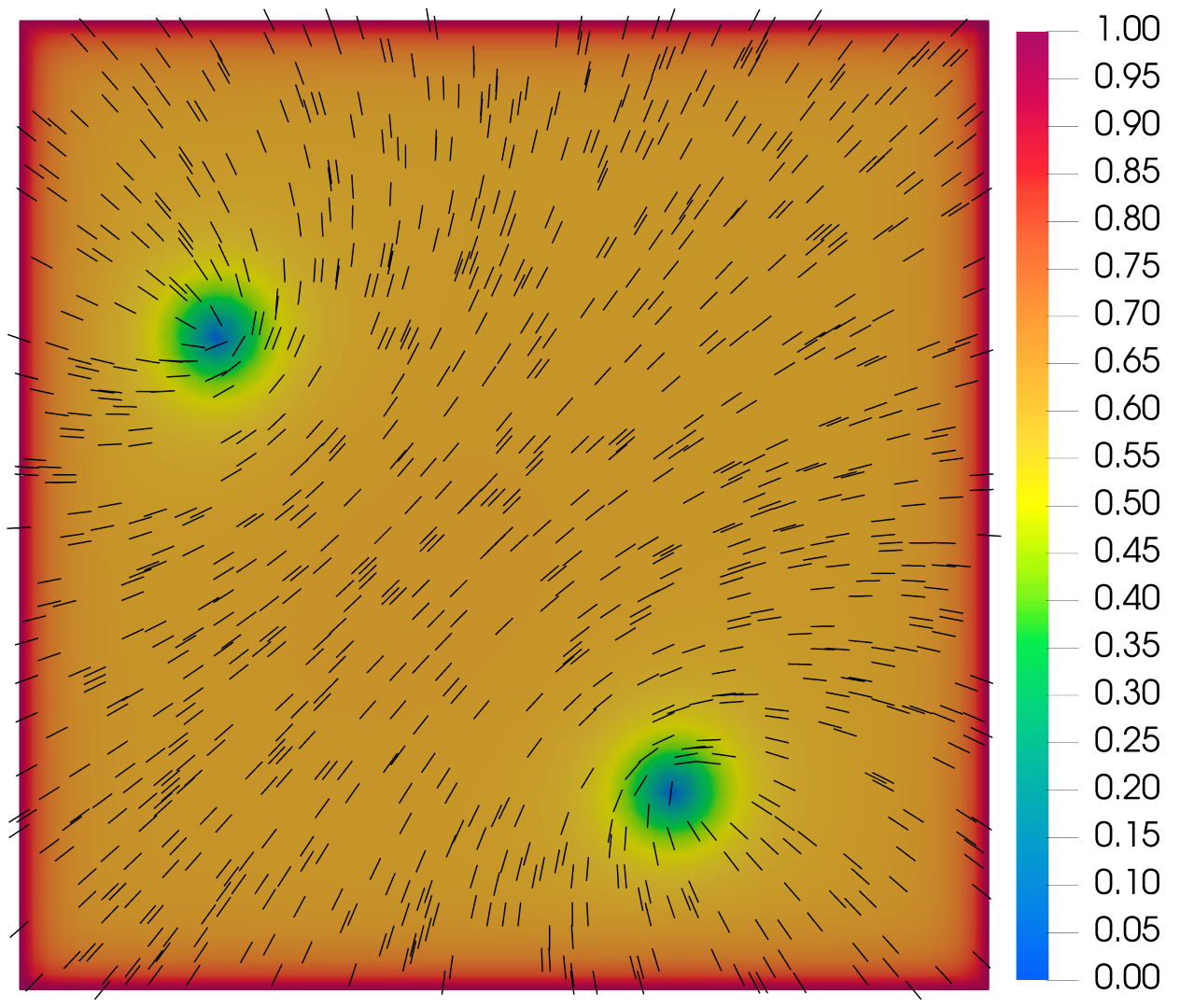}
        \caption{$t = 50.402$}
    \end{subfigure}
    \hfill
    \begin{subfigure}[b]{0.3\textwidth}
        \includegraphics[width=\textwidth]{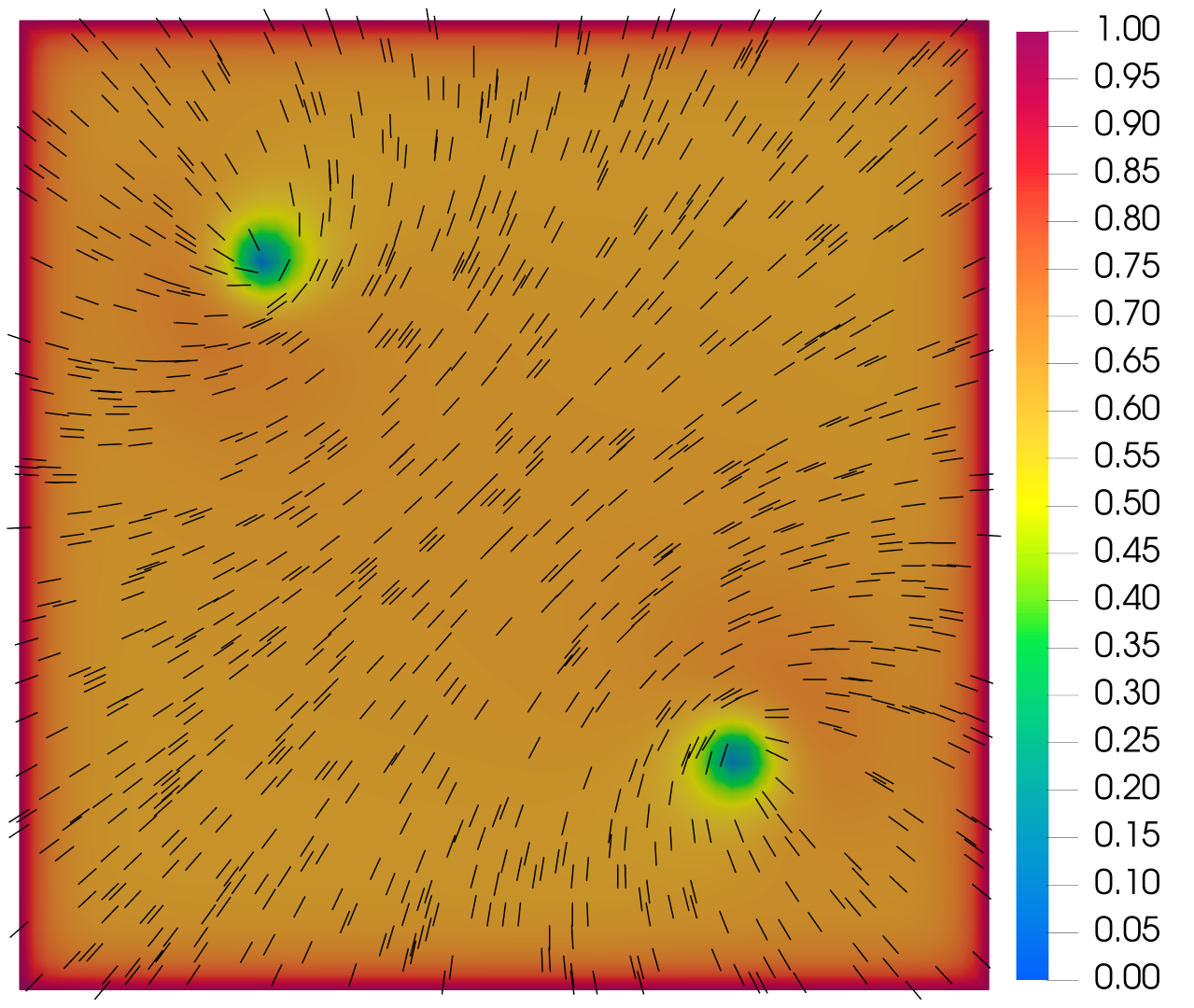}
        \caption{$t = 250$}
    \end{subfigure}
    
    \caption{\emph{Evolution of the major director field $m$ and order parameter $S$:}
    the initial +1 point defect is split into two +1/2 point defects \cite[Section 10]{Ball2017}.}.
    \label{fig:exp3-1}
\end{figure}

\begin{figure}
    \centering
    \begin{subfigure}[b]{0.4\textwidth}
        \includegraphics[width=\textwidth]{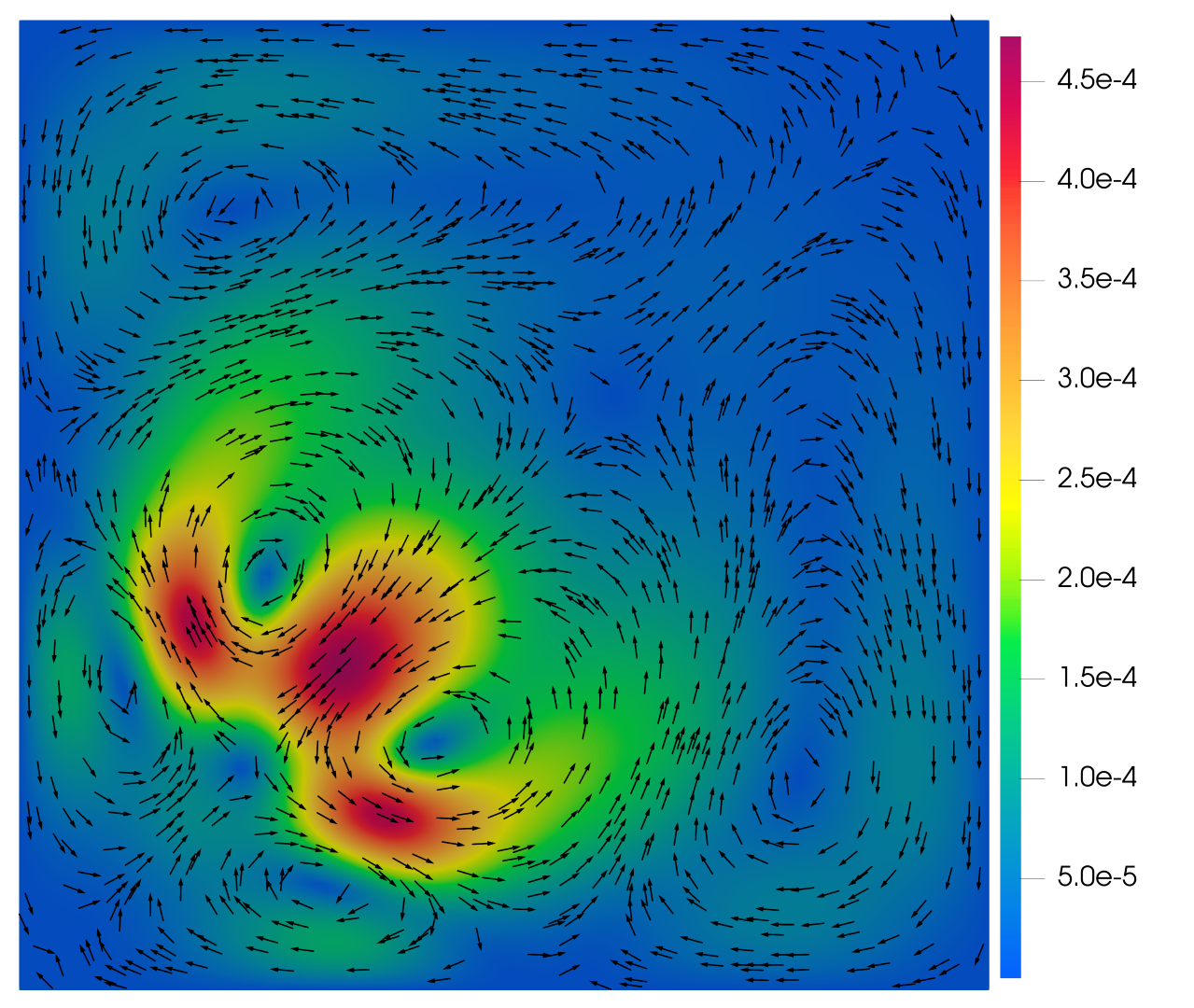}
        \caption{$t = 5.602$}
    \end{subfigure}
    \hfill
    \begin{subfigure}[b]{0.4\textwidth}
        \includegraphics[width=\textwidth]{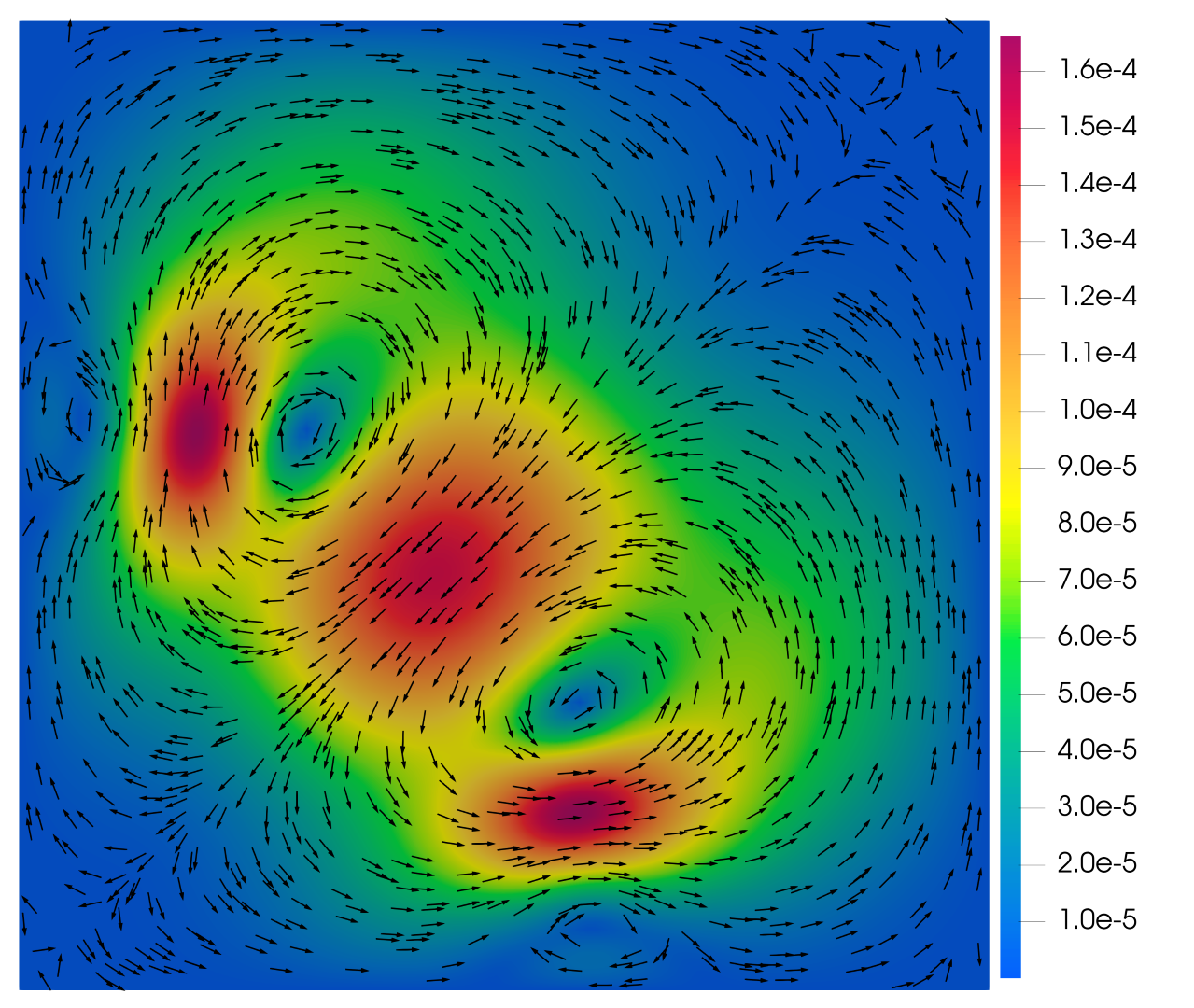}
        \caption{$t = 20.002$}
    \end{subfigure}
    
    \vspace{0.2cm}

        \begin{subfigure}[b]{0.4\textwidth}
        \includegraphics[width=\textwidth]{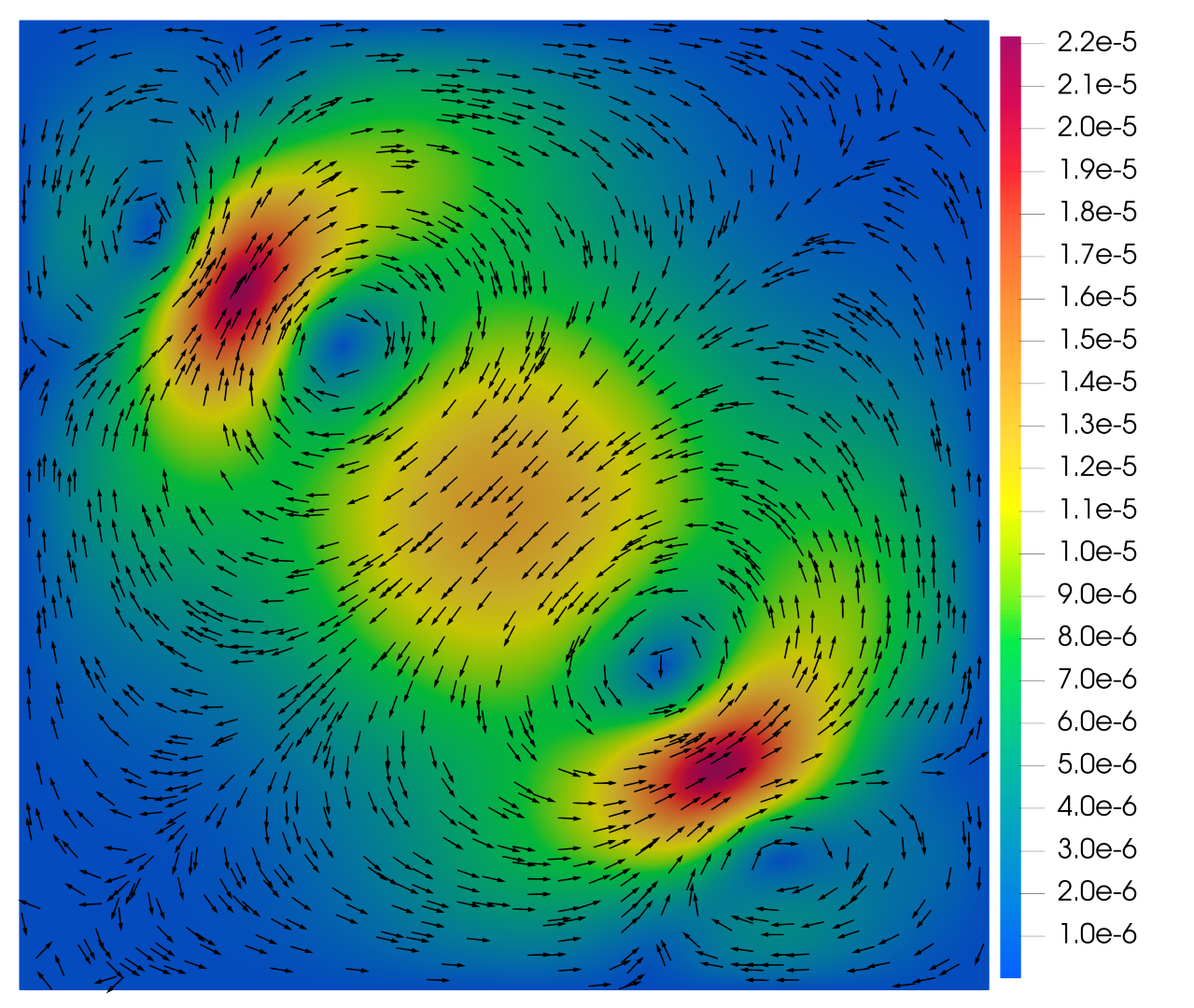}
        \caption{$t = 100$}
    \end{subfigure}
    \hfill
    \begin{subfigure}[b]{0.4\textwidth}
        \includegraphics[width=\textwidth]{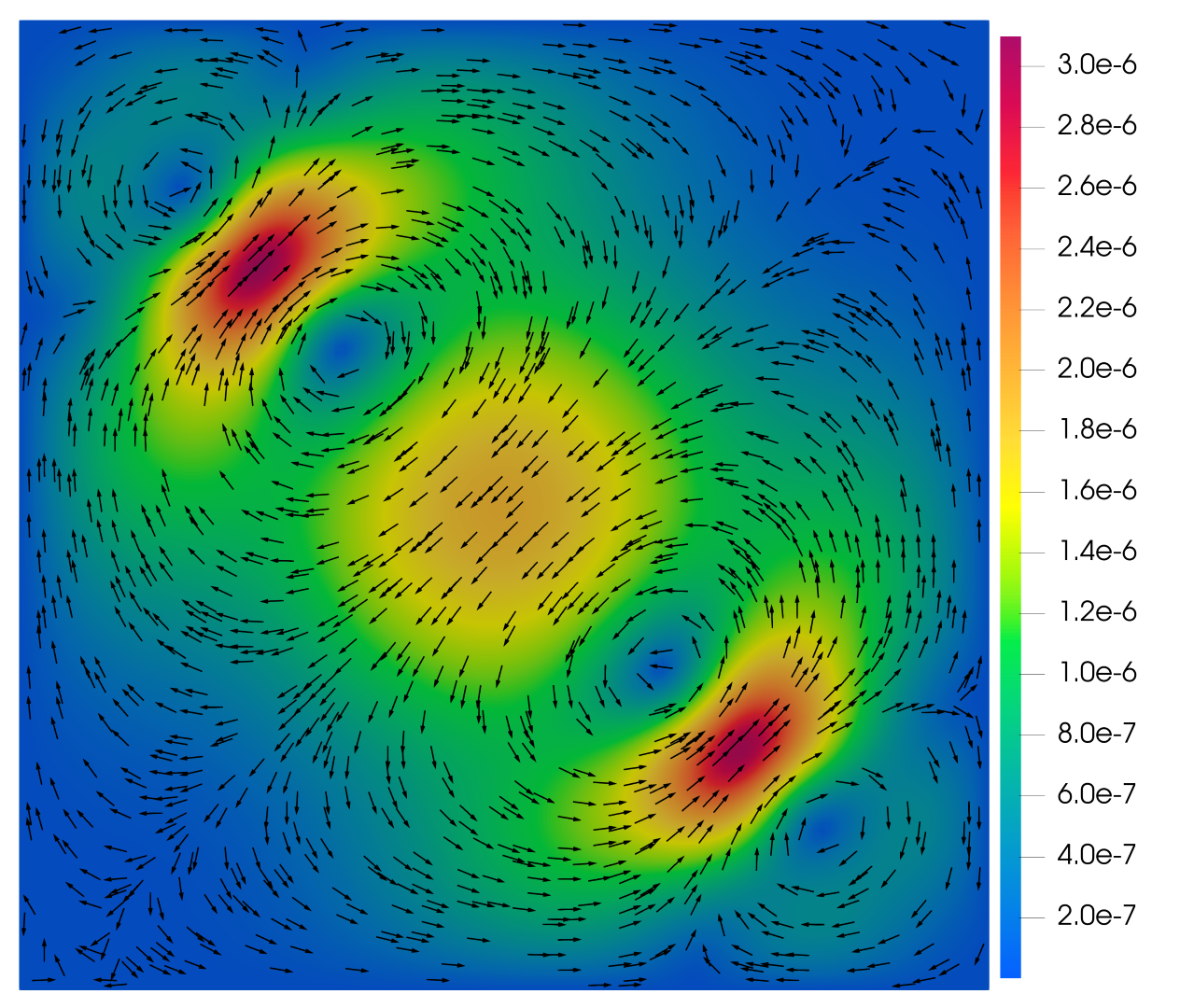}
        \caption{$t = 250$}
    \end{subfigure}
    
    \caption{\emph{Evolution of the velocity field:} notice that the magnitude is higher in the vicinity of the two +1/2 point defects.}
    \label{fig:exp3-2}
\end{figure}

\subsection{2.5D simulation: skyrmion dynamics}
A skyrmion, in the context of nematic liquid crystals, is a continuous configuration that cannot smoothly transition to a uniform state \cite{Ackerman2017,Bogdanov2003}.
Skyrmions' unique properties make them applicable as information carriers, and therefore highly attractive to be used in the development of data storage, sensing and optical communication \cite{Asilehan2025,Han2022,Fert2017,Shen2024}.
Inspired by the computational and laboratory experiments found in \cite{Coelho2023,Ackerman2017}, we investigate the hydrodynamical behavior of a ``baby skyrmion'' located in long channel that is subject to a constant pressure gradient.
More precisely, we consider $\Omega = (0,L_x) \times (0,L_y)$ with $L_x = L_y = 112$, triangulated with a uniform structured mesh of size $h = \sqrt{2}$, that is periodic on its left and right edges.
The computational time is $T = 900$, the time step is $\Delta t = 0.05$ and we set $A_0 = 500$.
As in \cite[Supplemental Material]{Coelho2023}, physical parameters are given by $a = 0.00256 \cdot 2$, $b = 0.208896 \cdot 3$, $c = 0.1152 \cdot 4$, $\mu = 0.909$, $\xi = 0.82$, $M = 0.474$, $L = 1.0226 \cdot 10^{-6}$ (the multiplying constants in $a$, $b$ and $c$ are due to the different scaling of $\mathcal{F}_B$ \eqref{eq:FBFE} in \cite{Coelho2023}).
Physical parameters in \cite{Coelho2023} that are not part of our (simpler) model are simply neglected.
Since out-of-plane orientations are essential for describing skyrmion dynamics, we allow the Q-tensor field $Q$ (and the molecular field $\HH$) to take values on $\R^{3\times3}$ while keeping the velocity field $u$ confined to $\R^2$ (this justifies the ``2.5D'' label);
the resulting system of partial differential equations can be seen as a simplification of the 3D Beris--Edwards system \eqref{seq:BerisEdwardsIEQ} under the assumptions that the velocity field is of the form $u = (u_1, u_2, 0)$, and that all variables are independent of $x_3$.
We stress that our analysis does not immediately apply to this ``hybrid'' setting.
We model the constant pressure gradient via a forcing term $f$ that induces a fluid flow from left to right given by the Poiseuille relation:
$f = (\beta,0)$, where $\beta = \frac{8\mu u_{\text{max}}}{L_y^2}$ and $u_{\text{max}} = 1.5$;
due to the coupling with $Q$, the maximum velocity may differ from $u_{\text{max}}$.
At the top and bottom of the domain, we consider the no-slip boundary condition $u=0$, and weakly anchor $Q$ to $Q_\text{wall} = e_3 \otimes e_3 - \frac{1}{3}I$ by considering the Rapini--Papoular (natural) boundary condition $L \nabla Q \cdot n = - W (Q-Q_{\text{wall}})$ on the top and bottom of the domain, where $W = 10^{-2}$.
At time $t=0$, we set $u = 0$ and $Q = Q_{\text{in}}$ and $r = r(Q_{\text{in}})$, where $Q_{\text{in}} = m_{\text{in}} \otimes m_{\text{in}} - \frac{1}{3}I$ is the out-of-plane uniaxial Q-tensor that describes the ``baby skyrmion'' \cite[page 3]{Coelho2023}\cite{Coelho2022}:
\begin{gather*}
m_{\text{in}} = (\sin(\tilde a)\sin(\omega \tilde b + g), \sin(\tilde a)\cos(\omega \tilde b + g), -\cos(\tilde a)),\\
\tilde a = \frac{\pi}{2}\left(1 - \tanh\left(\frac{B}{2}(\rho - R)\right)\right), \qquad \tilde b = \operatorname{atan2}\left(x-C_x,y-C_y\right),
\qquad \rho = \sqrt{(x-C_x)^2 + (y-C_y)^2},
\end{gather*}
where $\omega=1$, $g = \frac{\pi}{2}$, $R = 0.7 \cdot 16$, $B = 0.5$, $C_x = \frac{L_x}{2}$, $C_y = \frac{L_y}{2}$.
In order to fairly compare our results with the numerical experiments found in \cite[Figures 2, 3 and 5]{Coelho2023}, we report the evolution of the magnitude of the third component $m_3$ of the leading eigenvector $m$ of $Q$ in \autoref{fig:exp4-1}.
Albeit we do not provide snapshots of the velocity field, we find it important to mention that the constant pressure gradient was effective at inducing a parabolic Poiseuille-like profile.
Finally, we stress that unlike the experiments in \cite{Coelho2023} that are performed in fully periodic domains, we observe the formation of a small boundary layer in the $Q$ and $\HH$ tensor fields, that slowly propagates from the top and bottom of the domain to its interior, eventually inducing non-physical behavior.
Our computational time $T = 900$ is small enough to prevent such phenomena.

\begin{figure}
    \centering
    \begin{subfigure}[b]{0.3\textwidth}
        \includegraphics[width=\textwidth]{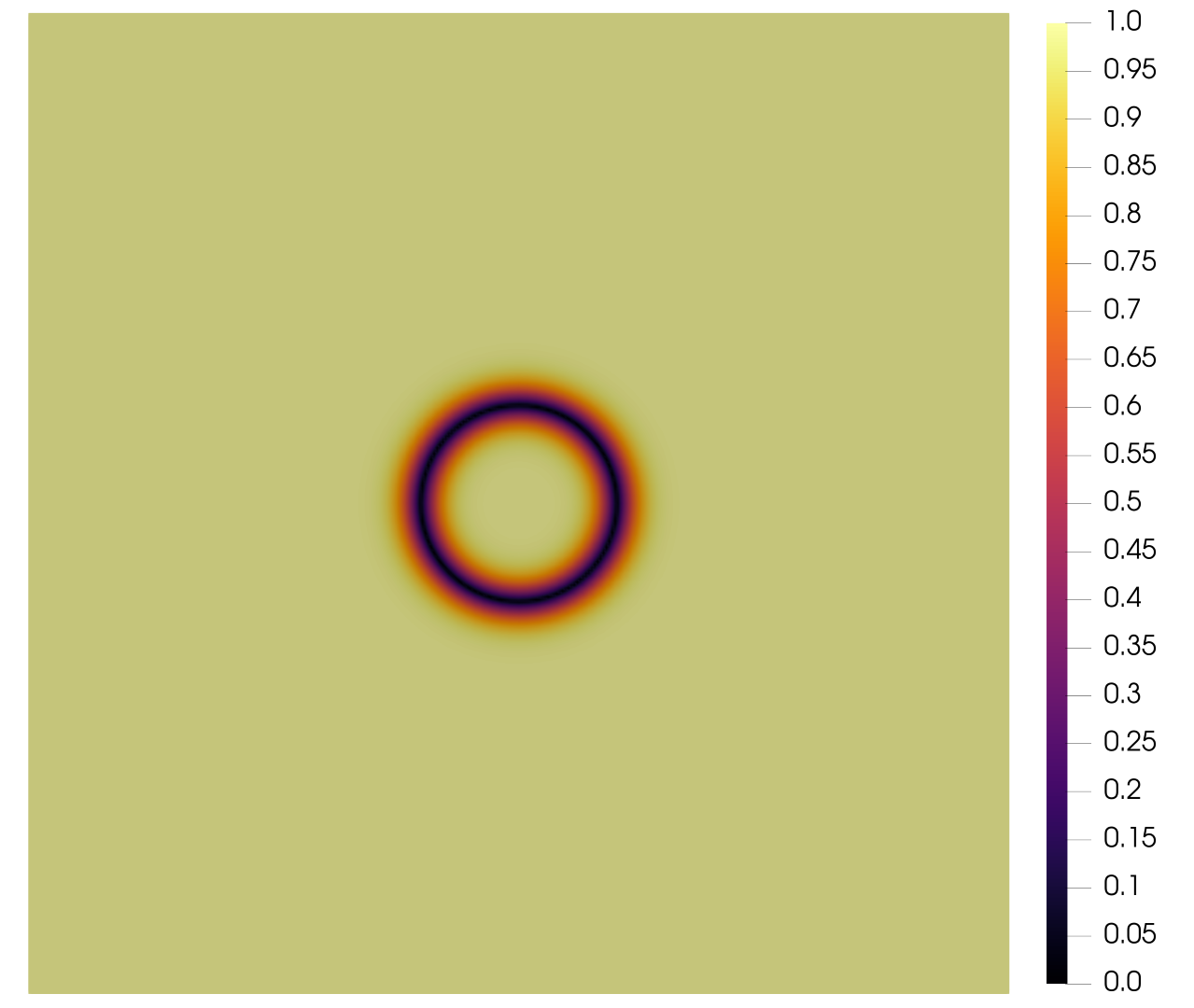}
        \caption{$t = 0$}
        \label{fig:exp2_nz_0}
    \end{subfigure}
    \hfill
    \begin{subfigure}[b]{0.3\textwidth}
		\includegraphics[width=\textwidth]{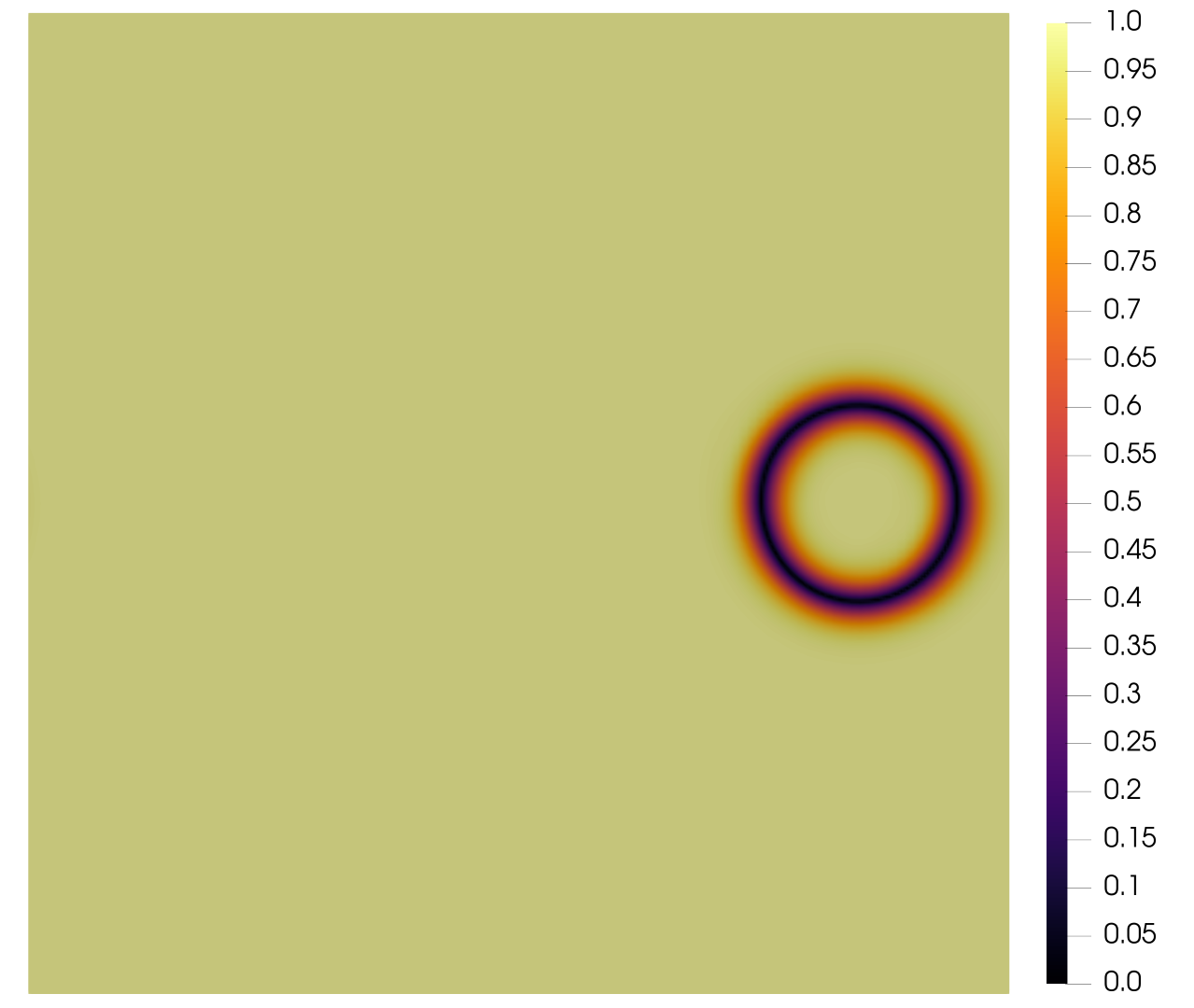}
        \caption{$t = 300$}
        \label{fig:exp2_nz_300}
    \end{subfigure}
   	\hfill
    \begin{subfigure}[b]{0.3\textwidth}
        \includegraphics[width=\textwidth]{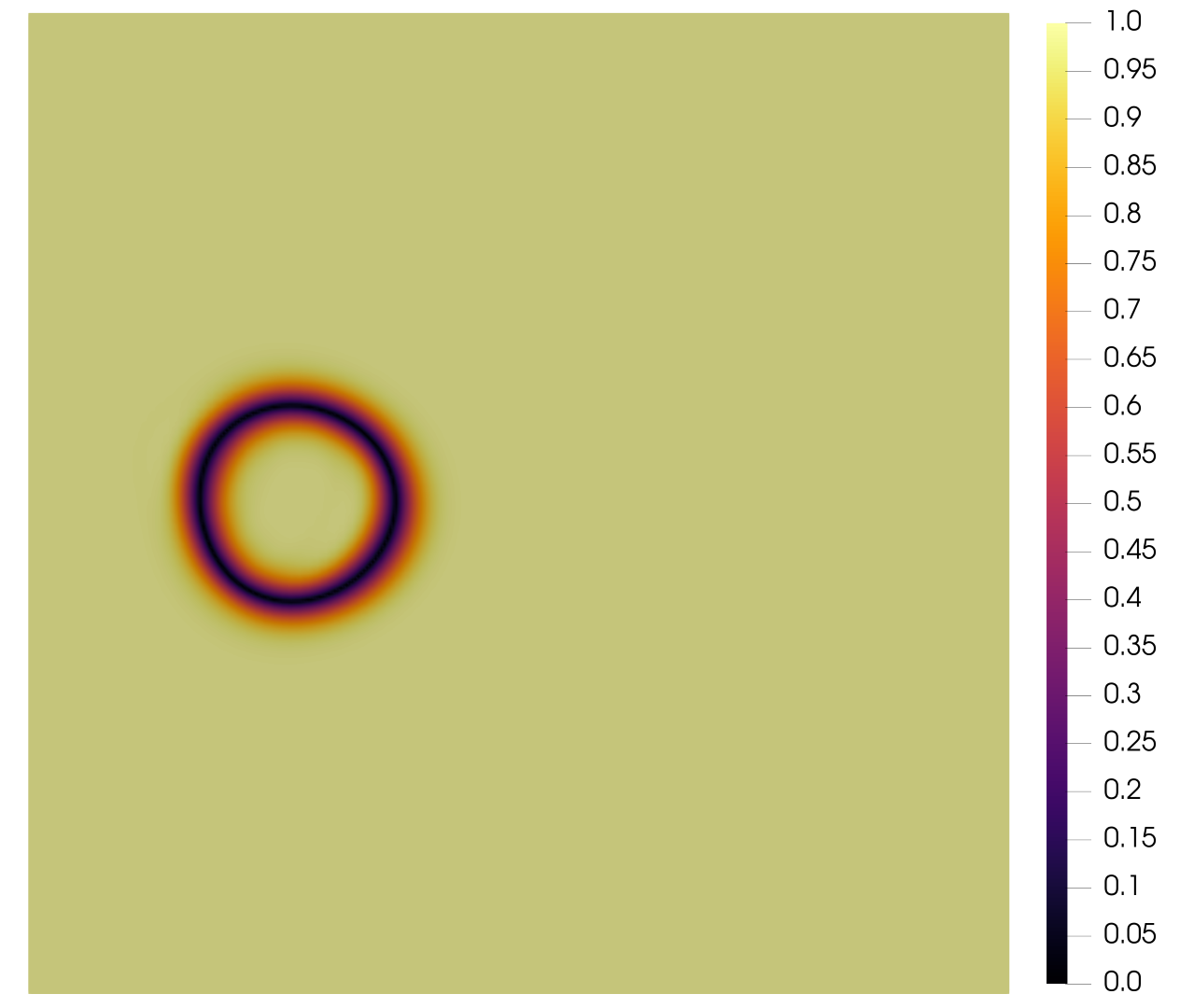}
        \caption{$t = 450$}
        \label{fig:exp2_nz_450}
    \end{subfigure}
    
    \vspace{0.2cm}

    \begin{subfigure}[b]{0.3\textwidth}
        \includegraphics[width=\textwidth]{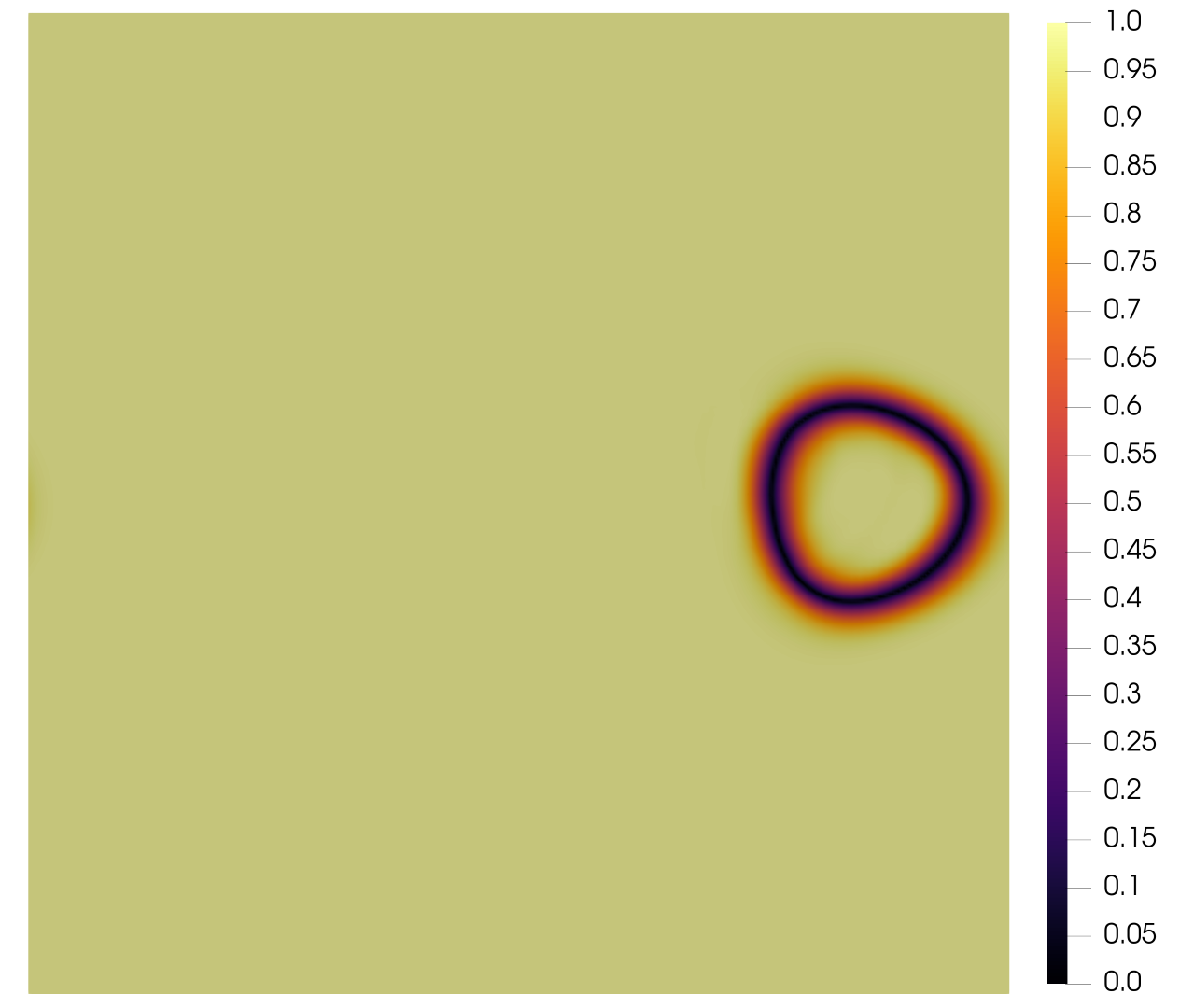}
        \caption{$t = 600$}
        \label{fig:exp2_nz_600}
    \end{subfigure}
    \hfill
    \begin{subfigure}[b]{0.3\textwidth}
        \includegraphics[width=\textwidth]{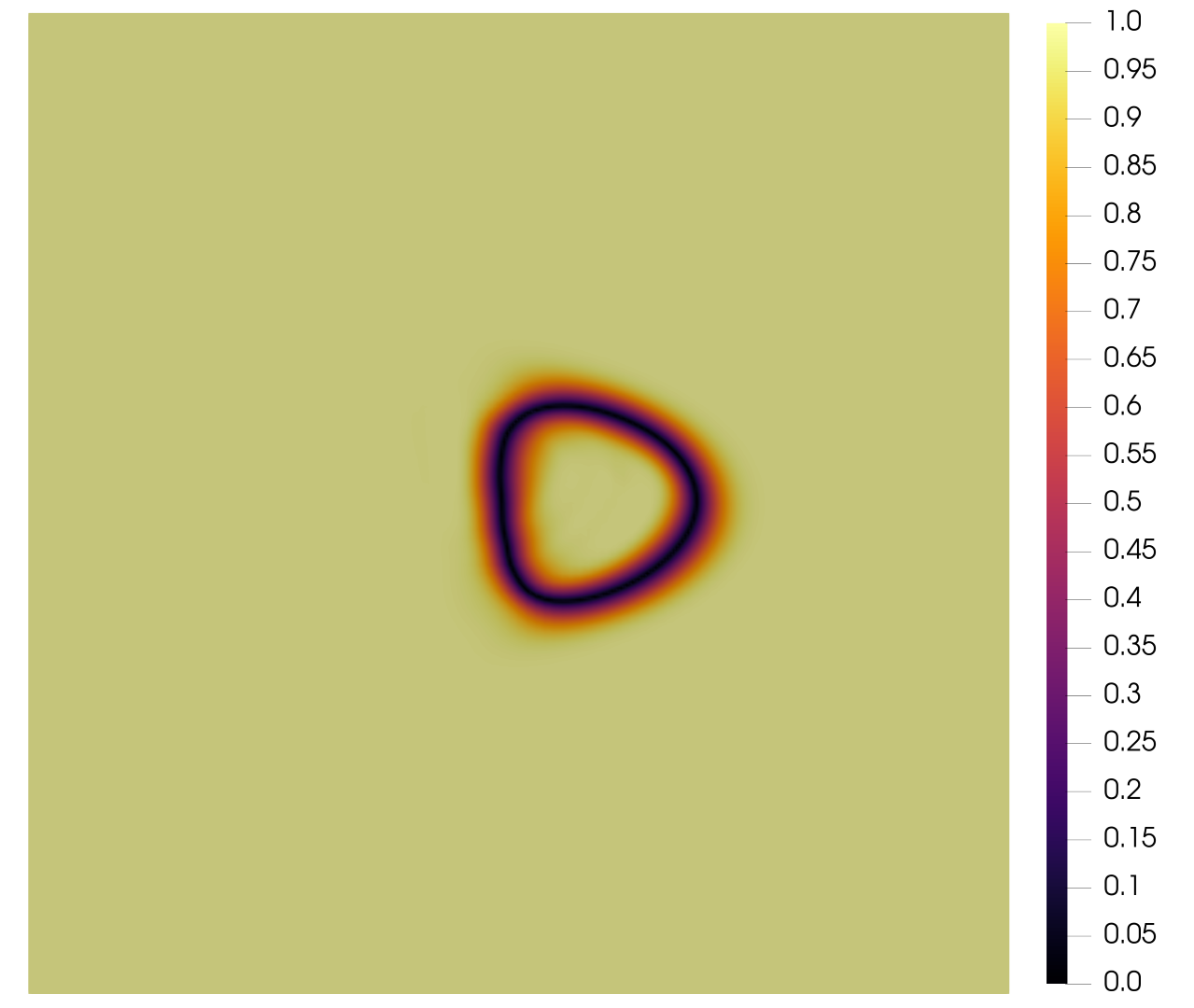}
        \caption{$t = 750$}
        \label{fig:exp2_nz_750}
    \end{subfigure}
    \hfill
    \begin{subfigure}[b]{0.3\textwidth}
        \includegraphics[width=\textwidth]{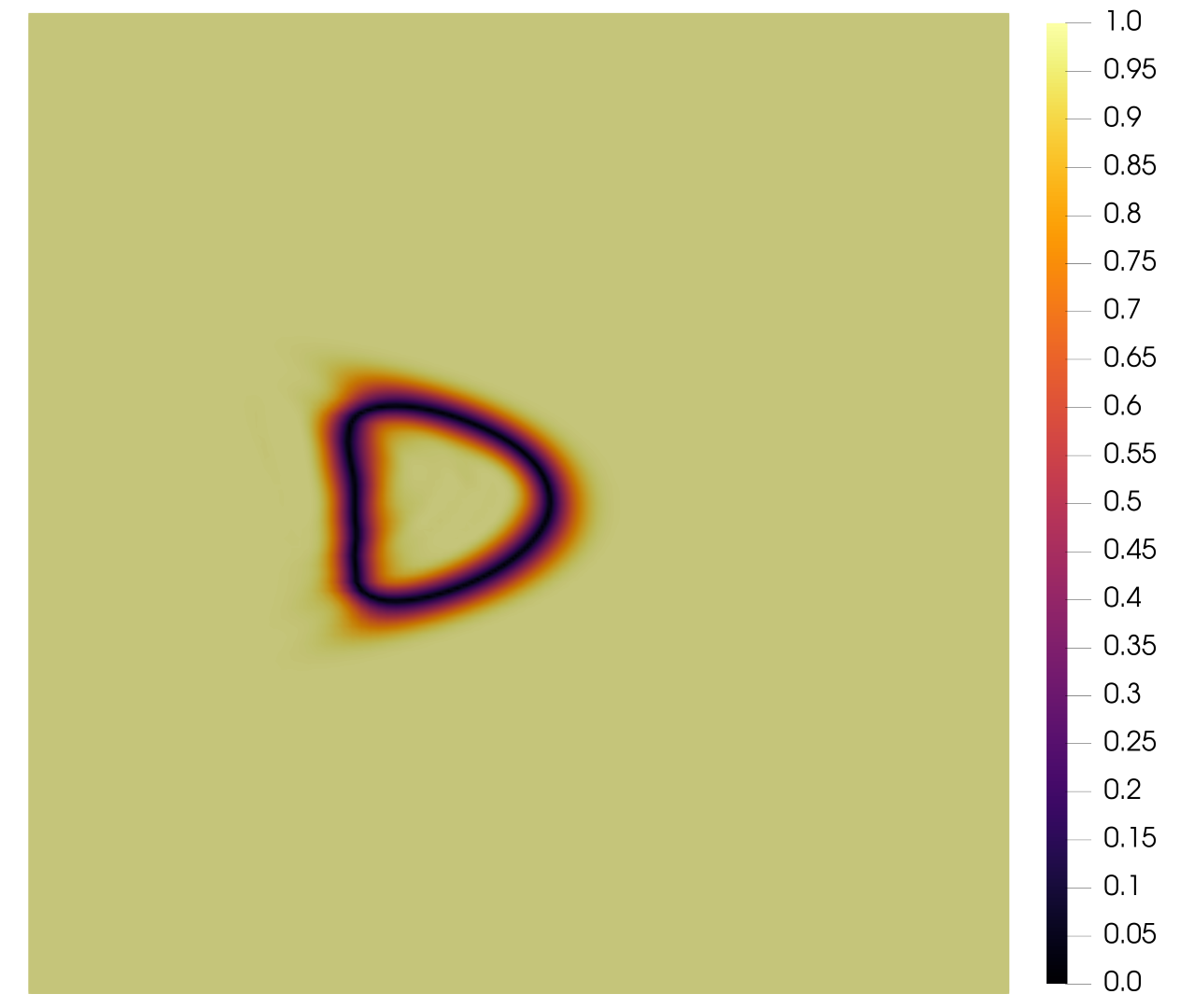}
        \caption{$t = 900$}
        \label{fig:exp2_nz_900}
    \end{subfigure}
    
    \caption{\emph{Evolution of the magnitude of the third component $n_3$ of the leading eigenvector $n$ of $Q$:}
    The skyrmion is transported from left to right and deformed by the action of the fluid flow.
    Our results are in qualitative agreement with \cite{Coelho2023}.
    }
    \label{fig:exp4-1}
\end{figure}

\section*{Acknowledgments}
F.W.~thanks Lucas Bouck for helpful discussions on the subject of this article. 
Claude (Anthropic) was used to assist with literature search and to check proofs of an earlier draft, while Gemini (Google) was used for support in implementing our numerical experiments.
ChatGPT (OpenAI) was used to review the final version and correct minor mistakes and inconsistencies.
All mathematical content was created and verified by the authors.

\appendix
\section{Discrete Gr\"onwall inequality}\label{app:gronwall}

				The following lemma from~\cite[Proposition 4.1]{Emmrich1999} is useful for proving an energy bound in the case that the external source in the fluid equation is nonzero.
\begin{lemma}[\cite{Emmrich1999}]
	\label{lem:discretegronwall}
	Let $\{a_n\}_{n\in \N}\geq 0$, and $\{b_n\}_{n\in \N}\geq 0$ be two nonnegative sequences satisfying
	\begin{equation}\label{eq:assgronwall}
		a_{n+1} \leq b_{n+1} + \nu\Delta t\sum_{j=1}^{n+1} a_j,\quad n=0,1,2,\dots
	\end{equation}
	for parameters $\nu,\Delta t\geq 0$ and assume that $1-\nu\Delta t>0$.
	Then $a_{n+1}$ satisfies
	\begin{equation}\label{eq:discgronwall}
		a_{n+1}\leq b_{n+1} +\frac{\nu\Delta t}{1-\nu\Delta t}\sum_{j=0}^n\left(\frac{1}{1-\nu\Delta t}\right)^{n-j} b_{j+1}.
	\end{equation}
	Moreover, if $\{b_n\}_{n\in\N}$ is monotonically increasing, it follows
	\begin{equation*}
		a_n \leq b_n \left(\frac{1}{1-\nu\Delta t}\right)^n.
	\end{equation*}
\end{lemma}
\begin{proof}
	For the proof, see \cite[Proposition 4.1]{Emmrich1999}.
\end{proof}
\section{Properties of the mass-lumped inner product}\label{app:masslumped}
Here, we prove Lemma~\ref{lem:masslumpedLp}. We restate it for convenience:

\lplumping* 
\begin{proof}
	We first prove the first inequality in~\eqref{eq:hhoelder}. Assume that $1<p,p'<\infty$. We  use H\"older's inequality and plug in the definition
	\begin{align*}
		|(f,g)_h | &= \left|\int_{\dom}\sum_{z\in \mathcal{N}_h} f(z)\varphi_z^{1/p}  g(z) \varphi_z^{1/p'} dx \right|\\
		& \leq \left(\int_{\dom}\sum_{z\in \mathcal{N}_h} |f(z)|^p\varphi_z dx \right)^{1/p}\left(\int_{\dom}\sum_{z\in \mathcal{N}_h} |g(z)|^{p'}\varphi_z dx \right)^{1/p'}\\
		& = \norm{f}_{h,p}\norm{g}_{h,p'}.
	\end{align*}
	The case $p=\infty$ follows by taking the maximum over all $z\in \mathcal{N}_h$ in the first identity.
	Now the second inequality will follow if we can prove~\eqref{eq:masslumpedstability}. If $p=\infty$, then $\norm{f}_{h,\infty}=\norm{f}_{L^\infty}$, since the maximum of a function in $\mathcal{S}^1(\mathcal{T}_h)$ is attained at a node. So we assume without loss of generality that $p<\infty$. 
	We first note that since $f\in \mathcal{S}^1(\mathcal{T}_h)$, we have for any $x\in \dom$
	\begin{equation*}
			|f(x)| = \left| \sum_{z\in \mathcal{N}_h} f(z)\varphi_z\right|,
	\end{equation*}
	thus with Jensen's inequality, and using that $0\leq \varphi_z\leq 1$, and therefore
	 $\varphi_z(x)^p\leq \varphi_z(x)$
	\begin{equation*}
		|f(x)|^p = \left| \sum_{z\in \mathcal{N}_h} f(z)\varphi_z\right|^p\leq \sum_{z\in \mathcal{N}_h} |f(z)|^p|\varphi_z| =\sum_{z\in \mathcal{N}_h} |f(z)|^p\varphi_z .
	\end{equation*}
	Integrating over the domain and using that $\varphi_z$ are piecewise linear, we get
	\begin{equation*}
		\norm{f}_{L^p}^p\leq \sum_{z\in\mathcal{N}_h} |f(z)|^p \int_{\dom} \varphi_z dx = \norm{f}_{h,p}^p.
 	\end{equation*}
	This proves the lower bound. To prove the upper bound, we consider a simplex $K$ and  first note that the value of a linear function $f$ at its barycenter $b$ is equal to its average of the values at the nodes:
	\begin{equation*}
		f(b)  = \frac{1}{d+1} \sum_{z\in \mathcal{N}_h\cap K} f(z).
	\end{equation*}
	Here
	\begin{equation*}
		b = \frac{z_0+z_1+\dots + z_d}{d+1}
	\end{equation*}
	where $z_0,z_1,\dots, z_d$ are the nodes that are adjacent to $K$. In particular, the basis functions satisfy
	\begin{equation*}
		\varphi_y(b) = \frac{1}{d+1}\sum_{z\in \mathcal{N}_h\cap K}\varphi_y(z) = \begin{cases}
			0,&\quad \text{if }\, y\notin K,\\
			\frac{1}{d+1},&\quad \text{if }\, y\in K. 
		\end{cases}
	\end{equation*}
	So if we write $f(x) = \sum_{z\in \mathcal{N}_h} f(z)\varphi_z(x)$, and compute
	\begin{equation}\label{eq:mean}
		\frac{1}{|K|}\int_K f(x) dx = \sum_{z\in \mathcal{N}_h} f(z)\frac{1}{|K|}\int_K\varphi_z(x) dx = \frac{1}{d+1}\sum_{i=0}^d f(z_i) = f(b),
	\end{equation}
	thus the linear function $f$ at the barycenter of $K$ is equal to the average over $K$. Now fix a node $z_0$ and consider the simplex $S$ resulting by replacing the node $z_0$ by the barycenter $b$. Then  $S$ has the barycenter $c$,
	\begin{equation*}
		c = \frac{1}{d+1}(z_1+z_2+\dots + b) = \frac{(d+2)b-z_0}{d+1}.
	\end{equation*}
	In other words $z_0 = (d+2)b - (d+1)c$. Since $f$ is linear on $K$, we therefore have using~\eqref{eq:mean} twice, once for $K$ and $b$ and once for $S$ and $c$,
	\begin{equation*}
		f(z_0) = (d+2)f(b)-(d+1)f(c) = \frac{d+2}{|K|}\int_K f(x)dx - \frac{d+1}{|S|}\int_Sf(x) dx.
	\end{equation*}
	Since $\varphi_{z_0}(z_0)=1$ and $\varphi_{z_0}(b)=1/(d+1)$, it follows that $|S| = |K|/(d+1)$ (the simplexes have the same base and the heights are rescaled by a factor $1/(d+1)$), therefore, we can estimate $|f(z_0)|$ by
	\begin{equation*}
		|f(z_0)| \leq  \left[(d+2)+(d+1)^2\right]\fint_K \left|f(x)\right|dx .
	\end{equation*}
	We take the $p$-th power, multiply with $\varphi_{z_0}$ (which is nonnegative) and sum over all nodes ($z_0$ and $K$ were arbitrary), to obtain using Jensen's inequality,
	\begin{align*}
		\sum_{z\in \mathcal{N}_h} |f(z)|^p \varphi_z &\leq \left[(d+2)+(d+1)^2\right]^p\sum_{z\in \mathcal{N}_h}\left|\fint_{K_z} |f(x)|dx\right|^p \varphi_z\\
		&\leq \left[(d+2)+(d+1)^2\right]^p\sum_{z\in \mathcal{N}_h}\fint_{K_z} |f(x)|^pdx \,\varphi_z.
	\end{align*}
	(here we used $K_z$ to indicate that the node $z$ is adjacent to $K$).
	Let $\omega_z$ denote the ``patch' of the node $z$.
	We integrate over the domain, and note that the left-hand side is equal to $\norm{f}_{h,p}^p$,
	\begin{align*}
		\norm{f}_{h,p}^p& = \sum_{z\in \mathcal{N}_h} |f(z)|^p \int_{\dom}\varphi_z dx \\
		&  \leq \left[(d+2)+(d+1)^2\right]^p \sum_{z\in \mathcal{N}_h}\fint_{K_z} |f(x)|^pdx  \int_{\dom}\varphi_z dx \\
		& = \frac{((d+2)+(d+1)^2)^p}{d+1}\sum_{z\in \mathcal{N}_h} \frac{|\omega_z|}{|K_z|}\int_{K_z} |f(x)|^pdx  \\
		& \leq C_{d,p}\norm{f}_{L^p}^p,
	\end{align*}
	where we have used that $\frac{|\omega_z|}{|K_z|} \leq C$ due to quasi-uniformity and shape-regularity, and that $\sum_{z \in \mathcal{N}_h} \int_{K_z} |f|^p \leq (d+1) \|f\|_{L^p}^p$ because each element can be selected at most by its $(d+1)$ vertices.
	This proves the second inequality in~\eqref{eq:masslumpedstability}.
The proof of~\eqref{eq:masslumpingLp} follows~\cite[Lemma 3.9]{Bartels2015book} closely.	We have 
	\begin{align*}
		|(f,g)_h-(f,g)| & = \left|\int_{\dom} \left(\Ih(f\cdot g) - f\cdot g\right) dx \right|\\
		& \leq \int_{\dom} \left|\Ih(f\cdot g) - f\cdot g\right| dx\\
		& = \sum_{K\in \mathcal{T}_h} \int_{K} \left|\Ih (f\cdot g) - f\cdot g\right| dx \\
		& \leq C h^2 \sum_{K\in \mathcal{T}_h} \int_K \left|\Grad^2 (f\cdot g)\right| dx \\
		& \leq C h^2  \sum_{K\in \mathcal{T}_h} \int_K |\Grad f| |\Grad g| dx,
	\end{align*}
	where we used that $f,g\in \mathcal{S}^1(\mathcal{T}_h)$ and therefore their second derivatives are zero on each $K$. Now, we use H\"older's inequality,
	\begin{equation*}
			|(f,g)_h-(f,g)| \leq C h^2 \int_{\dom} |\Grad f||\Grad g| dx \leq C h^2 \norm{\Grad f}_{L^p}\norm{\Grad g}_{L^{p'}}.
	\end{equation*}
	This proves the estimate for $\ell=1$. To prove the $\ell=0$ estimate, we use the inverse estimate.
	
\end{proof}
Next, we prove Lemma~\ref{lem:masslump2}:
\lplumperror*
\begin{proof}
	We write
	\begin{equation*}
		\norm{v_h F(w_h)-\Ih(v_hF(w_h))}_{L^1(\dom)} \leq \sum_{K\in \mathcal{T}_h} \int_K \sum_{z\in \mathcal{N}_h}\varphi_z(x)\left|v_h(z)F(w_h(x))-v_h(z)F(w_h(z))\right|dx.
	\end{equation*}
	Let $1<p<\infty$. Then we use H\"older's inequality and then the Lipschitz continuity of $F$ and that $w_h$ is affine on each $K$ to estimate
	\begin{align*}
		&\norm{v_h F(w_h)-\Ih(v_hF(w_h))}_{L^1(\dom)}\\
		&\leq \left(\sum_K\int_K \sum_{z\in \mathcal{N}_h}\varphi_z(x)|v_h(z)|^p dx \right)^{1/p} \left(\sum_K\int_K \sum_{z\in \mathcal{N}_h}\varphi_z(x)|F(w_h(x))-F(w_h(z))|^{p'} dx \right)^{1/p'} \\
		& \leq \norm{v_h}_{h,p}Ch \left(\sum_K\int_K \sum_{z\in \mathcal{N}_h}\varphi_z(x)|\Grad w_h(x)|^{p'} dx \right)^{1/p'} \\
		& = Ch \norm{v_h}_{h,p}\norm{\Grad w_h}_{L^{p'}}.
	\end{align*}
	If $p\in \{1,\infty\}$, we replace the integral in the corresponding term by a supremum over all nodes when applying H\"older's inequality. Now using~\eqref{eq:masslumpedstability}, the result follows.
	
\end{proof}
\section{A technical consequence following from the Rellich theorem}
\label{app:rellich}
\rellich*
\begin{proof}
	Let $\{e_k\}_{k=1}^\infty$ be an $L^2$-orthonormal basis of eigenfunctions of the Dirichlet Laplacian on $\dom$, i.e., $e_k$ solve
	\begin{align*}
		-\Delta e_k & = \lambda_k e_k,\quad \text{in }\,\dom,\\
		e_k & = 0,\quad \text{on }\, \partial \dom,
	\end{align*}
	with $(e_k,e_j)=\delta_{kj}$ and $0<\lambda_1\leq \lambda_2\leq \dots \leq \lambda_k\leq\dots$, where $\lambda_k\to \infty$. $e_k$ can be chosen such that $e_k\in H^1_0(\dom)\cap C^\infty(\dom)$~\cite[Theorem 9.31]{Brezis2011}. 
	So we can write $v = \sum_{k=1}^\infty v_k e_k$, where $v_k = (v,e_k)$. Now consider a function $w\in H^1_0(\dom)$. Then 
	\begin{equation*}
		\norm{w}_{H^{1}}^2 = \norm{w}_{L^2}^2 + \norm{\Grad w}_{L^2}^2 = \sum_{k=1}^\infty (1+\lambda_k)(w,e_k)^2.
	\end{equation*}
	Since $v\in L^2(\dom)$, we have $\sum_{k=1}^\infty (v,e_k)^2 <\infty$ and we can estimate
	\begin{align*}
		\norm{v}_{H^{-1}}& = \sup_{0\neq w\in H^1_0(\dom)}\frac{|(v,w)|}{\norm{w}_{H^1}}\\
		&  = \sup_{0\neq w\in H^1_0(\dom)}\frac{\left|\sum_{k=1}^\infty (v,e_k)(w,e_k)\right|}{\sqrt{\sum_{k=1}^\infty (1+\lambda_k)(w,e_k)^2}}\\
		&= \sup_{0\neq w\in H^1_0(\dom)}\frac{\left|\sum_{k=1}^\infty \frac{(v,e_k)}{\sqrt{1+\lambda_k}}\sqrt{1+\lambda_k}(w,e_k)\right|}{\sqrt{\sum_{k=1}^\infty (1+\lambda_k)(w,e_k)^2}}\\
		& \leq \sup_{0\neq w\in H^1_0(\dom)}\frac{\left(\sum_{k=1}^\infty \frac{(v,e_k)^2}{{1+\lambda_k}}\right)^{1/2}\left(\sum_{k=1}^\infty (1+\lambda_k)(w,e_k)^2\right)^{1/2}}{\sqrt{\sum_{k=1}^\infty (1+\lambda_k)(w,e_k)^2}}\\
		& = \sqrt{\sum_{k=1}^\infty \frac{(v,e_k)^2}{1+\lambda_k}}<\infty.
	\end{align*}
	This is an upper bound on the $H^{-1}$-norm of $v$.
	For the lower bound, suppose that $v \neq 0$ (otherwise the identity is immediate) and choose $w_0:= \sum_{k=1}^\infty \frac{(v,e_k)}{1+\lambda_k}e_k$.
	Then,
	\begin{equation*}
	\sqrt{\sum_{k=1}^\infty \frac{(v,e_k)^2}{1+\lambda_k}} = 	\frac{|(v,w_0)|}{\norm{w_0}_{H^1(\dom)}} \leq \norm{v}_{H^{-1}}.
	\end{equation*}
	Thus
	\begin{equation*}
		\norm{v}_{H^{-1}}\leq 	\sqrt{\sum_{k=1}^\infty \frac{(v,e_k)^2}{1+\lambda_k}} \leq \norm{v}_{H^{-1}}.
	\end{equation*}
	Now pick $N\in \N$ large enough such that $(1+\lambda_{N})^{-1}<\eta^2$. Then
	\begin{align*}
		\norm{v}_{H^{-1}}^2 &=\sum_{k=1}^N \frac{(v,e_k)^2}{1+\lambda_k} + \sum_{k=N+1}^\infty  \frac{(v,e_k)^2}{1+\lambda_k}\\
		& \leq \sum_{k=1}^N(v,e_k)^2 + \frac{1}{1+\lambda_{N+1}}\sum_{k=N+1}^\infty (v,e_k)^2\\
			& \leq \sum_{k=1}^N(v,e_k)^2  + \eta^2 \norm{v}^2_{L^2}.
	\end{align*}
	Choosing $\phi_k= e_k$, this proves the result.
\end{proof}

\bibliographystyle{abbrv}
\bibliography{BerisEdwardsfdbib}

\end{document}